\documentclass[a4paper,reqno,final]{amsart}

\usepackage{fullpage,enumitem}
\usepackage{amsmath,amssymb,mathrsfs}
\usepackage[only llbracket,rrbracket,longmapsfrom]{stmaryrd}
\usepackage{slashed} 
\usepackage{braket} 
\usepackage{mathtools} 
\usepackage{showkeys,version}
\usepackage{hyperref}
\hypersetup{bookmarks=false,draft=false,breaklinks,colorlinks}

\numberwithin{equation}{subsection}
\numberwithin{figure}{subsection}
\numberwithin{table}{subsection}
\allowdisplaybreaks
{\par \vspace{\baselineskip}%
 \noindent \textbf{Acknowledgements.}}%
{\par \vspace{\baselineskip}}

\usepackage{cleveref}
\crefname{thm}{Theorem}{Theorems}
\crefname{asm}{Assumption}{Assumptions}
\crefname{cor}{Corollary}{Corollaries}
\crefname{dfn}{Definition}{Definitions}
\crefname{fct}{Fact}{Facts}
\crefname{lem}{Lemma}{Lemmas}
\crefname{mth}{Theorem}{Theorem}
\crefname{ntn}{Notation}{Notations}
\crefname{prp}{Proposition}{Propositions}
\crefname{rmk}{Remark}{Remarks}
\crefname{eg}{Example}{Examples}
\crefname{section}{\S\!}{\S\S\!}
\crefname{subsection}{\S\!}{\S\S\!}
\crefname{subsubsection}{\S\!}{\S\S\!}
\crefname{equation}{equation}{equations}

\usepackage{amsthm}
\theoremstyle{definition}
\newtheorem{thm}{Theorem}[subsection]
\newtheorem{dfn}[thm]{Definition}
\newtheorem{lem}[thm]{Lemma}
\newtheorem{prp}[thm]{Proposition}

\newtheorem{cnj}[thm]{Conjecture}
\newtheorem{asm}[thm]{Assumption}
\newtheorem{fct}[thm]{Fact}
\newtheorem{rmk}[thm]{Remark}
\newtheorem{eg}[thm]{Example}
\newtheorem{mth}{Theorem}
\newtheorem{ntn}[thm]{Notation}
\newtheorem*{ntn*}{Notation}
\newtheorem*{rmk*}{Remark}

\newcommand{\ol}{\overline}
\newcommand{\ul}{\underline}

\newcommand{\wt}{\widetilde}
\newcommand{\sm}{\setminus}

\newcommand{\inj}{\hookrightarrow}
\newcommand{\srj}{\twoheadrightarrow}
\newcommand{\lto}{\longrightarrow}

\newcommand{\mto}{\mapsto}
\newcommand{\lmto}{\longmapsto}
\newcommand{\lmot}{\longmapsfrom}
\newcommand{\linj}{\lhook\joinrel\longrightarrow}
\newcommand{\lsrj}{\relbar\joinrel\twoheadrightarrow}

\newcommand{\xrr}[1]{\xrightarrow{\ #1 \ }}
\newcommand{\lsto}{\xrr{\sim}}
\newcommand{\lsot}{\xleftarrow{\ \sim \ }}

\newcommand{\bl}{\bullet}
\newcommand{\ep}{\epsilon}

\newcommand{\pd}{\partial}
\newcommand{\ceq}{\coloneqq} 

\newcommand{\ev}{\textup{ev}}
\newcommand{\od}{\textup{od}}
\newcommand{\cst}{(\textup{cst.})}
\newcommand{\tcl}{\textup{cl}}

\newcommand{\trd}{\textup{red}}
\newcommand{\pure}{\textup{pure}}
\newcommand{\txdR}{\textup{dR}}
\newcommand{\PBW}{\textup{PBW}}
\newcommand{\tsum}{{\textstyle \sum}}
\newcommand{\tprd}{{\textstyle \prod}}
\newcommand{\tboplus}{{\textstyle \bigoplus}}

\newcommand{\hf}{\frac{1}{2}}
\newcommand{\thf}{\tfrac{1}{2}}

\newcommand{\Zt}{\bbZ_2}
\newcommand{\oo}{\ol{1}} 
\newcommand{\oz}{\ol{0}} 
\newcommand{\Osc}{O_{\textup{sc}}}
\newcommand{\vac}{\ket{0}} 
\newcommand{\sld}{\slashed{\pd}} 

\newcommand{\bbC}{\mathbb{C}}
\newcommand{\bbD}{\mathbb{D}}

\newcommand{\bbN}{\mathbb{N}}
\newcommand{\bbR}{\mathbb{R}}

\newcommand{\bbZ}{\mathbb{Z}}

\newcommand{\clS}{\mathcal{S}}
\newcommand{\frg}{\mathfrak{g}}
\newcommand{\frm}{\mathfrak{m}}
\newcommand{\frp}{\mathfrak{p}}

\newcommand{\shO}{\mathscr{O}}
\newcommand{\scH}{\mathscr{H}}
\newcommand{\scL}{\mathscr{L}}
\newcommand{\scW}{\mathscr{W}}

\newcommand{\GK}{\textup{GK}}

\newcommand{\sfC}{\mathsf{C}}
\newcommand{\sfM}{\mathsf{M}}
\newcommand{\SAb}{\mathsf{SAb}}
\DeclareMathOperator{\res}{res}

\DeclareMathOperator{\Mod}{\mathsf{Mod}}

\DeclareMathOperator{\Sets}{\mathsf{Sets}}
\DeclareMathOperator{\SCom}{\mathsf{SCom}}
\DeclareMathOperator{\SMod}{\mathsf{SMod}}

\DeclareMathOperator{\TSCom}{\mathsf{TopSCom}}

\DeclareMathOperator{\gr}{gr}
\DeclareMathOperator{\id}{id}
\DeclareMathOperator{\GL}{GL}

\DeclareMathOperator{\Ann}{Ann}
\DeclareMathOperator{\Aut}{Aut}
\DeclareMathOperator{\Asc}{Aut^{sc}}
\DeclareMathOperator{\Cok}{Cok}
\DeclareMathOperator{\Der}{Der}
\DeclareMathOperator{\Dsc}{Der^{sc}_0}
\DeclareMathOperator{\End}{End}
\DeclareMathOperator{\iEnd}{\ul{End}}
\DeclareMathOperator{\SF}{SF}
\DeclareMathOperator{\spn}{span}
\DeclareMathOperator{\Hom}{Hom}

\DeclareMathOperator{\iHom}{\ul{Hom}}
\DeclareMathOperator{\Ker}{Ker}
\DeclareMathOperator{\Lie}{Lie}

\DeclareMathOperator{\Sym}{Sym}

\DeclareMathOperator{\fcS}{\mathcal{S}}
\DeclareMathOperator{\Sing}{Sing}
\DeclareMathOperator{\Spec}{Spec}
\DeclareMathOperator{\Supp}{Supp}
\DeclareMathOperator{\SSup}{SS}
\DeclareMathOperator{\SBFM}{SS^{\textup{BFM}}}
\newcommand{\abs}[1]{\left|{#1}\right|}
\newcommand{\tabs}[1]{\lvert{#1}\rvert}

\newcommand{\dbr}[1]{\llbracket #1 \rrbracket} 
\newcommand{\dpr}[1]{(\!( #1 )\!)} 
\newcommand{\fvf}[1]{\mathfrak{X}_{#1}^1} 
\newcommand{\pdd}[1]{\tfrac{\pd}{\pd #1}}
\newcommand{\rst}[2]{\left.#1\right|_{#2}}
\newcommand{\df}[2]{(#1)^\downarrow_{#2}}
\newcommand{\dR}[2]{\Omega_{#1}^{#2}} 
\newcommand{\tdR}[1]{\dR{#1}{\bl}} 
\newcommand{\fdR}[1]{\dR{#1}{1}} 
\newcommand{\HS}[2]{\mathrm{HS}_{#1}^{#2}} 
\newenvironment{sbm}{\left[\begin{smallmatrix}}{\end{smallmatrix}\right]}

\begin{document}

\title{Li filtrations of SUSY vertex algebras}
\author{Shintarou Yanagida}
\date{2021.11.10} 
\address{Graduate School of Mathematics, Nagoya University.
Furocho, Chikusaku, Nagoya, Japan, 464-8602}.
\email{yanagida@math.nagoya-u.ac.jp}
\thanks{The author is supported by supported by JSPS KAKENHI Grant Number 19K03399, %
 and also by the JSPS Bilateral Program %
 ``Elliptic algebras, vertex operators and link invariant".}

\begin{abstract}
Any vertex algebra has a canonical decreasing filtration, called Li filtration,
whose associated graded space has a natural structure of a vertex Poisson algebra.
In this note, we introduce an analogous filtration for any SUSY vertex algebra, which was 
introduced by Heluani and Kac as a superfield formalism of a supersymmetric vertex algebra.
We prove that the associated graded superspace of our filtration
has a structure of SUSY vertex Poisson algebras.
We also introduce and discuss related notions, such as Zhu's $C_2$-Poisson superalgebras, 
associated superschemes and singular supports, for SUSY vertex algebras.
\end{abstract}

\maketitle
{\small \tableofcontents}

\setcounter{section}{-1}
\section{Introduction}\label{s:intro}

\subsection*{Li filtration}

A basic philosophy in the theory of vertex algebras is that they are quantized objects.
This phrase can be understood in various ways,
and one of them is given by the work of Haisheng Li \cite{Li}.
He introduced a canonical decreasing filtration for an arbitrary vertex algebra $V$, 
nowadays called the \emph{Li filtration} of $V$.
The associated graded space has a natural structure of a vertex Poisson algebra 
\cite[Chap.\ 16]{FBZ}, i.e., a combined structure of a commutative vertex algebra
and a vertex Lie algebra.
A vertex Poisson algebra can be considered as a Poisson object in the world of vertex algebras,
and in this sense, the existence of the Li filtration shows that 
we can view any vertex algebra as a quantization of the associated graded vertex Poisson algebra.

For explicitness, let us recall the definition of the Li filtration (see also \cref{ss:NWLif}).
Let $V$ be a vertex algebra, and for $a \in V$ we denote by $a_{(n)} \in \End V$ 
the $(n)$-operator in the expansion
$Y(a,z)=\sum_{n \in \bbZ}z^{-n-1}a_{(n)}$ of the state-field correspondence.
Then the Li filtration of $V$ is a decreasing sequence of subspaces
\[
 V=E_0(V) \supset E_1(V) \supset \dotsb \supset E_n(V) \supset \dotsb
\]
consisting of 
\begin{align}\label{eq:intro-evLif}
 E_n(V) \ceq \spn\left\{
 a^1_{(-1-k_1)} \dotsm a^r_{(-1-k_r)} b \, \Bigg|
 \begin{array}{l}
  r \in \bbZ_{>0}, \, a^i,b \in V, \, k_i \in \bbN \\
  \text{satisfying $k_1+\dotsb+k_r \ge n$}
 \end{array}
 \right\}.
\end{align}

Besides the philosophical reason explained above, 
there have been increasing interest in and importance of  
the Li filtration and the associated graded vertex Poisson algebra.
In particular, there have been increasing studies with a geometric viewpoint
where one sees the spectrum of the associated graded vertex Poisson algebra
as a vertex algebra (a.k.a.\ chiral, or semi-infinite) analogue of a Poisson scheme.
Such geometric studies have a strong point that the complicated structure of 
the original vertex algebra can be treated with less complexity 
by using Poisson geometry arguments or by making appropriate analogue of them.
Among various studies, let us here only name the work of Arakawa and Moreau \cite{AM},
where the notion of chiral symplectic cores is introduced and studied
as a vertex Poisson analogue of symplectic leaves in Poisson varieties.

\subsection*{SUSY vertex algebras}

This note is written under the influence of those Poisson-geometric studies of vertex algebras.
We focus on \emph{SUSY vertex algebras} introduced by Heluani and Kac \cite{HK}.
It is a superfield formulation of vertex algebras equipped with supersymmetry.
The detailed explanation will be given in \S \ref{s:SUSY}, but here let us give a brief account.
There are two classes of structures, \emph{$N_W=N$ SUSY vertex algebras} 
and \emph{$N_K=N$ SUSY vertex algebras}, for a positive integer $N$.
Both structures are given by a linear superspace $V$, an element $\vac \in V$
and a linear operator $Y(\cdot,Z)$ on $V$ whose value for $a \in V$ is 
taken in a \emph{superfield} of the form
\begin{align}\label{eq:intro-jJ}
 Y(a,Z) = \sum_{(j|J)} Z^{j|J} a_{(j|J)}, \quad 
 Z^{j|J} \ceq z^j \zeta^{j_1} \dotsb \zeta^{j_r}, \quad a_{(j|J)} \in \End V.
\end{align}
Here $Z=(z,\zeta^1,\dotsc,\zeta^N)$ denotes a \emph{supervariable}
with an even $z$ and odd $\zeta^i$'s,
and the index $(j|J)$ runs over integers $j \in \bbZ$ and ordered subsets 
$J=\{j_1<\dotsb<j_r\}$ of the ordered set $[N] \ceq \{1<\dotsb<N\}$.
$V$ is called the superspace of states, $\vac$ is called the vacuum,
and $Y(\cdot,Z)$ is called the state-superfield correspondence.
These should satisfy the vacuum axiom and the locality axiom,
similarly as in the even case (i.e., the case of ordinary vertex algebras).

The difference between $N_W=N$ and $N_K=N$ cases lies in the translation invariance.
An $N_W=N$ SUSY vertex algebra is equipped with an even endomorphism $T$ 
and odd endomorphism $S^i$'s for $i \in [N]$ which satisfy
\[
 [T,Y(a,Z)] = \pd_z Y(a,Z), \quad [S^i,Y(a,Z)] = \pd_{\zeta^i} Y(a,Z).
\]
Here $[\cdot,\cdot]$ denotes the supercommutator 
(see \S \ref{ss:super} for a formal definition).
On the other hand, an $N_K=N$ SUSY vertex algebra is equipped with 
odd operations $S_K^i$'s for $i \in [N]$ satisfying
\[
 [S_K^i,Y(a,Z)] = D^i_Z Y(a,Z), \quad D_Z^i \ceq \pd_{\zeta^i}+\zeta^i \pd_z.
\]
In this $N_K=N$ case, we automatically have $(S_K^1)^2=\dotsb=(S_K^N)^2$
corresponding to $(D_Z^i)^2=\pd_z$,
and setting $T \ceq (S_K^i)^2$, we have $[T,Y(a,Z)]=\pd_z Y(a,Z)$.

Similarly as the superfield formalism in supersymmetric quantum field theories,
a SUSY vertex algebra has a compact description of the structure.
Besides of such formal convenience, it has also much importance.
Here we name only a few studies which have geometric flavor: 
Heluani's construction \cite{H1} of chiral algebras over super curves 
and superconformal curves, his another work \cite{H2} 
of SUSY structure on the chiral de Rham complex of a Calabi-Yau manifold,
and the study of super-modular property of SUSY conformal blocks for $N_W=1$ case 
by Heluani and Van Ekeren \cite{HV}.

\subsection*{Summary of results}

Now we can explain the main theme of this note.
For any $N_W=N$ or $N_K=N$ SUSY vertex algebra $V$,
we introduce a canonical decreasing filtration 
\[
 V=E_0(V) \supset E_1(V) \supset \dotsb \supset E_n(V) \supset \dotsb
\]
where, using the $(j|J)$-operator in \eqref{eq:intro-jJ}, we set 
\begin{align}\label{eq:intro-SUSYLif}
 E_n(V) \ceq \spn\left\{
 a^1_{(-1-k_1|K_1)} \dotsm a^r_{(-1-k_r|K_r)} b \, \Bigg|
 \begin{array}{l}
  r \in \bbZ_{>0}, \, a^i,b \in V, \, k_i \in \bbN, \, K_i \subset [N] \\
  \text{satisfying $k_1+\dotsb+k_r \ge n$}
 \end{array}
 \right\}.
\end{align}
We call it the \emph{Li filtration} of $V$ (\cref{dfn:NWLif,dfn:NKLif}).

Our definition is just neglecting the odd index $J$ and 
using the even case definition \eqref{eq:intro-evLif}.
Despite of its simple-mindedness, we need some careful arguments on SUSY vertex algebras to 
obtain the following SUSY analogue of Li's theorem in \cite[Theorem 2.12]{Li}.

\begin{mth}[{\cref{thm:NW2.12,thm:NK2.12}}]\label{mth:1}
Let $V$ be an $N_W=N$ (resp.\ $N_K=N$) SUSY vertex algebra, and $E_n = E_n(V)$
be the linear sub-superspaces of $V$ in \eqref{eq:intro-SUSYLif}.
Then the associated graded space 
\[
 \gr_E V \ceq \tboplus_{n \in \bbN} E_n/E_{n+1}
\]
has the following structure of an $N_W=N$ (resp.\ $N_K=N$) SUSY vertex Poisson algebra
(see \cref{dfn:NWP,dfn:NKP}). Let $a \in E_r$ and $b \in E_s$.
\begin{itemize}[nosep]
\item 
The commutative multiplication $\cdot$ is 
\[
 (a+E_{r+1}) \cdot (b+E_{s+1}) \ceq a_{(-1|N)}b+E_{r+s+1}.
\] 

\item
The unit is $1 \ceq \vac+E_1$.

\item
The even operator $\pd$ is $\pd(a+E_{r+1}) \ceq T a + E_{r+2}$,
and the odd operator $\delta^i$ for $i \in [N]$ is
$\delta^i(a+E_{r+1}) \ceq S^i a + E_{r+2}$
(resp.\ the odd operator $\sld^i$ is
$\sld^i(a+E_{r+1}) \ceq S_K^i a + E_{r+2}$).

\item
The SUSY vertex Lie structure $Y_-$ is 
\[
 Y_-(a+E_{r+1},Z)(b+E_{s+1}) \ceq 
 \sum_{(j|J), \, j \ge 0} Z^{-1-j|[N] \sm J} (a_{(j|J)}b+E_{r+s-j+1}).
\]
\end{itemize}
\end{mth}

Once we have obtained the Li filtration of SUSY vertex algebras,
we can discuss various related notions which are already established in the even case,
and can expect to develop a Poisson-geometric theory of SUSY vertex algebras.
In this note, as a first step toward such a geometric theory, 
we study a SUSY analogue of \emph{associated schemes} and \emph{singular supports} 
introduced by Arakawa \cite{A12}.
Our definitions and the arguments are straightforward analogue of loc.\ cit.
As a result, we have the following SUSY analogue of the equivalence between 
the $C_2$-cofiniteness and the lisse condition \cite[Theorem 3.3.3]{A12}.

\begin{mth}[\cref{thm:A12:3.3.3}]\label{mth:2}
Let $V$ be an $N_W=N$ or $N_K=N$ SUSY vertex algebra which is finitely strongly generated 
and has a lower-bounded grading (see \cref{asm:A12} for the detail). Then we have
\[
 \text{$V$ is $C_2$-cofinite (\cref{dfn:C2cof})} \iff 
 \text{$V$ is lisse (\cref{dfn:lisse})}.
\] 
\end{mth}

\subsection*{Organization}

Let us explain the organization of this note.
\cref{s:pre} is a preliminary part.
The first \cref{ss:super} summarizes the super notation used throughout this note.
The next \cref{ss:sder} is also a notational preliminary, focusing on the derivations
and the de Rham complex in the super setting. We remark that there are several notations
for the super de Rham complex in literature, and the aim is to choose and fix one of them.
The third \cref{ss:arc} is largely based on recollections, 
but has some new terminologies and lemmas.
There we give SUSY analogue of the infinite jet algebras (or arc algebras),
calling them the superjet algebras and the superconformal jet algebras, 
corresponding to $N_W=N$ and $N_K=N$ cases, respectively.
They will be the foundation of our study of commutative SUSY vertex algebras and 
SUSY vertex Poisson algebras in the following sections.

A large part of \cref{s:SUSY} gives recollection 
on the theory of SUSY vertex algebras established in \cite{HK}.
After introducing the superfield notation in \cref{ss:sf}, we explain the definition 
and basic properties of SUSY vertex algebras in \cref{ss:NW} and \cref{ss:NK},
dividing the $N_W=N$ case and $N_K=N$ case, respectively.
We also give SUSY analogue of Borcherds' iterate formula (\cref{lem:NWiter,lem:NKiter}),
which will be used in the proof of \cref{mth:1}.
The following \cref{ss:com} starts our original study.
There we study commutative SUSY vertex algebras, and show that 
the superjet algebra (resp.\ the superconformal algebra)
has a natural structure of a commutative $N_W=N$ (resp.\ $N_K=N$) SUSY vertex algebra.

\cref{s:Lif} is the main body of this note, introducing the Li filtration of
SUSY vertex algebras and showing \cref{mth:1}.
As a preliminary, we give in \cref{ss:VP} an exposition of SUSY vertex Poisson algebras.
They are introduced in \cite{HK}, using the language of SUSY Lie conformal algebras.
For the following analysis, we give a restatement in terms of SUSY vertex Lie algebras.
We also introduce the level 0 SUSY vertex algebra structure on the superjet and 
superconformal jet algebras (\cref{prp:NWlv0,prp:NKlv0}), which will be the foundation 
of the study of associated superschemes and singular supports in  \cref{s:ss}.
The $N_W=N$ case of \cref{mth:1} is shown in \cref{ss:NWLif}, 
and the $N_K=N$ case is in \cref{ss:NKLif}.
The strategy is just mimicking Li's argument for the even case \cite[\S 2]{Li},
which is based on Borcherds' commutator formula and iterate formula.
Although the SUSY analogues of commutator formula are already given in \cite{HK},
we cannot find the SUSY iterate formula in literature,
which is the reason for preparing \cref{lem:NWiter,lem:NKiter}.

\cref{s:ss} is an application of the previous \cref{mth:1}, and more or less 
a straightforward SUSY analogue of Arakawa's work \cite{A12}, as mentioned above.
In \cref{ss:C2}, we introduce the $C_2$-Poisson superalgebra (\cref{prp:C2}) and 
the associated superscheme (\cref{dfn:X_V}) of an $N_W=N$ or $N_K=N$ SUSY vertex algebra, 
and study their basic properties (e.g.\ \cref{prp:A12:2.5.1}). One remark is that 
there appears Poisson superalgebras of parity $N \bmod 2$ (\cref{dfn:sP}).
\cref{ss:ss} gives a study of singular supports, and shows \cref{mth:2}.

\subsection*{Notation and terminology}

Here we list the notations and terminology used throughout in this paper. 
\begin{itemize}[nosep]
\item
For a set $S$, we denote by $\abs{S}= \# S$ its cardinality,
and for $r,s \in S$, we denote by $\delta_{r,s}$ the Kronecker delta.

\item
We denote by $\bbN = \bbZ_{\ge 0} \ceq \{0,1,2,\dotsc\}$ the set of non-negative integers.

\item
For $\alpha,\beta \in \bbC$, we denote $\alpha \ge \beta$ if $\alpha-\beta \in \bbR_{\ge 0}$.

\item
The binomial coefficient is defined to be 
$\binom{n}{m} \ceq \frac{1}{m!} n(n-1)\dotsb(n-m+1)$ for $m \in \bbN$ and
some indeterminate or number $n$.

\item
A ring or an algebra means a unital and associative one unless otherwise stated.
A ring action on its module is denoted by the period ``$.$'', i.e., for a left module $M$ 
over a ring $R$, we denote by $r.m$ the action of $r \in R$ on $m \in M$.

\item
For a linear (super)space $V$ over a field and a subset $S \subset V$,
we denote by $\spn S$ the subspace of $V$ linearly spanned of $S$.

\item 
For a category $\sfC$, we denote $X \in \sfC$ to mean that $X$ is an object of $\sfC$.

\item
$\Sets$ denotes the category of sets.
\end{itemize}

\addtocontents{toc}{\protect\setcounter{tocdepth}{2}}
\section{Super preliminary}\label{s:pre}

\subsection{Super notation}\label{ss:super}

We start with delivering notations for super objects.
We follow \cite[Chap.\ 3]{Man2} and \cite[1.1]{KV2} with slight modifications on symbols.

\begin{ntn}\label{ntn:super}
Here is the first set of notations and terminologies:
\begin{itemize}[nosep]
\item
We denote by $\Zt \ceq \bbZ/2\bbZ = \{\oz,\oo\}$ 
the \emph{parity} or the \emph{supergrading}.

\item
A \emph{super abelian group} or a \emph{supermodule} means a $\Zt$-graded module 
\[
 M = M_{\oz} \oplus M_{\oo}.
\]
For each $\gamma \in \Zt$, an element $v \in M_{\gamma}$ is called \emph{of pure parity}, 
and for such $v$, we denote $p(v) \ceq \gamma$.
We denote by $M_{\pure} \ceq M_{\oz} \sqcup M_{\oo}$ the set of elements of pure parity.
An element $v \in M_{\oz}$ is called \emph{even},
and $v \in M_{\oo}$ is called \emph{odd}.
If $M=M_{\oz}$, then $M$ is called \emph{purely even},
and if $M=M_{\oo}$, then $M$ is called \emph{purely odd}.

\item
A \emph{morphism $M \to N$ of supermodules} is 
a module homomorphism which preserves the $\Zt$-gradings.
The category of supermodules is denoted by $\SAb$.
We denote by $\Pi\colon \SAb \to \SAb$ the parity change functor:
$(\Pi M)_{\oz}=M_{\oo}$ and $(\Pi M)_{\oo}=M_{\oz}$.
It is involutive: $\Pi^2 = \id$.

\item
The category $\SAb$ of supermodules is a symmetric monoidal category 
with monoidal structure $\otimes$ given by the standard graded tensor product.
The symmetry transformation is given by 
\begin{align}\label{eq:kos-sgn}
 M \otimes N \to N \otimes M, \quad m \otimes n \mapsto (-1)^{p(m) p(n)} n \otimes m.
\end{align}
We always consider $\SAb$ as the category equipped with this symmetric monoidal structure.

\item
A \emph{supering} $R$ is an associative ring object in $\SAb$.
The \emph{supercommutator} in $R$ is defined by 
\begin{align}\label{eq:sbra}
 [r,s] \ceq r s - (-1)^{p(r) p(s)} s r
\end{align}
for $r,s \in R$ of pure parity, and extended by linearity.
It satisfies the \emph{super Jacobi rule}
\[
 (-1)^{p(r)p(t)}[r,[s,t]]+(-1)^{p(s)p(r)}[s,[t,r]]+(-1)^{p(t)p(s)}[t,[r,s]]=0 
 \quad (r,s,t \in R_{\pure}).
\]

\item
A \emph{commutative superring} is a commutative associative ring object in $\SAb$,
i.e., a superring whose supercommutator always vanishes.
A \emph{morphism $R \to S$ of commutative superrings} is a morphism in $\SAb$
respecting the ring structures.
Given such a morphism $R \to S$ of commutative superrings, 
we call $S$ a \emph{superalgebra over $R$} or an \emph{$R$-superalgebra}.
Commutative $R$-superalgebras form a category, which is denoted by $\SCom R$.
\end{itemize}
\end{ntn}

Hereafter we always assume:
\begin{itemize}[nosep]
\item
In every commutative superring $R$, the element $2$ is invertible. 
\end{itemize}
As a consequence, if $r \in R_{\oo}$, then $r^2=\frac{1}{2}[r,r]=0$, 
so all odd elements are nilpotent.

\begin{ntn*}
Next we introduce terminologies for modules over a commutative superring $R$.
\begin{itemize}[nosep]
\item 
A \emph{left $R$-supermodule $L$} is a supermodule with a left module structure 
over the underlying ring of $R$ such that the module structure 
$R \times L \to L$, $r \otimes l \mapsto r.l$ is a morphism in $\SAb$.

A \emph{right $R$-supermodule} $M$ is similarly defined, 
and can be regarded as a left $R$-supermodule by 
$r.m \ceq (-1)^{p(r)p(m)}m.r$ for pure elements $r \in R$ and $m \in M$.
Conversely, a left $R$-supermodule is regarded as a right $R$-supermodule
by the same correspondence.
Henceforth we call a left $R$-supermodule just by an \emph{$R$-supermodule}.

\item
$R$-supermodules form a category, which is denoted by $\SMod R$.
For $M,N \in \SMod R$, a \emph{morphism} $M \to N$ of $R$-supermodules 
is an $R$-linear map preserving the parity.
The category $\SMod R$ inherits the symmetric monoidal structure $\otimes_R$ 
from that of $\SAb$.

\item
If $R=k$ is a field, then a $k$-supermodule is called a \emph{linear $k$-superspace}.
\end{itemize}
\end{ntn*}

Let $R$ be a commutative superring.
The category $\SMod R$ of $R$-supermodules has a natural $\Zt$-graded enhancement.
Let $M$ and $N$ be $R$-supermodules.
We denote by $\Hom_{\Mod R}(M,N)$ the module of homomorphisms $M \to N$ 
of the underlying ring of $R$. 
It is an $R$-supermodule with 
\[
 \Hom_{\Mod R}(M,N) = \Hom_{\Mod R}(M,N)_{\oz} \oplus \Hom_{\Mod R}(M,N)_{\oo},
\]
where $\Hom_{\Mod R}(M,N)_{\oz}$ and $\Hom_{\Mod R}(M,N)_{\oo}$ denote 
the submodule of parity preserving and of parity changing homomorphisms, respectively,
and the $R$-action is given by $(r.f)(m) \ceq r.f(m)$ 
for $r \in R$, $f \in \Hom_{\Mod R}(M,N)$ and $m \in M$.
Note that $\Hom_{\Mod R}(M,N)_{\oz}=\Hom_{\SMod R}(M,N)$, 
the morphism set of the category $\SMod R$.

\begin{ntn*}
We call an element of $\Hom_{\Mod R}(M,N)$ an \emph{$R$-homomorphism}.
\end{ntn*}

The $R$-supermodule $\Hom_{\Mod R}(M,N)$ is indeed an \emph{internal hom}
in the symmetric monoidal category $(\SMod R, \otimes_R)$.
To stress this point, we denote 
\begin{align}\label{eq:iHom}
 \iHom_{\SMod R}(M,N) := \Hom_{\Mod R}(M,N) \in \SMod R.
\end{align}
The definition of internal hom claims that we have 
\[
 \Hom_{\SMod R}(L,\iHom_{\SMod R}(M,N)) \cong \Hom_{\SMod R}(L \otimes_R M, N)
\]
for any $L,M,N \in \SMod R$.


The parity change functor $\Pi$ on $\SAb$ naturally lifts to $\SMod R$.
For $M \in \SMod$, the object $\Pi M$ is defined as in $\SAb$,
ant the $R$-supermodule structure is given by  
\begin{align}\label{eq:PiM}
 r.(\Pi m) \ceq (-1)^{p(r)} \Pi(r.m) \quad (r \in R \text{ of pure parity}, \  m \in M),
\end{align}
where $\Pi m \in \Pi M$ is the element corresponding to $m \in M$.
For a morphism $f \in \Hom_{\SMod R}(M,N)$, 
we define $f^{\Pi}\colon \Pi M \to \Pi N$ 
to be the morphism which agree with $f$ as a map of sets.
Here the symbol $f^{\Pi}$ is borrowed from \cite[Chap.\ 3, \S 1.5]{Man2}. 
This construction gives the parity change functor $\Pi\colon \SMod R \to \SMod R$.

For $f \in \Hom_{\Mod R}(M,N)$, we define $\Pi f \in \Hom_{\Mod R}(M, \Pi N)$ and 
$f \Pi \in \Hom_{\Mod R}(\Pi M,N)$ by
\begin{align}\label{eq:Pi-f-Pi}
 (\Pi f)(m) \ceq \Pi(f(m)), \quad (f \Pi)(\Pi m) \ceq f(m) \quad 
 (m \in M).
\end{align}
Then we have $f^{\Pi} = \Pi f \Pi$. 
More generally, we have the following isomorphisms of $S$-modules:
\begin{align}\label{eq:Pi-Hom}
\begin{array}{ccccc}
 \iHom_{\SMod R}(\Pi M,N) &\lsot &\iHom_{\SMod R}(M,N) &\lsto &\iHom_{\SMod R}(S,\Pi N), \\
            f \Pi &\lmot &f            &\lmto &\Pi f.
\end{array}
\end{align}

\subsection{Super derivations and the de Rham complex}\label{ss:sder}

Next, we introduce derivations in super setting.

\begin{dfn}[{c.f.\ \cite[Chap.\ 3, \S 2.8]{Man2}}]\label{dfn:der}
Let $R \to A$ be a morphism of commutative superrings, and $M$ be an $A$-supermodule.
\begin{enumerate}[nosep]
\item 
An \emph{$R$-derivation $D: A \to M$ of parity $q \in \Zt$} 
is an $R$-homomorphism $D \in \Hom_{\Mod R}(A,M)_q$ of parity $q$ 
(where $A,M$ are regarded as supermodules) 
such that $D(r)=0$ for $r \in R$, and that the super Leibniz rule holds
for $a,b \in A$ of pure parity:
\[
 D(a b) = D(a).b + (-1)^{p(a) q} a.D(b).
\]
An $R$-derivation of parity $\oz$ (resp.\ of parity $\oo$) is  also called 
an \emph{even $R$-derivation} (resp.\ an \emph{odd $R$-derivation}).

\item
The set of $R$-derivations $A \to M$ of parity $p$ is denoted by $\Der_R(A,M)_p$,
which is a module.
We also define a supermodule $\Der_R(A,M)$ by
\[
 \Der_R(A,M) \ceq \Der_R(A,M)_{\oz} \oplus \Der_R(A,M)_{\oo}.
\]
\end{enumerate}
\end{dfn}

Note that for an even $R$-derivation $D\colon A \to M$, the maps $\Pi D\colon A \to \Pi M$ 
and $D\Pi\colon \Pi A \to M$ given in \eqref{eq:Pi-f-Pi} are odd $R$-derivations.
Conversely, if $D$ is an odd $R$-derivation, then $\Pi D$ and $D \Pi$ are even $R$-derivations.
These are checked by using the $R$-supermodule structure \eqref{eq:PiM} of $\Pi M$.

\begin{ntn}\label{ntn:vf1}
In the case $M=A$, we denote 
\[
 \fvf{A/R} = \Der_R(A) \ceq \Der_R(A,A),
\]
which is a Lie superalgebra in terms of the supercommutator \eqref{eq:sbra}.
\end{ntn}

The supermodule $\Der_R(A,M)$ is an $A$-supermodule, 
since for $a,b,c \in A$ and $D \in \Der_R(A,M)$ of pure parity, 
we have $p(a)+p(D)=p(a D)$ and 
\[
 (a D)(b c) = a\bigl(D(b).c+(-1)^{p(b) p(D)}b.D(c)\bigr) 
            = (a D)(b).c+(-1)^{p(b) p(a D)}b.(a D)(c).
\]
In particular, we have a functor 
\[
 \Der_R(A,\cdot)\colon \SMod A \lto \SMod A.
\]
As in the non-super case \cite[Tag 00RM, 10.131 Differentials]{SP}, 
this functor is represented by the \emph{supermodule of K\"ahler differentials}.
In the super case, there are two versions.
For the definition, recall that a \emph{free supermodule over a commutative superring $R$} is 
a supermodule which is a module over the underlying algebra of $R$ with basis of pure parity.
See \cite[Chap.\ 3, \S 1.6]{Man2} for the detail.

\begin{dfn}\label{dfn:dR1}
Let $\varphi\colon R \to A$ be a morphism of commutative superrings.
Consider the free $A$-supermodules
\[
 E \ceq \bigoplus_{r \in R_{\pure}} A[r] \oplus 
 \bigoplus_{a,b \in S_{\pure}} A[(a,b)] \oplus 
 \bigoplus_{f,g \in S_{\pure}} A[(f,g)], \quad 
 F \ceq \bigoplus_{a \in S_{\pure}} A[a], 
\]
where $R_{\pure} \ceq R_{\oz} \sqcup R_{\oo}$ and $A_{\pure} \ceq A_{\oz} \sqcup A_{\oo}$
as sets (see \cref{ntn:super}), and the parity is defined by $p([(a,b)]) \ceq p(a)+p(b)$,
$p([r]):=p(r)$ and $p([a]) \ceq p(a)$.
We define two morphisms 
\[
 \psi_{\ev}, \psi_{\od}\colon E \lto F
\]
of $A$-supermodules by the following rules: 
\begin{align}\label{eq:Omega:map}
\begin{split}
&[r] \lmto [\varphi(r)], \quad [(a,b)] \lmto [a+b]-[a]-[b], \\
&[(f,g)] \lmto 
 \begin{cases}
 [f g]-           f[g]-(-1)^{p(f) p(g)}g[f] & (\text{for } \psi_{\ev}) \\
 [f g]-(-1)^{p(f)}f[g]-(-1)^{p(f) p(g)}g[f] & (\text{for } \psi_{\od})
 \end{cases}.
\end{split}
\end{align}
Then we define $A$-supermodules 
\[
 \fdR{A/R,\ev} \ceq \Cok \psi_{\ev}, \quad 
 \fdR{A/R,\od} \ceq \Cok \psi_{\od},
\]
and define maps 
\[
 d_{\ev}\colon A \lto \fdR{A/R,\ev}, \quad 
 d_{\od}\colon A \lto \fdR{A/R,\od},
\]
to be the ones sending $a$ to the class of $[a]$.
We call $\fdR{A/R,\ev}$ and $\fdR{A/R,\od}$ 
the \emph{$A$-supermodules of K\"aler differentials of $A$ over $R$},
and call $d_{\ev}$ and $d_{\od}$ the \emph{universal differentials}. 
\end{dfn}

The module $\fdR{A/R,\ev}$ (resp.\ $\fdR{A/R,\od}$) is spanned by elements 
of the form $a.d_{\ev} b$ (resp.\ $a.d_{\od} b$) with $a,b \in A$, respectively.
The rules \eqref{eq:Omega:map} yield
\begin{align*}
&d_{\ev}(r) = 0, & &d_{\ev}(a+b) = d_{\ev} a + d_{\ev} b, & 
&d_{\ev}(a b) = d_{\ev} a.b + a.d_{\ev} b\\
&d_{\od}(r) = 0, & &d_{\od}(a+b) = d_{\od} a + d_{\od} b, & 
&d_{\od}(a b) = d_{\od} a.b + (-1)^{p(a)}a.d_{\od} b. 
\end{align*}
By \cref{dfn:der}, $d_{\ev}$ and $d_{\od}$ are an even and odd $R$-derivation, 
respectively. We also have the natural isomorphism of $A$-supermodules
\begin{align}\label{eq:Omega-eo}
 \fdR{A/R,\od} \lsto \Pi \fdR{A/R,\ev}, \quad 
 a.d_{\ev} b \lmto \Pi(a.d_{\od}b) \quad (a,b \in A)
\end{align}
Here we used the symbol $\Pi(a)$ for $a \in A$ in the same meaning 
as in \eqref{eq:PiM} and \eqref{eq:Pi-f-Pi}.
Then, by \eqref{eq:Pi-Hom}, we have an isomorphism
$\iHom_{\SMod R}(A,\fdR{A/R,\ev}) \cong \iHom_{\SMod R}(A,\fdR{A/R,\od})$,
which corresponds to the identity $d_{\od}=\Pi d_{\ev}$.

The following statement is an analogue of the universality of the module of 
K\"ahler differentials in the non-super case \cite[Tag 00RM, Lemma 10.131.3]{SP}.

\begin{lem}
Let $R \to A$ be a morphism of commutative superrings. 
Then we have isomorphisms of $A$-supermodules
\begin{align}\label{eq:Omega-univ}
\begin{array}{ccccc}
 \iHom_{\SMod A}(\fdR{A/R,\od},M) &\lsot 
&\iHom_{\SMod A}(\fdR{A/R,\ev},M) &\lsto &\Der_R(A,M), \\
\phi \Pi &\lmot &\phi &\lmto &\phi \circ d_{\ev},
\end{array}
\end{align}
which are functorial with respect to $M \in \SMod A$.
\end{lem}

\begin{proof}
The left isomorphism is given by \eqref{eq:Pi-Hom} and \eqref{eq:Omega-eo}.
As for the right isomorphism, 
we can easily find that the map $\phi \mapsto \phi \circ d_{\ev}$ 
is an injective morphism of $A$-supermodules.
Given any $D \in \Der_R(A,M)$, we have the induced map 
$[D]\colon \bigoplus_{a \in A_{\pure}} A[a] \to M$, 
where we used the symbols in \cref{dfn:dR1}.
Then the condition for $D$ to be an even or odd $R$-derivation is equivalent to the one
that $[D]$ annihilates the image of the map \eqref{eq:Omega:map}.
\end{proof}

\begin{ntn*}
Hereafter, if no confusion may arise, we simply denote 
\[
 (\fdR{A/R},d) \ceq (\fdR{A/R,\ev},d_{\ev}).
\]
\end{ntn*}

Let us give another description of the $A$-supermodule $\fdR{A/R}$ as 
the first order neighborhood of the diagonal 
$\Delta_{A/R} \subset \Spec A \times_{\Spec R} \Spec A = \Spec A \otimes_R A$,
which is also an analogous statement as the non-super case.


\begin{lem}[{c.f.\ \cite[Tag 00RM, Lemma 10.131.13]{SP}}]
Let $R \to A$ be a morphism of commutative superrings,
and $I \ceq \Ker(A \otimes_R A \to A)$ be the kernel of the multiplication map.
Then the map 
\[
 \fdR{A/R} \lto I/I^2, \quad a d b \lmto (a \otimes b - ab \otimes 1 \bmod I^2)
\]
is an isomorphism of $A$-supermodules.
\end{lem}

\begin{proof}
The argument in the non-super case works with minor changes.
\end{proof}

Finally we explain the de Rham complex in super setting.
See also \cite[Chap.\ 3, \S 2.5, \S 4.4]{Man2}.

\begin{ntn}\label{ntn:tdR}
Let $R$ be a commutative superring.
\begin{itemize}[nosep]
\item 
For an $R$-supermodule $M$, we denote by $\Sym_R^{\bl} M$ 
the \emph{symmetric superalgebra of $M$ over $R$}.
It is a commutative $R$-superalgebra equipped with $\bbN$-grading expressed by $\bl$.

\item
Let $\fdR{A/R}$ be the supermodule of K\"ahler differentials 
for a commutative $R$-superalgebra $A$.
The \emph{de Rham complex of $A$ over $R$} is the complex defined to be  
\[
 \tdR{A/R} \ceq \bigl(\Sym_A^{\bl}(\Pi \fdR{A/R}),d_{\txdR}\bigr),
\]
where the \emph{de Rham differential} $d_{\txdR}\colon \dR{A/R}{p} \to \dR{A/R}{p+1}$ 
is induced by the odd universal differential 
\[
 \Pi d_{\ev}=d_{\od}\colon A \lto \Pi \fdR{A/R,\ev}=\fdR{A/R,\od}
\] 
in a similar fashion as in the non-super case.
If we focus on the $A$-superalgebra structure, we denote $\dR{A/R}{}$ and call it
the \emph{de Rham superalgebra of $A$ over $R$}.
\end{itemize}
\end{ntn}

Using the terminology of \cite[1.1.16]{BD}, we can say that 
$\Sym_R^{\bl} M$ is an $\bbN$-graded superalgebra, and $\tdR{A/R}$ is a dg $A$-superalgebra.
Henceforth we also call $\tdR{A/R}$ the de Rham dg superalgebra of $A$ over $R$.


$\dR{A/R}{}$ enjoys the universal property in \cref{fct:SdR} below.

\begin{ntn}
Let $\Lambda_R[\zeta]$ be the exterior algebra of one variable $\zeta$ over $R$.
In other words, it is a commutative $R$-superalgebra consisting of the elements 
of the form $r+s\zeta$ for $r,s \in R$ with $\zeta$ an odd element, 
and the multiplication is given by 
$(r+s\zeta)(r'+s'\zeta)=r r'+(r s'+(-1)^{(p(s)+1)p(r')}r' s)\zeta$.
\end{ntn}

\begin{fct}[{\cite[Proposition 3.3]{KV2}}]\label{fct:SdR}
For a morphism $R \to A$ of commutative superrings, the de Rham superalgebra 
$\dR{A/R}{}$ represents the functor 
$\Hom_{\SCom R}(A,\cdot \otimes_R \Lambda_R[\zeta])\colon \SCom R \to \Sets$.
In other words, there is a functorial bijection
\[
 \Hom_{\SCom R}(\dR{A/R}{},B) = \Hom_{\SCom R}(A, B \otimes_R \Lambda_R[\zeta]).
\]
\end{fct}


The construction $A \mto \dR{A/R}{}$ is functorial, so we can introduce:

\begin{ntn}
Following \cite{KV2}, we call the functor the \emph{de Rham spectrum functor} and denote 
\[
 \clS(A) \ceq \dR{A/R}{}.
\] 
\end{ntn}

We can apply the functor $\clS$ repeatedly, and the obtained superalgebra 
$\clS^N(A)$ for $N \in \bbZ_{>0}$ is equipped with the de Rham differentials 
$d_{\txdR}^i$ for $i=1,\dotsc,N$. These are composable, 
and satisfies the anti-commutation relation $[d_{\txdR}^i,d_{\txdR}^j]=0$. 
For later reference, we introduce:

\begin{ntn}\label{ntn:dRJ}
For an ordered set $J=\{j_1<\dotsb<j_r\} \subset \{1<2<\dotsb<N\}$, 
we define $d_{\txdR}^J\colon A \to \clS^N(A)$ by 
$d_{\txdR}^J \ceq d_{\txdR}^{j_1} \dotsm d_{\txdR}^{j_r}$.
\end{ntn}

We have an obvious multi-variable analogue of \cref{fct:SdR}.

\begin{lem}\label{lem:SdR}
Let $\Lambda_R[\zeta^1,\dotsc,\zeta^N]$ be exterior algebra of $N$ variables over $R$.
Then the $N$-times application of the de Rham spectrum functor $\clS^N(A)$
represents the functor 
$\Hom_{\SCom R}(A,? \otimes_R \Lambda_R[\zeta^1,\dotsc,\zeta^N])$.
In other words, we have a functorial bijection
\[
 \Hom_{\SCom R}\bigl(\clS^N(A),B\bigr) \lsto 
 \Hom_{\SCom R}\bigl(A,B \otimes_R \Lambda_R[\zeta^1,\dotsc,\zeta^N]\bigr).
\]
\end{lem}

\subsection{Superjet and superconformal jet algebras}\label{ss:arc}

We give a brief recollection on \emph{infinite jet algebras}, which we mean the spectrum 
of infinite jet affine schemes in the sense of 
\cite[\S 2,\S 3]{EM}, \cite{V} and \cite[2.3.2, 2.3.3]{BD}. 
We present it in super setting, following \cite[\S 4.2]{KV2}.
We also recall the notion of \emph{superconformal vector fields} or \emph{SUSY structure}
on super curves.
In this subsection, the base field $k$ is assumed to be of characteristic $0$.

\subsubsection{Superjet algebras}

As before, let $R$ be a commutative superring, 
and denote by $\SCom R$ the category of commutative $R$-superalgebras.
Recall that in \cref{dfn:dR1} we used free $R$-supermodules.
In the next \cref{dfn:HS}, 
we use the \emph{polynomial $R$-superalgebra} $R[x_j \mid j \in J]$, 
where the generators $x_j$ are of pure parity and commutative.
It can also be called the \emph{free commutative $R$-superalgebra} generated by $x_j$'s.
We also use the set $A_{\pure} \ceq A_{\oz} \sqcup A_{\oo}$ of elements of pure parity.

\begin{dfn}[{\cite[Definition 1.3]{V}}]\label{dfn:HS}
Let $A \in \SCom R$, and $f\colon R \to A$ be the structure morphism.
For $m \in \bbN \sqcup \{\infty\}$, define a commutative $A$-superalgebra 
$\HS{A/R}{m}$ to be the quotient of the polynomial $A$-superalgebra
\[
 A[x^{(i)} \mid x \in A_{\pure}, \, i=1,\dotsc,m]
\]
by the ideal $I$ generated by the following terms:
\begin{align*}
                             f(r)^{(i)} \quad &(  r \in R, \, i=1,\dotsc,m), \\
 (x+y)^{(i)}-x^{(i)}-y^{(i)}            \quad &(x,y \in A, \, i=1,\dotsc,m), \\
 (x y)^{(i)}-\sum_{j+k=i}x^{(j)}y^{(k)} \quad &(x,y \in A, \, i=0,\dotsc,m),
\end{align*}
where we denote $x^{(0)} \ceq x$ for $x \in A$,
and if $m=\infty$ we interpret the range ``$i=1,\dotsc,m$'' as $i \in \bbZ_{>0}$.
We call $\HS{A/R}{m}$ the ($A$-superalgebra of) \emph{Hasse-Schmidt derivations}.
We also define $R$-homomorphisms $d_i$ ($i=0,\dotsc,m$) by 
\[
 d_i\colon A \lto \HS{A/R}{m}, \quad x \lmto (x^{(i)} \bmod I),
\]
and call them the \emph{universal derivations}.
\end{dfn}

\begin{rmk}\label{rmk:HS}
A few remarks are in order.
\begin{enumerate}[nosep]
\item 
We have $\HS{A/R}{0} \cong A$, and $\HS{A/R}{1} \cong \Sym_A \fdR{A/R}$,
the symmetric $A$-superalgebra of the module $\fdR{A/R}$ of K\"ahler differentials,
forgetting the $\bbN$-grading in \cref{ntn:tdR}.

\item
$\HS{A/R}{m}$ is a commutative $R$-superalgebra by the structure morphism $f\colon R \to A$.

\item
$d_1$ an even derivation of the $R$-superalgebra $A$ (\cref{dfn:der}).
If the base field $k$ is of characteristic $0$, 
then $d_n = \frac{1}{m!}d_1^n$ for $n \in \bbN$.
\end{enumerate}
\end{rmk}

The Hasse-Schmidt derivations $\HS{A/R}{m}$ represents the functor of $m$-th jets,
explained in the following \cref{fct:HS}.

\begin{ntn*}
Let $R$ be a commutative superring, and $z$ be an even indeterminate.
\begin{itemize}[nosep]
\item
We regard the formal power series ring $R\dbr{z}$ as 
\[
 R\dbr{z} = \varprojlim_{m \in \bbN} R[z]/(z^{m+1}),
\]
and as a complete topological superring with respect to the $z$-adic topology.

\item
We denote by $\TSCom R$ the category of topological commutative $R$-superalgebras
and continuous morphisms.
In particular, we have $R\dbr{z} \in \TSCom R$.
\end{itemize}
\end{ntn*}

\begin{fct}[{\cite[Corollary 1.8]{V}}]\label{fct:HS}
For $A \in \SCom R$ and $m \in \bbN \cup \{\infty\}$, the Hasse-Schmidt derivations 
$\HS{A/R}{m}$ represents the functor $\SCom R \to \Sets$ given by
\[
 \begin{cases}
  \Hom_{ \SCom R}\bigl(A,? \otimes_R R[z]/(z^{m+1})\bigr) & (m \in \bbN) \\
  \Hom_{\TSCom R}\bigl(A,? \otimes_R R[z]/(z^{m+1})\bigr) & (m=\infty)
 \end{cases}.
\]
More precisely, we have a functorial bijection 
\[
 \Hom_{\SCom R} \bigl(\HS{A/R}{m},B\bigr) \lto 
 \Hom_{\SCom R} \bigl(A,B[z]/(z^{m+1})\bigr), \quad 
 \phi \lmto \bigl(x \mto \tsum_{i=0}^m z^i \phi(d_i x)\bigr)
\]
in the case $m \in \bbN$, and a similar one given by 
$\phi \mapsto \bigl(x \mapsto \sum_{i=0}^\infty z^i \phi(d_i x)\bigr)$
in the case $m=\infty$.
\end{fct}

Combining \cref{lem:SdR} and \cref{fct:HS}, we have \cref{prp:sj} below.
We use the following notations.

\begin{ntn}\label{ntn:sd}
Let $N \in \bbZ_{>0}$, and $k$ be a field of characteristic $0$.
We denote by $[N]$ the ordered set $\{1<2<\dotsb<n\}$.
\begin{itemize}[nosep]
\item
for $A \in \SCom k$ we denote
\begin{align}\label{eq:Ainf}
 A_\infty \ceq \HS{A/k}{\infty}
\end{align}
and call it the (\emph{infinite}) \emph{jet algebra} of $A$.

\item
Let $Z=(z,\zeta^1,\dotsc,\zeta^N)$ be a set of indeterminates 
with $z$ even and $\zeta^j$ odd which are (super-) commutative in each other.
We call $Z$ an \emph{$1|N$-supervariable}. 
For $j \in \bbZ$ and an ordered subset $J=\{j_1<\dotsb<j_r\} \subset [N]$, 
we denote
\[
 Z^{j|J} \ceq z^j \zeta^J = z^j \zeta^{j_1} \dotsm \zeta^{j_r}.
\]

\item
Let $\frm \ceq (Z)$ be the ideal of the polynomial superalgebra $k[Z] \in \SCom k$
generated by $z$ and $\zeta^i$'s.
The natural projections $k[Z]/\frm^{n+1} \to k[Z]/\frm^{n}$ 
form a filtered system, and we denote 
\[
 O = O^{1|N} \ceq k\dbr{Z} = \varprojlim_{n \in \bbN} k[Z]/\frm^{n+1},
\]
regarding it as a complete topological $k$-superalgebra by $\frm$-adic topology.
\end{itemize}
\end{ntn}

\begin{prp}\label{prp:sj}
Let $N \in \bbZ_{>0}$, $k$ be a field of characteristic $0$, 
and $Z=(z,\zeta^1,\dotsc,\zeta^N)$ be a supervariable.
For $A \in \SCom k$, the commutative superalgebra 
\[
 A^O = A^{k\dbr{Z}} \ceq \clS^N(A_\infty)
\]
represents the functor
$\Hom_{\TSCom k}(A,\cdot \otimes_k O)\colon \SCom k \to \Sets$.
More precisely, we have a functorial bijection
\begin{align}\label{eq:sj-bij}
 \Hom_{\SCom k}(A^O,B) \lsto 
 \Hom_{\TSCom k}(A,B \otimes_k O), \quad 
 \phi \lmto \sum_{n \in \bbN, \, J \subset [N]} Z^{n|J} \phi \circ d_{n|J},
\end{align}
where $d_{n|J}\colon A \to A^O$ is the composition 
\begin{align}\label{eq:sj:d}
 d_{n|J} \ceq d_{\txdR}^J \circ d_n
\end{align}
of the universal differentials $d_n\colon A \to A_\infty$ 
and the product of the de Rham differentials $d_{\txdR}^J$ in \cref{ntn:dRJ}.
We call $\clS^N(A_\infty)$ the (\emph{infinite}) \emph{$1|N$-superjet algebra} of $A$.
\end{prp}

The next \cref{lem:dnJ} describes a natural basis of the superalgebra $A^O$.
We omit the proof.

\begin{lem}\label{lem:dnJ}
In the situation of \cref{prp:sj}, we denote 
\[
 a^{[n|J]} \ceq d_{n|J}(a) \in A^O \quad (a \in A, \, n \in \bbN, \, J \subset [N]).
\]
Then $A^O$ is linearly spanned by the elements of the form
\[
 a_1^{[n_1|J_1]} \dotsm a_r^{[n_r|J_r]}  \quad 
 (a_i \in A, \, n_i \in \bbN, \, J_i \subset [N]).
\]
\end{lem}

Let us denote $e_i \ceq \{i\} \subset [N]$ for each $i \in [N]$.
Then, the operations $d_{1|0}$ and $d_{0|e_i}$ on $A^O$ are even and odd derivations,
respectively. Let us record this observation as:

\begin{lem}\label{lem:sj-TS}
The superalgebra $A^O$ has an even derivation $T$ and odd derivations $S^i$ for $i \in [N]$ 
determined by 
\[
 T(a^{[n|J]}) \ceq (n+1)a^{[n+1|J]}, \quad
 S^i(a^{[n|J]}) \ceq \sigma(e_i,J) a^{[n|J \sm e_i]}.
\]
Here $\sigma(e_i,J) \in \{\pm1,0\}$ is determined by the following rule:
If $i \in J$, then $\sigma(e_i,J) \ceq 0$.
Otherwise, let $e_i \cup J$ be the reordered set of the union, 
and determine $\sigma(e_i,J)$ by $\zeta^i \zeta^J = \sigma(e_i,J) \zeta^{e_i \cup J}$.
\end{lem}

For later reference, let us introduce the notion of differential superalgebras.

\begin{dfn}\label{dfn:dsa}
The following notions are defined over some base superring $R$, but we suppress it.
\begin{enumerate}[nosep]
\item
A \emph{differential superalgebra} is a pair $(A,D)$ of 
a commutative superalgebra $A$ and a derivation on $A$.
A \emph{morphism} $(A,D) \to (B,D')$ of differential superalgebras is naturally defined.

\item
A subset $I \subset A$ \emph{generates} $(A,D)$ if every element of $A$ can be written 
as a polynomial of $D^r a$ with $r \in \bbN$ and $a \in A$.
\end{enumerate}
\end{dfn}

We have a multi-derivation analogue of these notions: 
A tuple $(A,D_i)$ of a commutative superalgebra $A$ and derivations $D_i$'s 
is also called a differential superalgebra. 
Then, by \cref{lem:sj-TS,lem:dnJ}, we have:

\begin{lem}\label{lem:sj-dsa}
For any $A \in \SCom k$, $(A^O,T,S^i)$ is a differential superalgebra, 
and it is generated by $A$.
\end{lem}

Now the following assertion is easily shown:

\begin{lem}\label{lem:sj-gen}
Let $A,B \in \SCom k$, and assume the following conditions.
\begin{itemize}[nosep]
\item $B$ contains $A$ as a sub-superalgebra.
\item $B$ has an even derivation $\pd$ and $N$ odd derivations $\sld^i$ for $i \in [N]$.
\item The differential superalgebra $(B,\pd,\sld^i)$ is generated by $A$.
\end{itemize} 
Then the injection $A \inj B$ can be extended to a surjective morphism 
$(A^O,T,S^i) \srj (B,\pd,\sld^i)$ of differential superalgebras.
\end{lem}

Finally, let us introduce finite superjet algebras.

\begin{prp}\label{prp:fsj}
For a supervariable $Z=(z,\zeta^1,\dotsc,\zeta^N)$ and $m \in \bbN$, we denote
\[
 O_m = O^{1|N}_m \ceq k[Z]/(z^{m+1}).
\]
Then, for $A \in \SCom k$, the commutative superalgebra
\[
 A^{O_m} = A^{k[Z]/(z^{m+1})} \ceq \clS^N(\HS{A/k}{m})
\]
satisfies the universal property
\[
 \Hom_{\SCom k}(A^{O_m},B) \lsto \Hom_{\SCom k}(A,B \otimes_k O_m), \quad 
 \phi \lmto \sum_{0 \le n \le m, \, J \subset [N]} Z^{n|J} \phi \circ d_{n|J},
\]
where $d_{n|J}$ is given by the same formula as \eqref{eq:sj:d}.
We call $A^{O_m}$ the \emph{$m$-th superjet algebra} of $A$.
\end{prp}

Corresponding to the truncation map $O \to O_m$, $z^n \mapsto 0$ for $n \ge m+1$, we have 
a morphism of commutative superalgebras $\HS{A/k}{m} \to \HS{A/k}{\infty}=A_\infty$,
which induces another morphism 
\begin{align}\label{eq:fsj-sj}
 A^{O_m} \lto A^O.
\end{align}
We call it the \emph{projection} of superjet algebras.

\subsubsection{Superconformal jet algebras}\label{sss:scj}

Let us recollect the \emph{super Riemann surface} structure or 
the \emph{SUSY structure} on the superdisk $O=O^{1|N}$.
We denote the topological Lie superalgebra of $k$-derivations on $O$ (\cref{ntn:vf1}) as 
\begin{align}\label{eq:DerO}
 \Der O \ceq \Der_k(O) = \fvf{O/k}.
\end{align}

By \cite[\S 6.2]{FBZ} and \cite[\S 3.1.2]{H1}, the set of automorphisms
\[
 \Aut O \ceq \Aut_{\TSCom k}(O)
\]
is a group superscheme over $k$, and the set of $k$-valued points is described as 
\[
 (\Aut O)(k) = \Bigl\{ Z=(z,\zeta^i) \mapsto 
 \Bigl(\sum_{n+\abs{J}\ge1} a_{n,J} Z^{n|J}, 
       \sum_{n+\abs{J}\ge1} b^i_{n,J} Z^{n|J} \Bigr) \, | \, 
 \begin{sbm}a_{1,0} & a_{0,e_i} \\ b^i_{1,0} & b^i_{0,e_i}\end{sbm} \in \GL_k(1|N)\Bigr\}.
\]
The topological Lie superalgebra of $\Aut O$ is denoted by
\[
 \Der_0 O  \ceq \Lie(\Aut O),
\]
which is linearly compact and has a topological basis 
$\{z^n \pd_z,z^n \pd_{\zeta^i} \mid n \ge 1, i \in [N]\} \cup 
 \{z^m \zeta^i \pd_z, z^m \zeta \pd_{\zeta^i} \mid m \ge 0,  i \in [N]\}$.
It is actually a Lie sub-superalgebra of $\Der O$ given in \eqref{eq:DerO}.

\begin{dfn}\label{dfn:Dsc}
Consider the even differential 1-form on $O$ 
\[
 \omega_Z \ceq d z + \sum_{i=1}^N \zeta^i d \zeta^i \in (\Omega^1_{A/R})_{\oo},
\]
where $d=d_{\ev}$ denotes the even universal differential (\cref{ntn:tdR}).
We define $\Asc O$ to be the subgroup of $\Aut O$ consisting of automorphisms 
preserving $\omega_Z$ up to multiplication by a function. 
Its element is called a \emph{superconformal transformation} of $\bbD$.
The topological Lie superalgebra of $\Asc O$ is denoted by $\Dsc O \ceq \Lie \Asc O$,
and its element is called a \emph{superconformal vector field}.
\end{dfn}

Let us recall basic properties of the Lie superalgebra $\Dsc O$, following 
\cite[\S 3.1.3]{H1} and \cite[Example 2.12]{HK}.
For $i \in [N]$, we define an odd derivation $D_Z^i$ of $O$ by
\begin{align}\label{eq:DZ}
 D_Z^i \ceq \pd_{\zeta} + \zeta \pd_z^i.
\end{align}
Then $\Dsc O$ is equal to the Lie subalgebra of $\Der_0 O$ 
\[
 \Dsc O = \{W \in \Der_0 O \mid [W,D_Z^i] = f D_Z^i, \ \exists f \in O\}.
\]
Also, $\Dsc O$ consists of the derivations of the form
\[
 D^f = f \pd_z+\frac{1}{2}(-1)^{p(f)}\tsum_{i=1}^N (D_Z^i f) D_Z^i
\]
with $f \in O$ of pure parity.
In particular, it contains 
\begin{align*}
&l_n \ceq -z^{n+1} \pd_z-\tfrac{n+1}{2}z^n \tsum_{i=1}^N \zeta^i \pd_{\zeta^i} 
& &(n \in \bbN), \\
&g_r^j \ceq -z^{r+\hf}(\pd_{\zeta^i}-\zeta^i \pd_z) + 
 \bigl(r+\thf\bigr)z^{r-\hf}\zeta^j \tsum_{i=1}^N \zeta^i \pd_{\zeta^i} 
& &(r \in \bbN+\thf, \, j \in [N]).
\end{align*}
In the case $N=1$, the elements $l_n$ ($n \in \bbN$) and $g_r \ceq g_r^1$ ($r \in \hf+\bbN$) 
are topological generators of $\Dsc O$, subject to the relations
\begin{align*}
 [l_m,l_n] = (m-n)l_{m+n}, \quad 
 [l_n,g_r] = (\tfrac{n}{2}-r) g_{n+r}, \quad 
 [g_r,g_s] = 2l_{r+s}.
\end{align*}
Thus, in the case $N=1$, $\Dsc O$ is a Lie sub-superalgebra of 
the \emph{Neveu-Schwarz algebra} of central charge $0$ (see \eqref{eq:NS}).


Recall that the superjet algebra $A^O$ enjoys a universal property: 
\[
 \Hom_{\SCom k}(A^O,k) \lsto \Hom_{\SCom k}(A,O).
\]
Now we would like to find the odd derivations $S_K^i$ on $A^O$ which corresponds to 
the action of $D_Z^i=\pd_{\zeta^i}+\zeta^i \pd_z$ on $O$. 
Let $\phi(Z)=\sum_{(n|J), \, n \ge 0}Z^{n|J}\phi_{n|J}$ be the series corresponding to 
an element $\phi \in \Hom_{\SCom k}(A,O)$. Then, we have
\[
 (D_Z^i \phi)(Z) = 
 \sum_{(n|J), \, n \ge 0, \, i \notin J} Z^{n|J} \sigma(e_i,J) \phi_{n|J \cup e_i}
+\sum_{(n|J), \, n \ge 0, \, i \in J} Z^{n|1} (n+1) 
 \sigma(e_i,J \sm e_i) \phi_{n+1|J \sm e_i},
\]
where we used the same symbols as in \cref{lem:sj-TS}.
This calculation indicates that the desired $S_K^i$ is given by 
\begin{align}\label{eq:csj:Sact}
 S_K^i(a^{[n|J]}) \ceq \begin{cases} \sigma(e_i,J) a^{[n|J \cup e_i]} & (i \notin J) \\
 (n+1) \sigma(e_i,J \sm e_i) a^{[n+1|J \sm e_i]} & (i \in J)\end{cases}.
\end{align}
Using this action twice, we see that the square $T_K=(S_K^i)^2$ is independent of $i$,
and its action is given by 
\begin{align}\label{eq:csj:Tact}
 T_K(a^{[n|J]}) = (n+1)a^{[n+1|J]}.
\end{align}
Then, we immediately have:

\begin{lem}\label{lem:STL}
The operators $S_K^i$'s and $T_K$ form a Lie superalgebra with commutation relation
\[
 [S_K^i,S_K^j] = 2\delta_{i,j}T_K, \quad [T_K,S_K^i]=0.
\]
\end{lem}

Recalling \cref{dfn:dsa}, we find that 
$(A^O,S_K^i)$ is a differential superalgebra for any $A \in \SCom k$.
Consulting \cref{dfn:Dsc}, we may name:

\begin{dfn}\label{dfn:scj}
For $A \in \SCom k$, the differential superalgebra $(A^O,S_K^i)$ is called 
the \emph{superconformal jet algebra} of $A$, and denoted by $A^{\Osc}$.
\end{dfn}

By \cref{lem:dnJ}, we have:

\begin{lem}\label{lem:scj-dsa}
For any $A \in \SCom k$, the superconformal jet algebra $A^{\Osc}$ 
is generated by $A$ as a differential superalgebra.
\end{lem}

Then we have the following analogous of \cref{lem:sj-gen}:

\begin{lem}\label{lem:scj-gen}
Let $A,B \in \SCom k$, and assume the following conditions.
\begin{itemize}[nosep]
\item $B$ contains $A$ as a linear sub-superalgebra.
\item $B$ has $N$ odd derivations $\sld^i$ ($i \in [N]$)
      which form the Lie superalgebra in \cref{lem:STL}.
\item The differential superalgebra $(B,\sld^i)$ is generated by $A$.
\end{itemize} 
Then the injection $A \inj B$ can be extended to a surjective morphism 
$(A^O,S_K^i) \srj (B,\sld^i)$ of differential superalgebras.
\end{lem}

\section{SUSY vertex algebras}\label{s:SUSY}

This section gives a recollection of SUSY vertex algebras introduced in \cite{HK}.
We also give a few preliminary observation on commutative SUSY vertex algebras.
We will work over a fixed field $k$ of characteristic $0$.
Also we fix a positive integer $N$, and denote the ordered set $\{1<\dotsb<N\}$ by $[N]$.

\subsection{Superfields}\label{ss:sf}

Here we explain the formalism of superfields used throughout the rest of this note.
The original reference is \cite[\S 2.6]{HK}.

Let us continue to use \cref{ntn:sd}. So $Z=(z,\zeta^1,\dotsc,\zeta^N)$ is 
a $1|N$-supervariable with even $z$ and odd $\zeta^i$'s, and denote 
\[
 Z^{j|J} \ceq z^j \zeta^J = z^j \zeta^{j_1}\zeta^{j_2}\dotsm \quad 
 (j \in \bbZ, \, J =\{j_1,j_2,\dotsc\} \subset [N]).
\]
We also use a simplified notation
\[
 Z^{j|N} \ceq Z^{j|[N]} = z^j \zeta^1 \dotsm \zeta^N.
\]
Below we will repeatedly use the sign notation given in \cite[\S 3.1.1]{HK}.
For disjoint ordered sets $I,J \subset [N]$, we define $\sigma(I,J) \in \{\pm 1\}$ by
\begin{align}\label{eq:sigma}
 \zeta^I \zeta^J = \sigma(I,J) \zeta^{I \cup J},
\end{align}
where $I \cup J$ denotes the re-ordered set.
If $I \cap J \neq \emptyset$, then we define $\sigma(I,J) \ceq 0$.
Also, for $i \in [N]$ and $J \subset [N]$, we use the following abbreviations.
\begin{align}\label{eq:sigma2}
 e_i \ceq \{i\}, \quad N \sm J \ceq [N] \sm J.
\end{align}
For example, we denote $\sigma(N \sm J,e_i) \ceq \sigma([N] \sm J,\{i\})$.

For a linear superspace $V$, we define $V\dbr{Z}$ and $V\dbr{Z^{\pm1}}$ 
to be the linear spaces of series with coefficients in $V$:
\[
 V\dbr{Z} \ceq \Biggl\{\sum_{n \in \bbN, \, J \subset [N]} Z^{n|J} v_{n|J} 
                       \ \Bigg| \ v_{n|J} \in V \Biggr\},  \quad
 V\dbr{Z^{\pm1}} \ceq \Biggl\{\sum_{n \in \bbZ, \, J \subset [N]} Z^{j|J} v_{j|J} 
                              \ \Bigg| \ v_{j|J} \in V \Biggr\}.
\]
Here the summations are possibly infinite, and the index $J$ runs over 
ordered subset in $[N]$.
These are linear superspaces by setting the parity as 
$p(Z^{j|J}v) \ceq \abs{J}+p(v) \bmod 2$ for $v \in V$ of pure parity. We also denote 
\[
 V\dpr{Z}\ceq 
 \Biggl\{\sum_{j \in \bbZ, \, J \subset [N]} Z^{j|J} v_{j|J} \in V\dbr{Z^{\pm1}} 
         \ \Bigg| \ v_{j|J}=0 \ \forall j \ll 0\Biggr\}.
\]

Let $\iEnd(V) \ceq \iHom(V,V)$ be the internal end in the category of linear superspaces
(see \eqref{eq:iHom}). In particular, it is a linear superspace. 
An \emph{$\iEnd(V)$-valued superfield} of $1|N$-supervariable $Z$ is a series of the form
\begin{align}\label{eq:(j|J)-sum}
 a(Z) = \sum_{(j|J)} Z^{-1-n|N \sm J} a_{(j|J)} \quad (a_{(j|J)} \in \iEnd(V)),
\end{align}
where the indices run over the range $j \in \bbZ$, $J \subset [N]$, 
and $N \sm J \ceq [N] \setminus J$, 
such that for each $v \in V$ we have $a(Z)v \in V\dpr{Z}$.
We denote the linear superspace of $\iEnd(V)$-valued superfields of supervariable $Z$ by 
\[
 \SF(V,Z) = \SF(V).
\]

Next, we recall the locality of superfields \cite[\S 3.1.2]{HK}.
Let us given two (super-)commuting $1|N$-supervariables $Z=(z,\zeta^1,\dotsc,\zeta^N)$ 
and $W=(w,\omega_1,\dotsc,\omega_N)$, and let $A$ be a superalgebra. 
An \emph{$A$-valued formal distribution in two variables} $Z$ and $W$ 
is a (possibly infinite) series of the form
\[
 a(Z,W) = \sum_{(i|I),(j|J)} Z^{i|I} W^{j|J} a_{i|I,j|J} \quad (a_{i|I,j|J} \in A)
\]
where the running index means $i,j\in\bbZ$ and $I,J \subset [N]$, and we wrote 
$Z^{i|I} W^{j|J} \ceq z^i \zeta^I w^j \omega^J = z^i w^j \zeta^I \omega^J$.
We denote by $A\dbr{Z^{\pm1},W^{\pm1}}$ 
the linear superspace of $A$-valued formal distributions.
A formal distribution $a(Z,W) \in A\dbr{Z^{\pm1},W^{\pm1}}$ is called \emph{local} 
if there exists $n \in \bbN$ such that 
\begin{align}\label{eq:loc}
 (z-w)^n a(Z,W) = 0.
\end{align}
Note that the left hand side is well-defined in $A\dbr{Z^{\pm1},W^{\pm1}}$.

\subsection{$N_W=N$ SUSY vertex algebras}\label{ss:NW}

In this subsection, we cite from \cite[\S 3]{HK} the notion of $N_W=N$ SUSY vertex algebras.
We continue to use the notation in \S \ref{ss:sf}. In particular, 
$Z=(z,\zeta^i)=(z,\zeta^1,\dotsc,\zeta^N)$, and $W=(w,\omega^i)$ are $1|N$-supervariables.

\subsubsection{Definition and basic properties}

\begin{dfn}[{\cite[Definition 3.3.1]{HK}}]\label{dfn:NW}
An \emph{$N_W=N$ SUSY vertex algebra} (\emph{$N_W=N$ SUSY VA} for short) 
is a data $(V,\vac,T,S^i,Y)$ consisting of 
\begin{itemize}[nosep]
\item a linear superspace $V=V_{\oz} \oplus V_{\oo}$, called the \emph{state superspace},
\item an even vector $\vac \in V_{\oz}$, called the \emph{vacuum},
\item an even operator $T \in \End(V)_{\oz}$,
\item odd operators $S^i \in \End(V)_{\oo}$ for $i \in [N]$, and 
\item an even linear map $Y(\cdot,Z)\colon V \to \SF(V,Z)$,
      called the \emph{state-superfield correspondence},
\end{itemize}
which satisfies the following axioms.
\begin{itemize}[nosep]
\item \emph{Vacuum axiom}: For any $a \in V$ and $i \in [N]$, we have 
\begin{align}\label{eq:NWvac}
 Y(a,Z)\vac=a+O(Z), \quad T\vac=S^i\vac=0,
\end{align}
where $O(Z)$ is an element of $V\dbr{Z}$ which vanishes at $Z=0$ (i.e., $z=\zeta^i=0$).
\item \emph{Translation invariance}: For each $a \in V$, we have
\begin{align}\label{eq:NWtrs}
 [T,Y(a,Z)] = \pd_z Y(a,Z), \quad 
 [S,Y(a,Z)] = \pd_{\zeta^i} Y(a,Z) \quad (i \in [N]),
\end{align}
where we denoted $\pd_z \ceq \pdd{z}$ and $\pd_{\zeta^i} \ceq \pdd{\zeta^i}$,
and used the supercommutator \eqref{eq:sbra}.
\item \emph{Locality axiom}: For any $a,b \in V$, there exists an $n \in \bbN$ such that 
\begin{align}\label{eq:NWloc}
 (z-w)^n [Y(a,Z),Y(b,W)]=0.
\end{align}
The identity is regarded as that of $(\End V)\dbr{Z^{\pm1},W^{\pm1}}$ (see \eqref{eq:loc}).
\end{itemize}
We abbreviate $V=(V,\vac,T,S^i,Y)$ if no confusion may arise.
\end{dfn}

For $a \in V$, $j \in \bbZ$ and an ordered subset $J \subset [N]$, 
we define the \emph{Fourier mode} or the \emph{$(j|J)$-operator}
$a_{(j|J)} \in \End(V)$ by the expansion 
\begin{align}\label{eq:(n|J)}
 Y(a,Z) = \sum_{(j|J)} Z^{-1-j|N \sm J} a_{(j|J)},
\end{align}
where we used similar symbols as in \eqref{eq:(j|J)-sum}.
We also denote $a_{(j|N)} \ceq a_{(j|[N])}$.

Using $(j|J)$-operators, the vacuum axiom \eqref{eq:NWvac} is equivalent to 
\begin{align}\label{eq:NWvac2}
 a_{(-1|N)}\vac = a, \quad a_{(n|J)}\vac = 0 \quad (n \in \bbN, \, J \subset [N]),
\end{align}
and by \cite[(3.3.1.5)]{HK}, the translation invariance \eqref{eq:NWtrs} is equivalent to 
\begin{align}\label{eq:NWrec}
 [T,a_{(j|J)}] = -j a_{(j-1|J)}, \quad 
 [S^i,a_{(j|J)}] = \begin{cases} \sigma(N \sm J,e_i) a_{(j|J \sm e_i)} & (i \in J) \\ 
                                 0 & (i \notin J) \end{cases},
\end{align}
where we used the sign $\sigma$ in \eqref{eq:sigma} and the abbreviations \eqref{eq:sigma2}.

As in the even case, the locality axiom \eqref{eq:NWloc} implies 
operator product expansion (OPE for short). To explain it, let us cite some notation 
on formal delta functions for $N_W=N$ SUSY VAs from \cite[\S 2.3]{HK}.
Let $Z=(z,\zeta^i)$ and $W=(w,\omega^i)$ be two commuting $1|N$-supervariables.
We denote by 
\begin{align}\label{eq:izw}
 i_{z,w}\colon k\dpr{Z,W} \linj k\dpr{Z}\dpr{W}, \quad  
 i_{w,z}\colon k\dpr{Z,W} \linj k\dpr{W}\dpr{Z}
\end{align}
the embeddings obtained by the expansions with respect to $\frac{w}{z}$ 
(in the domain $\abs{z}>\abs{w}$),
and with respect to $\frac{z}{w}$ (in the domain $\abs{z}<\abs{w}$), respectively.
For $j\in\bbZ$ and an ordered subset $J=\{j_1<\dotsb<j_r\} \subset [N]$, we set
\begin{align}\label{eq:Z-W}
 (Z-W)^{j|J} \ceq (z-w)^j (\zeta-\omega)^J 
 = (z-w)^j(\zeta^{j_1}-\omega^{j_1})\dotsm(\zeta^{j_r}-\omega^{j_r}).
\end{align}
The formal $\delta$-function for $N_W=N$ case is given by 
\begin{align}\label{eq:NWdelta}
 \delta(Z,W) \ceq (i_{z,w}-i_{w,z})(Z-W)^{-1|N} 
 = (i_{z,w}-i_{w,z})\frac{(\zeta-\omega)^N}{z-w}
 = (i_{z,w}-i_{w,z})\frac{(\zeta^1-\omega^1)\dotsm(\zeta^N-\omega^N)}{z-w}.
\end{align}
For $n \in \bbN$ and $J = \{j_1<\dotsb<j_r\} \subset [N]$, let
\begin{align}\label{eq:NWD}
 \pd_W^{n|J} \ceq 
 \pd_w^n \pd_{\omega^{j_1}} \dotsm \pd_{\omega^{j_r}}, \quad 
 \pd_W^{(n|J)} \ceq \frac{(-1)^{\binom{r+1}{2}}}{j!}\pd_W^{j|J}.
\end{align}
Then, by \cite[\S 3.1.4 (4)]{HK}, we have 
\begin{align}\label{eq:NWpddel}
 \pd_W^{(n|J)}\delta(Z,W) = (i_{z,w}-i_{w,z})(Z-W)^{-1-j|N \sm J}.
\end{align}
Finally, for $f(Z) = \sum_{(j|J)} Z^{j|J} f_{j|J} \in V\dbr{Z^{\pm1}}$, we denote
\begin{align}\label{eq:res}
 \res_Z f(Z) \ceq f_{-1|N}.
\end{align}
Then, by \cite[Lemma 3.1.3]{HK}, any local distribution $a(Z,W)$ 
is decomposed in the following finite sum.
\begin{align}\label{eq:NWexp}
 a(Z,W) = \sum_{(j|J), \, j \ge 0} \bigl(\pd_W^{(j|J)}\delta(Z,W)\bigr)c_{j|J}(W), \quad
 c_{j|J}(W) \ceq \res_Z (Z-W)^{j|J} a(Z,W).
\end{align}

Now, let $V$ be an $N_W=N$ SUSY VA, and $a, b \in V$ of pure parity.
Then the locality axiom \eqref{eq:NWloc} and the expansion \eqref{eq:NWexp} imply
the following OPE formula \cite[Theorem 3.3.8 (5)]{HK}:
\begin{align}\label{eq:NWope}
 [Y(a,Z),Y(b,W)] = 
 \sum_{(j|J), \, j \ge 0} \bigl(\pd_Z^{(j|J)}\delta(Z,W)\bigr) Y(a_{(j|J)}b,W).
\end{align}

As in the even case, an $N_W=N$ SUSY VA $V$ is called \emph{strongly generated} 
by a subset $\{a^i \mid i \in I\} \subset V$ if $V$ is spanned by the elements 
of the form $a^{i_1}_{(-n_1|J_1)} \dotsm a^{i_r}_{(-n_r|J_r)} \vac$
with $r \in \bbN$,  $n_i \in \bbZ_{\ge 1}$ and $J_i \subset [N]$.

We also have the notion of a \emph{module} over an $N_W=N$ SUSY VA $V$.
We omit the detailed definition, and only give the notation: Let $M$ be a $V$-module.
For $a \in V$ and  $m \in M$, we denote the $V$-action on $M$ by 
\begin{align}\label{eq:YM}
 Y^M(a,Z)m = \sum_{(j|J)} Z^{-1-j|N \sm J} a^M_{(j|J)} m \in M\dpr{Z}.
\end{align}

\subsubsection{Conformal $N_W=N$ SUSY vertex algebras}\label{sss:cNW}

We cite from \cite[\S 5]{HK} some examples of $N_W=N$ SUSY vertex algebras.
To explain those examples, we use the $\Lambda$-bracket introduced by \cite{HK},
which efficiently encodes OPE \eqref{eq:NWope}.

\begin{ntn}[{\cite[\S 3.2.1]{HK}}]\label{ntn:NWLam}
Let $V$ be an $N_W=N$ SUSY VA. 
\begin{itemize}[nosep]
\item 
We denote the commutative superalgebra freely generated by even $\lambda$ and 
odd $\chi^i$ ($i \in [N]$) as 
\[
 \scL_W = k[\Lambda] = k[\lambda,\chi^1,\dotsc,\chi^N].
\]

\item
For $m \in \bbN$ and $M =\{m_1<\dotsb<m_r\} \subset [N]$, we denote 
$\Lambda^{m|M} \ceq \lambda^m \chi^M = \lambda^m \chi^{m_1} \dotsm \chi^{m_r}$.

\item
Finally, for $a,b \in V$, we define $[a_\Lambda b] \in \scL_W \otimes_k V$ by 
\[
 [a_\Lambda b] = \sum_{(j|J), \, j \ge 0} \sigma(J,N \sm J)
 (-1)^{\binom{\abs{J}+1}{2}} \frac{1}{j!} \Lambda^{j|J} a_{(j|J)}b, \quad 
 \Lambda^{j|J} \ceq  \lambda^j \chi^{j_1} \dotsm \chi^{j_r}
\]
We call $[a_\Lambda b]$ the \emph{$\Lambda$-bracket}.
\end{itemize}
\end{ntn}

In the rest of this \cref{sss:cNW}, we assume the base field $k=\bbC$.

\begin{eg}[{\cite[Example 5.1]{HK}}]
$V(\scW_N)$ is an $N_W=N$ SUSY vertex algebra 
strongly generated by an element $L$ of parity $N \bmod 2$ 
and elements $Q^i$, $i \in [N]$, of parity $N+1 \bmod 2$ whose OPEs are
\begin{align}\label{eq:WNope}
 [L_\Lambda L] = (T+2\lambda) L, \quad 
 [Q^i_\Lambda Q^j] = (S^i+\chi^i)Q^j-\chi^j Q^i, \quad 
 [L_\Lambda Q^i] = (T+\lambda)Q^i+(-1)^N \chi^i L.
\end{align}
If $N=1,2$, then it admits a central extension. 
For $N=1$, it is given by 
\begin{align}\label{eq:WN1}
 [L_\Lambda L] = (T+2\lambda) L, \quad 
 [Q_\Lambda Q] =  S Q+\frac{\lambda \chi}{3}C, \quad 
 [L_\Lambda Q] = (T+\lambda)Q-\chi L+\frac{\lambda^2}{6}C. 
\end{align}
For $N=2$, we have a central extension 
\begin{align}\label{eq:WN2}
\begin{split}
&[L_\Lambda L] = (T+2\lambda) L, \quad 
 [L_\Lambda Q^i] = (T+\lambda)Q^i+\chi^i L \\
&[Q^i_\Lambda Q^j] = S^i Q^j, \quad 
 [Q^1_\Lambda Q^2] = (S^1+\chi^1)Q^2-\chi^2 Q^1+\frac{\lambda}{6}{C}.
\end{split}
\end{align}
\end{eg}

$V(\scW_N)$ is an $N_W=N$ SUSY analogue of Virasoro vertex algebra.
Recall that an even vertex algebra is called conformal if it has an element
which generates Virasoro vertex algebra. The following is its $N_W=N$ SUSY analogue.

\begin{dfn}[{First part of \cite[Definition 5.2]{HK}}]\label{dfn:cNW}
An $N_W=N$ SUSY VA $(V,\vac,T,S^i,Y)$ over $\bbC$ is called \emph{conformal} 
if it has elements $\nu,\tau^1,\dotsc,\tau^N$ satisfying the following conditions.
\begin{itemize}[nosep]
\item 
The superfields $L(Z) \ceq Y(\nu,Z)$ and $Q^i(Z) \ceq Y(\tau^i,Z)$ 
satisfy the OPEs \eqref{eq:WNope} (or the central extension \eqref{eq:WN1}, \eqref{eq:WN2}).

\item
We have $\nu_{(0|0)}=2T$ and $\tau^i_{(0|0)}=S^i$.

\item
The operator $\nu_{(1|0)}$ is semisimple and the eigenvalues $d$ are bounded below,
i.e., there is a finite subset $\{d_1,\dotsc,d_s\} \subset \bbC$ 
such that $d \in d_i+\bbR_{\ge 0}$ for some $i$.
\end{itemize}
We call the elements $\nu$ and $\tau^i$ the \emph{conformal elements}.
\end{dfn}

%

For later use, let us also cite:

\begin{dfn}[{Second part of \cite[Definition 5.2]{HK}}]\label{dfn:scNW}
An $N_W=N$ SUSY VA $V$ is called \emph{strongly conformal} if it is conformal and 
the operators $\nu_{(1|0)}$ and $\sum_{i=1}^N \sigma(e_i,N \sm e_i) \tau^i_{(0|e_i)}$
have integer eigenvalues.
We call the eigenvalue $\Delta$ of $\nu_{(1|0)}$ the \emph{conformal weight}.
\end{dfn}

As mentioned in \cite{HK} and extensively studied in \cite{H1},
for a strongly conformal $N_W=N$ SUSY VA $V$, we can construct a vector bundle 
on an arbitrary super Riemann surface whose fiber is $V$. 

\subsubsection{Identities in $N_W=N$ SUSY vertex algebras}

In this part, we explain basic identities valid in $N_W=N$ SUSY vertex algebras,
In particular, we give analogue for Borcherds' commutator formula (\cref{fct:NWcmt})
and iterate formula (\cref{lem:NWiter}), which will be used for the discussion of
Li filtration in \cref{ss:NWLif} and \cref{ss:NKLif}.

Throughout this subsection, $V=(V,\vac,T,S^i,Y)$ denotes an $N_W=N$ SUSY VA.
For a supervariable $Z=(z,\zeta^i)$, we denote
\[
 Z \nabla \ceq z T + \tsum_{i=1}^N \zeta^i S^i.
\]
We also use the formal exponential $e^X = \exp X \ceq \sum_{n \in \bbN}\frac{1}{n!}X^n$.

Using these relations and following the argument of the even (or non-super) case 
in \cite[3.1.1--3.1.7]{FBZ}, we can show the following statements.

\begin{fct}[{\cite[Proposition 3.3.6 (1), Theorem 3.3.8 (3), (4), Corollary 3.3.9]{HK}}]
\label{fct:NW-1}
For $a \in V$ and $i \in [N]$, we have the following.
\begin{enumerate}[nosep]
\item 
$Y(a,Z)\vac = e^{Z \nabla}a = e^{z T}(1+\zeta^1 S^1)\dotsm(1+\zeta^N S^N)a$.
\item
Using the supercommutator, we have 
\begin{align}\label{eq:lem-NW}
 Y(T a,Z) = \pd_z Y(a,Z) = [T,Y(a,Z)], \quad 
 Y(S^i a,Z) = \pd_{\zeta^i} Y(a,Z) = [S^i,Y(a,Z)].
\end{align}
In terms of the Fourier modes, we have
\begin{align}\label{eq:NW-TSa}
 (T a)_{(j|J)} = -j a_{(j-1|J)}, \quad 
 (S^i a)_{(j|J)} = \sigma(e_i,N \sm J) a_{(j|J \setminus e_i)}.
\end{align}
\end{enumerate}
\end{fct}

Next, we recall that operators $T$ and $S^i$ can be regarded 
as even and odd derivations, respectively.

\begin{fct}[{\cite[Corollary 3.3.9]{HK}}]\label{fct:NW:Sder} 
For $a,b \in V$ of pure parity, $j \in \bbZ$, $J \subset [N]$ and $i \in [N]$, we have 
\[
 T(a_{(j|J)}b) = (T a)_{(j|J)}b+a_{(j|J)}(T b), \quad 
 S^i(a_{(j|J)}b) = (-1)^{N \sm J}\bigl((S^i a)_{(j|J)}b+(-1)^{p(a)} a_{(j|J)}(S^i b)\bigr)
\]
with $(-1)^{N \sm J} \ceq (-1)^{N - \abs{J}}$.
\end{fct}

Let us also recall the skew-symmetry of the state-superfield correspondence.

\begin{fct}[{\cite[Proposition 3.3.12]{HK}}]\label{fct:NWskew}
For $a,b \in V$ of pure parity, we have 
\[
 Y(a,Z)b = (-1)^{p(a) p(b)} e^{Z \nabla}Y(b,-Z)a.
\]
\end{fct}

Now, recall Borcherds' commutator formula in the even case 
\cite[(2.1)]{Li}, \cite[\S 3.3.10, p.56]{FBZ}:
\begin{align}\label{eq:cmt-ev}
 [a_{(l)},b_{(m)}] = \sum_{j \in \bbN} \binom{l}{j} (a_{(j)}b)_{(l+m-j)}.
\end{align}
We have the following $N_W=N$ SUSY analogue.

\begin{fct}[{\cite[Proposition 3.3.18, (3.3.4.13)]{HK}}]\label{fct:NWcmt}
For $a,b \in V$, $l \in \bbZ$ and $L \subset [N]$, we have 
\[
 [a_{(l|L)},Y(b,W)] = \sum_{(j|J), \, j \ge 0} 
 (-1)^{\abs{J}\abs{L}+\abs{J}N+\abs{L}N} \pd_W^{(j|J)}W^{l|L}Y(a_{(j|J)},W)
\]
with $\pd_W^{(j|J)}$ given in \eqref{eq:NWD}.
In terms of Fourier modes, we have
\begin{align}\label{eq:NWcmt}
 [a_{(l|L)},b_{(m|M)}] = \sum_{\substack{(j|J), \\ j \ge 0, \, J \supset L \cap M}} 
 (-1)^\alpha \wt{\sigma} \cdot \binom{l}{j}
 (a_{(j|J)}b)_{(l+m-j|M \cup (L \sm J))},
\end{align}
where we denoted
\begin{align*}
 \alpha &\ceq (p(a)+N-\abs{L})(N-\abs{M}) + (\abs{L}-\abs{J})(n-\abs{J}), \\
 \wt{\sigma} &\ceq \sigma(J,N \sm J) \sigma(L,N \sm L) \sigma(J, L \sm J) 
                   \sigma(L \sm J,(N \sm M) \sm (L \sm J)).
\end{align*}
\end{fct}

Note that the $J=1$ part of the last formula coincides 
with the even case formula \eqref{eq:cmt-ev}.

Finally, let us discuss an $N_W=N$ SUSY analogue of Borcherds' iterate formula.
The even case is 
\begin{align}\label{eq:ev-iter}
 (a_{(k)} b)_{(l)}  = 
 \sum_{j \ge 0} (-1)^j \binom{k}{j}(a_{(k-j)}b_{(l+j)}-(-1)^k b_{(k+l-j)}a_{(j)})
\end{align}
for $a,b$ of an even vertex algebra and $k,l \in \bbZ$
(see \cite[(2.2)]{Li} and \cite[\S 3.3.10, p.56]{FBZ}).
Here and hereafter, we denote $a_{(k)}b_{(l)}u \ceq a_{(k)}(b_{(l)}u)$ for $u$ 
of a vertex algebra module, and in a similar way for the super case.

\begin{lem}\label{lem:NWiter}
Let $V$ be an $N_W=N$ SUSY VA and $M$ be a $V$-module.
For $a,b \in V$ of pure parity, $u \in M$, $k,l\in\bbZ$ and $K,L \subset [N]$, we have
\begin{align*}
 (a_{(k|K)}b)_{(l|L)}u = \sum_{\substack{(j|J), \\  j \ge 0, \, J \subset K}} 
 (-1)^{j+\alpha} \wt{\sigma} \cdot \binom{k}{j} \Bigl(
   (-1)^{(p(a)+N-\tabs{J}) \tabs{N \sm J'}} a_{(k-j|J)}b_{(l+j|J')}u \\
  -(-1)^{k+p(a)p(b)+(p(b)+N-\tabs{J})\tabs{N \sm J'}} 
    b_{(k+l-j|J')} a_{(j|J)}u\Bigr)
\end{align*}
with $J' \ceq L \cup (K \sm J)$ and 
\begin{align*}
 \alpha \ceq \abs{K \sm J}(1+\abs{N \sm J}), \quad
 \wt{\sigma} \ceq \sigma(J,K \sm J) \sigma(J,N \sm J) \sigma(L,K \sm J) \sigma(J',N \sm J').
\end{align*}
\end{lem}

\begin{proof}
Let $a, b \in V$ be elements of pure parity, and recall the OPE \eqref{eq:NWope}:
\[
 [Y(a,Z),Y(b,W)] = \sum_{(j|J), \, j \ge 0} 
 \bigl(\pd_Z^{j|J}\delta(Z,W)\bigr) Y(a_{(j|J)}b,W).
\]
Following the argument in the even case \cite[\S 3.3.10]{FBZ}, 
we multiply both sides of this equality by some 
$f(Z,W) \in k\bigl[Z^{1|J},W^{1|J},(Z-W)^{-1|J} \mid J \subset [N]\bigr]$.
Then, using \eqref{eq:NWpddel}, we have
\begin{align}\label{eq:NWiter}
\begin{split}
& \res_W \res_Z f(Z,W) a(Z) b(W) - (-1)^{p(a)p(b)} \res_W \res_Z f(Z,W) b(W) a(Z) \\
&=\res_W \res_{Z-W} \sum_{(j|J), \, j \ge 0} f(Z,W) (Z-W)^{-1-j|N \sm J}Y(a_{(j|J)}b,W).
\end{split}
\end{align}
To show \cref{lem:NWiter}, we set $f(Z,W) \ceq W^{l|L} (Z-W)^{k|K}$.
Let us calculate the first term in the left hand side. We have
\begin{align*}
  \res_Z (Z-W)^{k|K} a(Z) 
&=\res_Z (z-w)^k (\zeta-\omega)^K \sum_{(j|J)} z^{-1-j}\zeta^{N \sm J} a_{(j|J)} \\
&=\res_Z \Bigl[(z-w)^k 
  \cdot \sum_{M \subset K} (-1)^{\abs{K \sm M}} \sigma(M,K \sm M) \zeta^M \omega^{K \sm M} 
  \cdot \sum_{(j|J)} z^{-1-j}\zeta^{N \sm J} a_{(j|J)}\Bigr] \\
&=\res_Z \Bigl[\sum_{M \subset K} \sum_{(j|J)} (-1)^{\beta} \tau \cdot 
  (z-w)^k z^{-1-j} \zeta^{M \cup(N \sm J)} \omega^{K \sm M} a_{(j|J)}\Bigr] 
\end{align*}
with $\beta \ceq \abs{K \sm M}+\abs{K \sm M} \abs{N \sm J}$ and 
$\tau \ceq \sigma(M,K \sm M) \sigma(M,N \sm J)$.
By definition of $\res_Z$, the terms with $J = M$ survive, and 
\begin{align*}
&\res_Z (Z-W)^{k|K} a(Z) = \sum_{\substack{(j|J), \\ j \ge 0, \, J \subset K}} 
 (-1)^{\beta+k-j} \tau \cdot \binom{k}{j} w^{k-j} \omega^{K \sm J}a_{(j|J)}. 
\end{align*}
Then we have
\begin{align}
\nonumber
& \res_W \res_Z W^{l|L} (Z-W)^{k|K} a(Z) b(W) \\
\nonumber
&=\res_W \sum_{\substack{(j|J), \\ j \ge 0, \, J \subset K}}
 (-1)^{\beta+k-j} \tau \cdot \binom{k}{j} W^{l|L} w^{k-j} \omega^{K \sm J}a_{(j|J)}
 \sum_{(p|P)} W^{-1-p|N \sm P} b_{(p|P)} \\
\label{eq:NWiter:R1}
&=\sum_{\substack{(j|J), \\ j \ge 0, \, J \subset K}} 
  (-1)^{k-j+\beta} \upsilon \cdot \binom{k}{j} 
  (-1)^{(p(a)+N-\tabs{J}) \tabs{N \sm J'}} a_{(j|J)} b_{(l+k-j|J')}
\end{align}
with $J' \ceq L \cup (K \sm J)$ and 
\[
 \beta = \abs{K \sm J}+\abs{K \sm J} \abs{N \sm J} = \alpha, \quad 
 \upsilon \ceq \tau \cdot \sigma(L,K \sm J) \sigma(J',N \sm J') = \wt{\sigma}.
\]
Replacing $(j,k-j) \mapsto (k-j,j)$, we have $k-j+\alpha \mapsto j+\alpha$,
and \eqref{eq:NWiter:R1} is equal to the first term in the right hand side of the statement.
By a similar calculation, we see the second terms in \eqref{eq:NWiter} and the statement coincide.
For the right hand side of \eqref{eq:NWiter}, we have by \eqref{eq:NWpddel} that 
$\res_{Z-W} \pd_W^{(j|J)}\delta(Z,W) = \delta_{j,0} \delta_{J,\emptyset}$, which yields
\[
 \res_W \res_{Z-W} \sum_{(j|J)} W^{l|L}(Z-W)^{k|K} (Z-W)^{-1-j|1-J}Y(a_{(j|J)}b,W)
=(a_{(k|K)}b)_{(l|L)}.
\]
Hence we have the consequence.
\end{proof}

\subsection{$N_K=N$ SUSY vertex algebras}\label{ss:NK}

In this subsection, we recall the notion of $N_K=N$ SUSY vertex algebras 
introduced in \cite[\S 4]{HK}.
Throughout this subsection, we fix a positive integer $N$.

\subsubsection{Definition and basic properties}

\begin{ntn}[{\cite[4.1]{HK}}]
For a $1|N$-supervariable $Z=(z,\zeta^1,\dotsc,\zeta^N)$ and $i \in [N]$, 
we define odd derivations $D_Z^i$ and $\ol{D}_Z^i$ by 
\[
      D_Z^i \ceq \pd_{\zeta^i}+\zeta^i\pd_{z}, \quad 
 \ol{D}_Z^i \ceq \pd_{\zeta^i}-\zeta^i\pd_{z}.
\]
\end{ntn}

We have $(D_Z^i)^2=\pd_{z}$ and $(\ol{D}_Z^i)^2=-\pd_{z}$. 

\begin{dfn}[{\cite[Definition 4.13]{HK}}]\label{dfn:NK}
An \emph{$N_K=N$ SUSY vertex algebra} is a data $(V,\vac,S_K^i,Y)$ consisting of 
\begin{itemize}[nosep]
\item a linear superspace $V=V_{\oz} \oplus V_{\oo}$, called the \emph{state superspace},
\item an even element $\vac \in V_{\oz}$, called the \emph{vacuum}, 
\item $N$ odd operators $S_K^i \in \End(V)_{\oo}$ ($i \in [N]$),
\item an even linear map $Y(\cdot,Z)\colon V \to \SF(V,Z)$, 
      called the \emph{state-superfield correspondence},
\end{itemize}
which satisfies the following axioms.
\begin{enumerate}[nosep]
\item 
Vacuum axiom: $Y(a,Z)\vac=a+O(Z)$ for all $a \in V$, and $S\vac=0$.

\item 
Translation invariance: For all $a \in V$, we have 
\begin{align}\label{eq:NK-trs}
 [S_K^i,Y(a,Z)] = \ol{D}^i_Z Y(a,Z).
\end{align}
\item 
Locality axiom: For all $a,b \in V$, there is an $n \in \bbN$ such that
$(z-w)^n [Y(a,Z),Y(b,W)]=0$.
\end{enumerate}
We abbreviate $(V,\vac,S_K^i,Y)$ to $V$ if no confusion may arise.
Also, we often abbreviate the word ``vertex algebra'' to VA.
\end{dfn}


As in the $N_W=N$ case, we define the \emph{Fourier mode} or the \emph{$(j|J)$-operator}
$a_{(j|J)} \in \End(V)$ for $a \in V$, $j \in \bbZ$ and $J \subset [N]$ 
by the expansion \eqref{eq:(n|J)}.
Then the vacuum axiom is equivalent to the same formula as in \eqref{eq:NWvac2},
and by \cite[(4.6.1)]{HK}\footnote{There is a typo in the second half of \cite[(4.6.1)]{HK}.}, 
the translation invariance \eqref{eq:NK-trs} is equivalent to 
\begin{align}\label{eq:NKrecS}
 [S_K^i,a_{(j|J)}] = 
 \begin{cases} \sigma(N \sm J,e_i) a_{(j|J \sm e_i)} & (i \in J) \\
 -\sigma(N \sm (J \cup e_i),e_i) j a_{(j-1|J \cup e_i)} & (i \notin J) \end{cases}.
\end{align}
We have $[S_K^i,[S_K^i,a_{(j|J)}]] = -j a_{(j-1|J)}$ for every $i \in [N]$, and 
the super Jacobi identity implies that 
\[
 T \ceq \thf [S_K^i,S_K^i] = (S_K^i)^2 \in \iEnd(V)_{\oz}
\]
gives a well-defined even operator which satisfies 
\begin{align}\label{eq:NKrecT}
 [T,a_{(j|J)}] = -j a_{(j-1|J)} \quad (j \in \bbZ, \, J \subset [N]), 
 \quad \text{ i.e., } \ [T,Y(a,Z)] = \pd_z Y(a,Z).
\end{align}
The operators $S_K^i$'s and $T$ satisfy the commutation relation
\[
 [S_K^i,S_K^j] = 2\delta_{i,j}T, \quad [T,S_K^i]=0,
\]
which is the Lie superalgebra in \cref{lem:STL}.

Now let us explain the OPE formula for $N_K=N$ SUSY VAs, following \cite[\S 4.1, \S 4.2]{HK}. 
It involves a new delta function which is different from the $N_W=N$ case \eqref{eq:NWdelta}.
Let $Z=(z,\zeta^1,\dotsc,\zeta^N)$ and $W=(w,\omega^1,\dotsc,\omega^N)$ 
be two commuting $1|N$-supervariables, 
and let $i_{z,w}\colon k\dpr{Z,W} \inj k\dpr{Z}\dpr{W}$ and 
$i_{w,z}\colon k\dpr{Z,W} \inj k\dpr{W}\dpr{Z}$ be the embeddings in \eqref{eq:izw}.
For $j\in\bbZ$ and $J \subset [N]$, we set
\begin{align}\label{eq:NKZ-W}
 (Z-W)_K^{j|J} \ceq 
 \bigl(z-w-\tsum_{i=1}^N \zeta^i \omega^i \bigr)^j \bigl(\zeta-\omega\bigr)^J.
\end{align}
Then the formal $\delta$-function for the $N_K=N$ case is defined to be
\[
 \delta_K(Z,W) \ceq (i_{z,w}-i_{w,z})(Z-W)_K^{-1|N} 
 = (i_{z,w}-i_{w,z})\frac{\zeta-\omega}{z-w-\tsum_{i=1}^N \zeta^i \omega^i}
 = (i_{z,w}-i_{w,z})\frac{\zeta-\omega}{z-w},
\]
where in the last equality we used 
$(Z-W)_K^{-1|0}=\frac{1}{z-w-\sum_i \zeta^i \omega^i}=\frac{1+\sum_i \zeta^i \omega^i}{z-w}$.

The binomial \eqref{eq:NKZ-W} behaves with respect to the odd derivations 
$D_W^i \ceq \pd_{\zeta^i} + \zeta^i \pd_w$ in the same way as in the even case.
For $n \in \bbN$ and $J = \{j_1<\dotsb<j_r\} \subset [N]$, let
\begin{align}\label{eq:NKD}
 D_W^{n|J} \ceq \pd_w^n D_W^{j_1} \dotsm D_W^{j_r}, \quad 
 D_W^{(n|J)} \ceq (-1)^{\binom{r+1}{2}} \tfrac{1}{n!}D_W^{n|J}.
\end{align}
Then, by \cite[Lemma 4.1]{HK}, we have 
\begin{align}\label{eq:NKDdel}
 D_W^{(n|J)}\delta_K(Z,W) = (i_{z,w}-i_{w,z})(Z-W)_K^{-1-j|N \sm J}.
\end{align}
Also, by \cite[Lemma 4.4]{HK}, 
any local distribution $a(Z,W)$ is decomposed as 
\begin{align}\label{eq:NKexp}
 a(Z,W) = \sum_{(j|J), \, j \ge 0} \bigl(D_W^{(j|J)}\delta_K(Z,W)\bigr)c_{j|J}(W), \quad
 c_{j|J}(W) \ceq \res_Z (Z-W)_K^{j|J} a(Z,W),
\end{align}
where $\res_Z$ is defined in \eqref{eq:res}.

Now, by \cite[Theorem 4.16 (4)]{HK}, the OPE formula for an $N_K=N$ SUSY VA is 
\begin{align}\label{eq:NKope}
 [Y(a,Z),Y(b,W)] = 
 \sum_{(j|J), \, j \ge 0} \bigl(D_Z^{j|J}\delta_K(Z,W)\bigr) Y(a_{(j|J)}b,W).
\end{align}

We also have the notion of a \emph{module} $M$ over an $N_K=N$ SUSY VA $V$.
We will use a similar notation for the $V$-action on $M$ 
as in the $N_K=N$ case \eqref{eq:YM}.

\subsubsection{Conformal $N_K=N$ SUSY vertex algebras}

Here we give some examples of $N_K=N$ SUSY vertex algebras.
We use the following $\Lambda$-bracket to encode OPE \eqref{eq:NKope}, 
which is different from the one for $N_W=N$ case (\cref{ntn:NWLam}).

\begin{ntn}[{\cite[(3.3.4.6)]{HK}}]\label{ntn:NKLam}
Let $V$ be an $N_K=N$ SUSY VA. 
We define the \emph{$\Lambda$-bracket} $[a_\Lambda b]$ for $a, b \in V$ as follows.
\begin{itemize}
\item 
Let $\scL_K$ be the (non-commutative) superalgebra generated by even $\lambda$ and 
odd $\chi^i$ ($i \in [N]$) subject to the commutation relation 
\[
 [\lambda,\chi^i]=0, \quad  [\chi^i,\chi^j]=-2\lambda \delta_{i,j}.
\]

\item
For $m \in \bbN$ and $M =\{m_1<\dotsb<m_r\} \subset [N]$, we denote 
$\Lambda^{m|M} \ceq \lambda^m \chi^M = \lambda^m \chi^{m_1} \dotsm \chi^{m_r}$.

\item
Finally, for $a,b \in V$, we define $[a \Lambda b] \in \scL_K \otimes_k V$ by 
\[
 [a_\Lambda b] \ceq \sum_{(m|M), \, m \ge 0} 
  \sigma(M,N \sm M) (-1)^{\binom{\abs{M}+1}{2}} \frac{1}{m!} \Lambda^{m|M} a_{(m|M)}b.
\]
\end{itemize}
\end{ntn}

Let us start with a standard example of $N_K=1$ SUSY VA, 
which comes from an $N=1$ superconformal vertex superalgebra. 
In the $N_K=1$ case, a supervariable is denoted by $Z=(z,\zeta)$,
and a superfield is expressed as 
\begin{align*}
 Y(a,Z) &= \sum_{j \in \bbN, \, J=0,1} Z^{-1-j|1-J} a_{(j|J)}
         = \sum_{j \in \bbN, \, J=0,1} z^{-1-j}\zeta^{1-J} a_{(j|J)} \\
        &= \sum_{j \in \bbN} z^{-1-j} a_{(j|1)}+\sum_{j \in \bbN} z^{-1-j}\zeta a_{(j|0)}.
\end{align*}
Also, we denote the odd operator by $S_K$, and its square by $T$.

\begin{eg}\label{eg:NS}
Let $V$ be an $N=1$ superconformal vertex superalgebra in the sense of \cite[p.180]{K}.
Thus, it is a vertex superalgebra $(V,\vac,T,Y^{\tcl})$ over $\bbC$ together with 
an odd element $\tau \in V$ whose Fourier modes in 
$Y^{\tcl}(\tau,z) = \sum_{r \in \bbZ+\hf} z^{-r-\frac{3}{2}} G_r$ 
satisfy the following conditions.
\begin{enumerate}[nosep]
\item $G_{-\hf}\tau= 2\nu$ with $\nu$ a Virasoro element,
i.e., an even element $\nu \in V$ whose Fourier modes in 
$Y^{\tcl}(\nu,z)=\sum_{n \in \bbZ}z^{-n-2}L_n$ satisfy 
$L_{-1}=T$, $L_0\tau=\frac{3}{2}\tau$ and $L_0$ being diagonalizable.

\item
$G_{\frac{3}{2}}\tau=\frac{2}{3}c \vac$ for some $c \in \bbC$.

\item
$G_r \tau =0$ for $k>\frac{3}{2}$. 
\end{enumerate}
We call $\tau$ the \emph{Neveu-Schwarz element} of $V$ 
(which is called the $N=1$ superconformal element in \cite[Definition 5.9]{K}).  

The above Fourier modes form the \emph{Neveu-Schwarz algebra} of central charge $c$.
It is a Lie superalgebra with central extension 
which is generated by odd $G_r$ ($r \in \bbZ+\hf$) and even $L_n$ ($n \in \bbZ$)
subject to the following commutation relations.
\begin{align}\label{eq:NS}
\begin{split}
&[G_r,G_s] = 2L_{r+s}+\tfrac{1}{3}(r^2-\tfrac{1}{4})\delta_{r+s,0}c, \qquad 
 [G_r,L_n] = (r-\tfrac{n}{2})G_{n+r}, \\ 
&[L_m,L_n] = (m-n) L_{m+n} + \delta_{m+n,0}\tfrac{m^3-m}{12}c. 
\end{split}
\end{align}

Such $V$ has a structure $(\vac,S_K,Y)$ of $N_K=1$ SUSY VA.
The vacuum $\vac$ is the same one, the state-superfield correspondence $Y$ is given by 
\[
 Y(a,Z) \ceq Y^{\tcl}(a,z) + \zeta Y^{\tcl}(G_{-\hf}a,z)
\]
for $a \in V$, and the odd operator is given by $S_K \ceq G_{-\hf}=\tau_{(0|1)}$.
\end{eg}

\begin{eg}\label{eg:NS2}
As a special case in \cref{eg:NS}, we have the \emph{Neveu-Schwarz SUSY vertex algebra}
whose state superspace $V$ is the linear superspace of polynomials 
$\bbC[S_K^n \tau \mid n \in \bbN]$ and the vacuum is given by $\vac = 1$.
The state-superfield correspondence for the Neveu-Schwarz element $\tau$ is 
\[
 Y(\tau,Z) = \sum_{r \in \bbZ+\hf}z^{-r-\frac{3}{2}}G_r + 
            2\sum_{n \in \bbZ}z^{-n-2}\zeta L_n,
\]
and the Fourier modes 
\begin{align}\label{eq:tau_op}
 \tau_{(j|1)} = G_{j-\hf}, \quad \tau_{(j|0)} = 2 L_{j-1} \quad (j \in \bbZ)
\end{align}
satisfy the commutation relations \eqref{eq:NS}.
We call $\nu \ceq \frac{1}{2}S_K \tau \in V$ the Virasoro element of $V$.
Using the $\Lambda$-bracket (\cref{ntn:NKLam}), the OPE of $Y(\tau,Z)$ is summarized as 
\[
 [\tau_\Lambda \tau] = (2 T+\chi S_K+3\lambda)\tau+\frac{c}{3} \lambda^2 \chi \vac.
\]
For later reference, we give some formulas.
\begin{align}\label{eq:tau-nu}
 \tau = \tau_{(-1|1)}\vac = G_{-\frac{3}{2}}\vac, \quad 
  \nu = \thf (S_K \tau)_{(-1|1)}\vac = \thf \tau_{(-1|0)}\vac = L_{-2}\vac.
\end{align}
In the latter calculation, we used the recursion \eqref{eq:NK-TSa} given below.
\end{eg}

The above \cref{eg:NS} can be extended to higher $N_K$ cases.

\begin{dfn}[{First part of \cite[Definition 5.6]{HK}}]\label{dfn:cNK}
For $N \in \{1,2,3,4\}$, an $N_K=N$ SUSY VA $(V,\vac,S_K^i,Y)$ is called 
\emph{conformal} if it has an element $\tau \in V$ of parity $N \bmod 2$, 
satisfying the following conditions.
\begin{itemize}[nosep]
\item 
Using \cref{ntn:NKLam} of the $\Lambda$-bracket, the OPE of $Y(\tau,Z)$ is given by
\[
 [\tau_\Lambda \tau] = \bigl(2t+(4-N)\lambda+\tsum_{i=1}^N \chi^i S_K^i\bigr)\tau
 + \begin{cases}
    \frac{c}{3}\lambda^{3-N}\chi^N \vac & (N \le 3) \\ \lambda c \vac & (N=4)
   \end{cases}.
\] 
\item
Using \eqref{eq:sigma2}, we have 
\[
 \tau_{(0|0)}=2T, \quad \tau_{(0|e_i)}=\sigma(N \sm e_i,e_i) S_K^i.
\]

\item
The operator $\tau_{(1|0)}$ is semisimple and the eigenvalues $d$ are bounded below,
i.e., there is a finite subset $\{d_1,\dotsc,d_s\} \subset \bbC$ 
such that $d \in d_i+\bbR_{\ge 0}$ for some $i$.
\end{itemize}
We call $\tau$ the \emph{conformal element}.
\end{dfn}

\begin{rmk*}
The condition $N \le 4$ comes from the fact that there is no central extension 
of the Lie superalgebra corresponding to the $\Lambda$-bracket
$[\tau_\Lambda \tau] = (2t+(4-N)\lambda+\tsum_{i=1}^N \chi^i S_K^i)\tau$.
\end{rmk*}

For later use, let us also cite:

\begin{dfn}[{Second part of \cite[Definition 5.6]{HK}}]\label{dfn:scNK}
An $N_K=N$ SUSY VA $V$ is called \emph{strongly conformal}
if it is conformal and satisfies the following conditions.
\begin{itemize}[nosep]
\item $\tau_{(1|0)}$ has integer eigenvalues.
\item If $N=2$, then the operator $\sqrt{-1}\tau_{(0|N)}$ has integer eigenvalues.
\end{itemize}
\end{dfn}

As in the $N_W=N$ case (\cref{dfn:scNW}), the strongly conformal condition guarantees 
us to construct a vector bundle on an arbitrary super Riemann surface whose fiber is $V$. 
See \cite{H1} for the detail.

\subsubsection{Identities in $N_K=N$ SUSY vertex algebras}

We explain basic identities for $N_K=N$ case..
Throughout this part, $V=(V,\vac,S_K^i,Y)$ denotes an $N_K=N$ SUSY VA, and 
\[
 Z \nabla \ceq z T + \tsum_{i=1}^N \zeta^i S_K^i
\]
for the operators $T=(S_K^i)^2$ and $S_K^i$ of $V$.

The following is an $N_K=N$ analogue of \cref{fct:NW-1}.

\begin{fct}[{\cite[Proposition 4.15 (1), Theorem 4.16 (3), Corollary 4.18]{HK}}]
For $a \in V$, we have: 
\begin{enumerate}[nosep]
\item 
$Y(a,Z)\vac = e^{Z \nabla}a
 = e^{z T}(1+\zeta^1 S_K^1) \dotsm (1+\zeta^N S_K^N)a$.
\item
$Y(T a,Z)=\pd_z Y(a,Z)$ and $Y(S_K^i a,Z)=D_Z^i Y(a,Z)=(\pd_{\zeta^i}+\zeta^i \pd_z)Y(a,Z)$.
In terms of the Fourier modes, we have
\begin{align}\label{eq:NK-TSa}
 (T a)_{(j|J)} = -j a_{(j-1|J)}, \quad 
 (S_K^i a)_{(j|J)} = 
 \begin{cases} \sigma(N \sm J,e_i)a_{(j|J \setminus e_i)} & (i \in J) \\
 -j \sigma(N \sm (J \cup e_i),e_i)a_{(j-1|J \cup e_i)} & (i \notin J) \end{cases}.
\end{align}
\end{enumerate}
\end{fct}

\begin{rmk}
Some comments are in order.
\begin{enumerate}[nosep]
\item 
As noted in \cite[Remark 4.17]{HK} and \cite[Remark 3.3]{H2}, we have 
\[
 Y(T a,Z)=\pd_z Y(a,Z)=[T,Y(a,Z)], \quad 
 Y(S_K^i a,Z) = D_Z^i Y(a,Z) \neq \ol{D}_Z^i Y(a,Z) =[S_K^i,Y(a,Z)].
\]
The inequality $Y(S_K^i a,Z) \neq [S_K^i,Y(a,Z)]$ should be 
contrasted with the $N_W=N$ case \eqref{eq:lem-NW}.

\item
By the equivalent form \eqref{eq:NWvac2} of the vacuum axiom, 
we can recover the odd operators $S_K^i$ from $Y$ in the way
\begin{align}\label{eq:NK:S}
 S_K^i a = (S_K^i a)_{(-1|N)} \! \vac = a_{(-1|N \sm e_i)} \! \vac.
\end{align}
Thus the operators $S_K^i$ can be omitted from \cref{dfn:NK}.
Similarly, $T$ is recovered as 
\[
 T a = (T a)_{(-1|N)} \! \vac = a_{(-2|N)} \! \vac,
\]
which is a reminiscent of the even case formula $T a = a_{(-2)} \! \vac$.
\end{enumerate}
\end{rmk}

Next we recall that the odd operator $S_K^i$ can be regarded as an derivation.

\begin{fct}[{\cite[Corollary 4.18]{HK}}]\label{fct:Sder} 
For $a,b \in V$ of pure parity and $i \in [N]$, we have 
\[
 S_K^i(a_{(j|J)}b) = 
 (-1)^{N \sm J}\bigl((S_K^i a)_{(j|J)}b+(-1)^{p(a)} a_{(j|J)}(S_K^i b)\bigr)
\]
with $(-1)^{N \sm J} \ceq (-1)^{N - \abs{J}}$.
\end{fct}

Let us also recall the skew-symmetry of the state-superfield correspondence.

\begin{fct}[{\cite[Proposition 3.3.12, \S 4.8]{HK}}]\label{fct:skew}
For $a,b \in V$ of pure parity, we have 
\[
 Y(a,Z)b = (-1)^{p(a) p(b)} e^{Z \nabla}Y(b,-Z)a.
\]
\end{fct}

Next, we have the following $N_K=N$ SUSY analogue of Borcherds' commutator formula
$ [a_{(l)},b_{(m)}] = \sum_{j \in \bbN} \binom{l}{j} (a_{(j)}b)_{(l+m-j)}$ 
in \eqref{eq:cmt-ev}.

\begin{fct}[{\cite[Proposition 3.3.18, (4.3.3), \S 4.10]{HK}}]\label{fct:NKcmt}
For $a,b \in V$, $l \in \bbZ$ and $L \subset [N]$, we have 
\[
 [a_{(l|L)},Y(b,W)] = \sum_{(j|J), \, j \ge 0} 
 (-1)^{\abs{J}\abs{L}+\abs{J}N+\abs{L}N} D_W^{(j|J)}W^{l|L}Y(a_{(j|J)},W),
\]
where $D_W^{(j|J)}$ is given in \eqref{eq:NKD}.
In terms of Fourier modes, we have
\begin{align}\label{eq:NKcmt}
 [a_{(l|L)},b_{(m|M)}] = \sum_{\substack{(j|J), \, j \ge 0, \\ M \cap (J \Delta L)=\emptyset}} 
 (-1)^{\alpha+\beta} \wt{\sigma} \cdot \frac{\df{l}{j+\#(J \sm L)}}{j!} 
 (a_{(j|J)}b)_{(l+m-j-\#(J \sm L)|M \cup (J \Delta L))},
\end{align}
where we denoted $\df{l}{n} \ceq l(l-1)\dotsm(l-n+1)$ for $n \in \bbN$,  
$J \Delta L \ceq (J \sm L) \cup (L \sm J)$ and 
\begin{align*}
 \alpha &\ceq  (p(a)+N-\abs{L})(N-\abs{M}), \quad 
   \beta \ceq \abs{J}(N-\abs{L})+\abs{L} N+\binom{\abs{J \cap L}}{2}+\binom{\abs{J}+1}{2}, \\
 \wt{\sigma} &\ceq
 \sigma\bigl(J \Delta L, N \sm (M \cup (J \Delta L))\bigr)
 \sigma(J,N \sm J) \sigma(L,N \sm L)
 \sigma(J \sm  L,J \cap L) \sigma(J \cap L,L \sm  J)  \sigma(J \sm  L,L \sm  J).
\end{align*}
\end{fct}

In the $N_K=1$ case, \eqref{eq:NKcmt} reduces to 
\begin{align*}
&[a_{(l|0)},b_{(m|0)}] = \sum_{j \in \bbN, \, J=0,1} (-1)^{\binom{J}{2}} 
 \frac{\df{l}{(j+J)}}{j!} (a_{(j|J)}b)_{(l+m-j-J|J)} \times (-1)^{p(a)+1}, \\
&[a_{(l|0)},b_{(m|1)}] = \sum_{j \in \bbN, \, J=0,1} (-1)^{\binom{J}{2}} 
 \frac{\df{l}{(j+J)}}{j!} (a_{(j|J)}b)_{(l+m-j-J|1)}, \\
&[a_{(l|1)},b_{(m|0)}] = \sum_{j \in \bbN, \, J=0,1} (-1)^{J+1} 
 \binom{l}{j} (a_{(j|J)}b)_{(l+m-j|1-J)} \times (-1)^{p(a)}, \\
&[a_{(l|1)},b_{(m|1)}] = \sum_{j \in \bbN, \, J=0,1} (-1)^{J+1} 
 \binom{l}{j} (a_{(j|J)}b)_{(l+m-j|1)}
\end{align*}
for $a,b \in V$ of pure parity and $m,n \in \bbZ$.
Note that the $J=1$ part of $[a_{(l|1)},b_{(m|1)}]$ coincides 
with the even case formula \eqref{eq:cmt-ev}.

Finally, let us discuss an $N_K=N$ SUSY analogue of Borcherds' iterate formula.
Recall the even case formula \eqref{eq:ev-iter}:
$(a_{(k)} b)_{(l)}  = 
 \sum_{j \ge 0} (-1)^j \binom{k}{j}(a_{(k-j)}b_{(l+j)}-(-1)^k b_{(k+l-j)}a_{(j)})$.

\begin{lem}\label{lem:NKiter}
Let $V$ be an $N_K=N$ SUSY VA, and $W$ be a $V$-module.
For $a,b \in V$ of pure parity, $w \in W$, $k,l\in\bbZ$ and $K,L \subset [N]$, we have
\begin{align*}
 (a_{(k|K)}b)_{(l|L)}m = \sum_{(j|J), \, j \ge 0} \sum_{M \subset J \cap K} 
 &(-1)^{j+\alpha} \wt{\sigma} \cdot \binom{k}{j,\abs{J \sm M}} \\
 \Bigl(
 &(-1)^{(p(a)+N-\tabs{J}) \tabs{N \sm L'}} a_{(k-j-\#(J \sm M)|J)}b_{(l+j|L')}w \\
-&(-1)^{k+p(a)p(b)+(p(b)+N-\tabs{J})\tabs{N \sm L'}} 
   b_{(l+k-j-\#(J \sm M)|L')} a_{(j|J)}w \Bigr),
\end{align*}
where $\binom{k}{i,j} \ceq \binom{k}{i}\binom{k-i}{j}$, 
$L' \ceq L \cup (J \sm M) \cup (K \sm M)$ and 
\begin{align*}
&\alpha \ceq \binom{\abs{J \sm M}+1}{2}+
 \abs{J \sm M}(\abs{M}+\abs{N \sm J})+\abs{K \sm M}(1+\abs{N \sm J})
 +\abs{N \sm J} \abs{N \sm L'}, \\
&\wt{\sigma} \ceq 
 \sigma(M, K \sm M) \sigma(J \sm M, M) \sigma (J, N \sm J) \sigma(J \sm M,K \sm M)
 \sigma(L,(J \sm M) \cup (K \sm M)) \sigma(L', N \sm L').
\end{align*}
\end{lem}

\begin{proof}
The strategy is the same as in \cref{lem:NWiter}.
Recall the OPE for the $N_K=N$ case \eqref{eq:NKope}:
\[
 [Y(a,Z),Y(b,W)] = \sum_{(j|J), \, j \ge 0} \bigl(D_Z^{j|J}\delta(Z,W)\bigr) Y(a_{(j|J)}b,W).
\]
For each $f \in k[Z^{1|J},W^{1|J},(Z-W)_K^{-1|J} \mid J \subset [N]]$,
it yields the equality 
\begin{align}\label{eq:NKiter}
\begin{split}
& \res_W \res_Z f(Z,W) a(Z) b(W) - (-1)^{p(a)p(b)} \res_W \res_Z f(Z,W) b(W) a(Z) \\
&=\res_W \res_{(Z-W)_K} \sum_{(j|J), \, j \ge 0} f(Z,W) (Z-W)^{-1-j|N \sm J}Y(a_{(j|J)}b,W),
\end{split}
\end{align}
where $\res_{(Z-W)_K}$ means $\res_{(z-w-\sum_{j=1}^N \zeta^j \omega^j,\zeta^i-\omega^i)}$.
We apply this equality for
\[
 f(Z,W) \ceq W^{l|L} (Z-W)_K^{k|K} 
 = W^{l|L} (z-w-\tsum_{i=1}^N \zeta^j \omega^j)^k (\zeta-\omega)^K.
\]
The first term in the left hand side of \eqref{eq:NKiter} is 
\begin{align*}
  \res_Z (Z-W)_K^{k|K} a(Z) 
&=\res_Z (z-w-\tsum_{i=1}^N \zeta^i \omega^i)^k (\zeta-\omega)^K 
  \sum_{(j|J)} z^{-1-j}\zeta^{N \sm J} a_{(j|J)} \\
&=\res_Z \Bigl[\sum_{I \subset [N]} (-1)^{\binom{\abs{I}+1}{2}} \binom{k}{\abs{I}} 
  (z-w)^{k-\abs{I}} \zeta^I \omega^I \\
&\hspace{3em}
  \cdot \sum_{M \subset K} (-1)^{\abs{K \sm M}} \sigma(M,K \sm M) \zeta^M \omega^{K \sm M} 
  \cdot \sum_{(j|J)} z^{-1-j}\zeta^{N \sm J} a_{(j|J)}\Bigr] \\
&=\res_Z \Bigl[\sum_{I \subset [N]} \sum_{M \subset K} \sum_{(j|J)} 
  (-1)^{\beta} \tau \binom{k}{\abs{I}} 
  (z-w)^{k-\abs{I}} z^{-1-j} \zeta^{I \cup M \cup(N \sm J)} \omega^{I \cup (K \sm M)} 
  a_{(j|J)}\Bigr] 
\end{align*}
with $\beta \ceq \binom{\abs{I}+1}{2}+\abs{K \sm M}
+\abs{I}\abs{M}+(\abs{I}+\abs{K \sm M})\abs{N \sm J}$ and 
$\tau \ceq \sigma(M,K \sm M) \sigma(I,M) \sigma(I \cup M,N \sm J) \sigma(I,K \sm M)$.
By definition of $\res_Z$, the terms with $J = I \sqcup M$ survive, and 
\begin{align*}
&\res_Z (Z-W)_K^{k|K} a(Z) = \sum_{(j|J)} \sum_{M \subset J \cap K} 
 (-1)^{\gamma} \tau \binom{k}{\abs{J \sm M}} \binom{k - \abs{J \sm M}}{j} 
 w^{k-\abs{J \sm M}-j} \omega^{(J \sm M)\cup(K \sm M)}a_{(j|J)}, \\
&\gamma \ceq \beta+k-\abs{J \sm M}-j
        = k-j+\binom{\abs{J \sm M}}{2}+\abs{J \sm M}(\abs{M}+\abs{N \sm J})+
          \abs{K \sm M}(1+\abs{N \sm J}), \\
&\tau = \sigma(M,K \sm M)\sigma(J \sm M,M)\sigma(J,N \sm J)\sigma(J \sm M,K \sm M).
\end{align*}
Then we have
\begin{align}
\nonumber
& \res_W \res_Z W^{l|L} (Z-W)_K^{k|K} a(Z) b(W) \\
\nonumber
&=\res_W \sum_{(j|J)} \sum_{M \subset J \cap K} 
  (-1)^{\gamma} \tau \binom{k}{j,\abs{J \sm M}} W^{l|L} w^{k-j-\abs{J \sm M}} 
  \omega^{(J \sm M)\cup(K \sm M)} a_{(j|J)} \sum_{(p|P)} W^{-1-p|N \sm P} b_{(p|P)} \\
\label{eq:NKiter:R1}
&=\sum_{(j|J)} \sum_{M \subset J \cap K} (-1)^{\gamma+\abs{N \sm J} \tabs{N \sm L'}} 
  \upsilon \binom{k}{j,\abs{J \sm M}} (-1)^{(p(a)+N-\tabs{J}) \tabs{N \sm L'}}
  a_{(j|J)} b_{(l+k-j-\#(J \sm M)|L')}, \\
\nonumber
& L' \ceq L \cup (J \sm M) \cup (K \sm M), \quad 
  \upsilon \ceq \tau \cdot \sigma(L,(J \sm M) \cup (K \sm M)) \sigma(L',N \sm L') = \wt{\sigma}.
\end{align}
Replacing $k-j-\abs{J \sm M}$ by $j$ and vice-a-versa, 
we have $\gamma+\abs{N \sm J} \abs{N \sm L'} \mapsto j+\alpha$,
and \eqref{eq:NKiter:R1} is equal to the first term in the right hand side of the statement.
By a similar calculation, we see the second term in \eqref{eq:NKiter} and the statement coincide.
For the right hand side of \eqref{eq:NKiter}, we have by \eqref{eq:NKDdel} that 
$\res_{(Z-W)_K} D_W^{(j|J)}\delta(Z,W) = \delta_{j,0} \delta_{J,\emptyset}$, which yields
\[
 \res_W \res_{(Z-W)_K} \sum_{(j|J)} W^{l|L}(Z-W)_K^{k|K} (Z-W)_K^{-1-j|1-J}Y(a_{(j|J)}b,W)
=(a_{(k|K)}b)_{(l|L)}.
\]
Hence we have the consequence.
\end{proof}

\subsection{Commutative SUSY vertex algebras}\label{ss:com}

Hereafter, the word ``SUSY vertex algebra" means an $N_W=N$ or $N_K=N$ SUSY vertex algebra.

\subsubsection{SUSY analogue of Borcherds' equivalence}

Recall that, in the even case, a \emph{commutative} vertex algebra is equivalent to 
a unital commutative algebra with derivation \cite[\S 4]{B}, \cite[\S 1.4]{FBZ}.
The following \cref{lem:comW,lem:comK} are analogue of this fact for 
$N_W=N$ and $N_K=N$ SUSY vertex algebras, respectively.

\begin{dfn}
A SUSY VA $V$ is \emph{commutative} 
if we have $n=0$ for any $a,b \in V$ in the locality axiom, i.e.,
$[Y(a,Z),Y(b,W)]=0$.
\end{dfn}

Similarly as in the even case, we have the following restatement:

\begin{lem}\label{lem:com}
A SUSY VA $V$ is commutative if and only if 
$Y(a,Z) \in (\End V)\dbr{Z}$ for any $a \in V$.
\end{lem}

\begin{proof}
The argument is quite similar to that for the even case \cite[\S 1.4]{FBZ},
but let us write down it for completeness.
Assume that $V$ is commutative. Then, for any $a,b \in V$ of pure parity, we have 
\begin{align}\label{eq:scv}
 Y(a,Z)b = \rst{\vphantom{1^{(}}Y(a,Z)Y(b,W)\vac}{W=0} 
         = \rst{(-1)^{p(a)p(b)}Y(b,W)Y(a,Z)\vac}{W=0}.
\end{align}
The right hand side has no negative powers of $z$ by the vacuum axiom, 
and hence $Y(a,Z)b \in V\dbr{Z}$. Thus we have $Y(a,Z) \in (\End V)\dbr{Z}$.

Conversely, assume $Y(a,Z) \in (\End V)\dbr{Z}$ for any $a \in V$.
Then, for any $a,b \in V$ of pure parity, both of the series 
$f_1(Z,W) \ceq Y(a,Z)Y(b,W)$ and $f_2(Z,W) \ceq Y(b,W)Y(a,Z)$ are in $(\End V)\dbr{Z,W}$.
Now the locality axiom says that there exists $n \in \bbN$ such that 
$(z-w)^n f_1(Z,W) = (z-w)^n (-1)^{p(a)p(b)} f_2(Z,W)$.
The last equality yields $f_1(Z,W)=(-1)^{p(a)p(b)} f_2(Z,W)$,
and thus $V$ is commutative.
\end{proof}

\begin{lem}\label{lem:comW}
Let $(V,\vac,T,S^i,Y)$ be a commutative $N_W=N$ SUSY VA.
Then $V$ is a commutative $k$-superalgebra with multiplication 
$a \cdot b \ceq a_{(-1|N)} b$, unit $\vac$, 
even derivation $T$, and odd derivations $S^i$ (see \cref{dfn:der}).
Conversely, given a commutative $k$-superalgebra $V$ with multiplication $\cdot$, 
unit $1$, even and odd derivations $T$ and $S^i$, 
we have a commutative $N_W=1$ SUSY VA $(V,1,T,S^i,Y)$ with 
\begin{align}\label{eq:NW-scom}
 Y(a,Z) \ceq e^{Z \nabla} a = \Biggr(\sum_{n \ge 0} \frac{1}{n!} z^n T^n\Biggr) 
 (1+\zeta^1 S^1) \dotsm (1+\zeta^N S^N) (a \cdot ?) \in (\End V)\dbr{Z},
\end{align}
where the symbol $a \cdot ?$ denotes the multiplication operator. 
\end{lem}

\begin{proof}
This one is also a simple analogue of the even case \cite[\S 1.4]{FBZ}.
First, assume that $V$ is a commutative $N_W=1$ SUSY VA,
and let $a,b \in V$ be of pure parity.
Then \eqref{eq:scv} yields $Y(a,Z)b=(-1)^{p(a)p(b)} b_{(-1|1)}a+O(Z,W)$, 
where $O(Z,W)$ is an element vanishing at $Z=W=0$.
Taking the coefficient of $Z^{0|0}$, we have $a_{(-1|1)}b=(-1)^{p(a)p(b)}b_{(-1|1)}a$.
Then, denoting $Y_a := a_{(-1|1)}$, we have $Y_a Y_b = (-1)^{p(a)p(b)}Y_b Y_a$.
By the argument in \cite[\S 1.3.3]{FBZ}, it shows the associativity and 
the commutativity of $a \cdot b \ceq Y_a b$.
The vacuum axiom yields that $\vac$ is the unit of this multiplication $\cdot$,
and \cref{fct:NW:Sder} means that $T$ and $S$ are an even and odd derivation, respectively.
The converse statement is checked directly from the formula \eqref{eq:NW-scom}, 
and we omit the detail.
\end{proof}

\begin{lem}\label{lem:comK}
Let $V$ be a $k$-linear superspace.
Then the structure $(V,\vac,S^i,Y)$ of commutative $N_K=N$ SUSY VA is equivalent to the 
structure $(V,\cdot,1,S^i)$ of unital commutative $k$-superalgebra with 
odd derivations $S^i$ ($i \in [N]$) which satisfy the commutation relation
\[
 [S^i,S^j] = 2\delta_{i,j}T, \quad [T,S^i]=0.
\]
The correspondence is given by $a_{(-1|N)} b = a \cdot b$, $\vac = 1$ and 
\begin{align}\label{eq:NK-scom}
 Y(a,Z) = e^{Z \nabla}a = e^{z T}(1+\zeta^1 S^1)\dotsm(1+\zeta^N S^N)a.
\end{align}
\end{lem}

\begin{proof}
The argument in \cref{lem:comW} works, using \cref{lem:STL} and \cref{fct:Sder}.
\end{proof}

\subsubsection{$N_W=N$ SUSY VA structure of superjet algebras}

Recall that in \S \ref{ss:arc} we introduced the superjet algebra 
\[
 A^O = A^{k\dbr{Z}} = \clS^N(A_\infty)
\]
for each $A \in \SCom$ (\cref{prp:sj}).
Here $O=k\dbr{Z}$ denotes the topological superalgebra 
$O^{1|N}=k\dbr{z,\zeta^1,\dotsc,\zeta^N}$.
We argue that $A^O$ has a natural structure of $N_W=N$ SUSY VA (\cref{prp:NWsj}).

Let us recall some basic properties of $A^O$.
It is a commutative $k$-superalgebra equipped with 
the universal differentials $d_{n|J}=d_{\txdR}^J d_n$ for $n \in \bbN$ and $J \subset [N]$.
Since we are assuming $k$ is of characteristic $0$, 
we have $d_n = \frac{1}{n!}d_1^n$ by \cref{rmk:HS}. 
Note that $d_1$ and $d_{\txdR}$ is an even and odd derivation, respectively.
The even derivation $d_1$ comes from the Hasse-Schmidt derivations 
$A_\infty=\HS{A/k}{\infty}$, and the odd derivation $d_{\txdR}$ 
is the differential of the de Rham complex.
Also, recall that $A^O$ is generated by $a^{[n|J]} \ceq d_{n|J}a$ with $a \in A$, 
$n \in \bbN$ and $J \subset [N]$ (see \cref{lem:dnJ}).

By \cref{lem:sj-TS}, the superjet algebra $A^O$ has an an even derivation $T$ 
and odd derivations $S^i$ for $i \in [N]$ with
\begin{align}\label{eq:NWcom:TS}
 T(a^{[n|J]}) \ceq (n+1)a^{[n+1|J]}, \quad
 S^i(a^{[n|J]}) \ceq \sigma(e_i,J) a^{[n|J \sm e_i]}.
\end{align}
Then, \cref{lem:comW} yields:

\begin{prp}\label{prp:NWsj}
For every $A \in \SCom k$, the superjet algebra $A^O$ 
has a structure of commutative $N_W=N$ SUSY VA with $T$ and $S^i$ given by \eqref{eq:NWcom:TS}.
\end{prp}

\subsubsection{$N_K=N$ SUSY VA structure of superconformal jet algebras}

Recall \cref{dfn:scj}, which says that for any $A \in \SCom k$ 
we have the superconformal jet algebra $A^{\Osc}=(A^O,S_K^i)$.
It is a unital commutative superalgebra equipped with odd derivations 
\begin{align}\label{eq:NKcom:Sact}
 S_K^i(a^{[n|J]}) \ceq \begin{cases} \sigma(e_i,J) a^{[n|J \cup e_i]} & (i \notin J) \\
 (n+1) \sigma(e_i,J \sm e_i) a^{[n+1|J \sm e_i]} & (i \in J)\end{cases}.
\end{align}
$S_K^i$'s form the Lie superalgebra in \cref{lem:STL}.
Hence \cref{lem:comK} yields:

\begin{prp}\label{prp:sjCV}
For every $A \in \SCom k$, the superjet algebra $A^O$ has a structure of a commutative 
$N_K=N$ SUSY VA with $S^i$ being the odd derivation determined by \eqref{eq:NKcom:Sact}.
\end{prp}

For later reference, we give a lemma on the obtained structure on $A^O$.

\section{Li filtration of SUSY vertex algebra}\label{s:Lif}

In this subsection, we introduce SUSY analogue of Li's canonical filtration \cite{Li}.
As in the previous \cref{s:SUSY}, we fix a positive integer $N$.

\subsection{SUSY vertex Poisson algebras}\label{ss:VP}

Recall that, in the even case, a vertex Poisson algebra is a data $(V,\vac,T,Y_+,Y_-)$ 
combining a commutative VA structure $(\vac,T,Y_+)$ and a 
vertex Lie algebra structure $(V,T,Y_-)$ with common derivation $T$,
subject to the condition that $Y_-$ gives Poisson-like operations for 
the commutative multiplication associated $Y_+$.
See \cite[Chap.\ 16]{FBZ} and \cite[\S 2]{Li} for the detail.
A SUSY analogue of this notion is introduced in \cite{HK}, where 
the notion of Lie conformal algebra is used instead of vertex Lie algebras.
Below we give a definition using vertex Lie algebras.

\subsubsection{$N_W=N$ SUSY vertex Lie algebras}\label{sss:NWL}

We begin with the introduction of $N_W=N$ SUSY vertex Lie algebras.
Recall that the notion of an even vertex Lie algebra is obtained by taking 
singular part (called polar part in \cite[16.1.2. Definition]{FBZ}) of a vertex algebra.
For a linear superspace $V$ and a series 
$f(Z)=\sum_{(j|J)}Z^{j|J}v_{j|J} \in V\dbr{Z^{\pm1}}$ 
of a $1|N$-supervariable $Z=(z,\zeta^1,\dotsc,\zeta^N)$, we denote
\begin{align}\label{eq:Sing}
 \Sing(f) \ceq \sum_{(j|J), \, j<0} Z^{j|J}v_{j|J} \in Z^{-1}V[Z^{-1}]
\end{align}
and call it the singular part of $f$, 
where $Z^{-1}V[Z^{-1}] \ceq \{\sum_{(j|J), j<0} Z^{j|J}v_{j|J} \mid v_{j|J} \in V\}$.

\begin{dfn}\label{dfn:NWL}
An \emph{$N_W=N$ SUSY vertex Lie algebra}is a data $(V,T,S^i,Y_-)$ consisting of
\begin{itemize}[nosep]
\item 
a linear superspace $V$,
\item
an even operator $T \in \iEnd(V)_{\ol{0}}$,
\item
$N$ odd operators $S^i \in \iEnd(V)_{\ol{1}}$ for $i \in [N]$, and
\item
an even linear map $Y_-(\cdot,Z)\colon V \to \Hom(V,Z^{-1}V[Z^{-1}])$,
\end{itemize}
which satisfies the following axioms.
\begin{enumerate}[nosep]
\item 
Translation invariance: 
$Y_-(T a,Z) = \pd_{z} Y_-(a,Z)$, 
$Y_-(S^i a,Z) = \pd_{\zeta^i} Y_-(a,Z)$.
\item
Skew-symmetry:
$Y_-(a,Z)b = \Sing\bigl((-1)^{p(a) p(b)} e^{Z \nabla}Y(b,-Z)a\bigr)$
with $Z \nabla \ceq z T+\sum_{i=1}^N \zeta^i S^i$.

\item
Supercommutator:
$[a_{(m|M)},Y_-(b,Z)]=\Sing\bigl(Y(e^{-Z \nabla} a_{(m|M)} e^{Z \nabla} b,Z)\bigr)$,
where $a_{(m|M)}$ denotes the Fourier mode of the expansion 
$Y_-(a,Z) = \sum_{(m|M), m \ge 0}Z^{-m-1|N \sm M}a_{(m|M)}$.
\end{enumerate}
We often abbreviate the word ``vertex Lie algebra'' to VLA.
\end{dfn}

\begin{lem}\label{lem:NWL}
Given an $N_W=N$ SUSY VA $(V,\vac,T,S^i,Y)$, we have an $N_W=N$ SUSY VLA $(V,T,S^i,Y_-)$ 
by setting $Y_-(a,Z) \ceq \Sing(Y(a,Z))$.
\end{lem}

\begin{proof}
The axioms in \cref{dfn:NWL} are consequences of the translation invariance 
in \cref{dfn:NW}, \cref{fct:NWskew} and \cref{fct:NWcmt}.
\end{proof}

We also have the notion of a \emph{module} $M$ over an $N_W=N$ SUSY VLA $V$,
similar to SUSY VA modules.
For $a \in V$ and $m \in M$, we denote the SUSY vertex Lie action of $V$ on $W$ by 
\begin{align}\label{eq:VLact}
 Y^W_-(a,Z)m = \sum_{(j|J), \, j \ge 0}Z^{-1-j|N \sm J}a_{(j|J)}m \in M\dpr{Z}.
\end{align}

Our \cref{dfn:NWL} is equivalent to 
the notion of an $N_W=N$ SUSY Lie conformal algebra in \cite[Definition 3.2.2]{HK}.
Using the latter, we can describe the ``Lie algebra structure'' more explicitly.
Let us recall \cref{ntn:NWLam} of the $\Lambda$-bracket for $N_W=N$ case.
In particular, $\scL_W = \bbC[\Lambda] = \bbC[\lambda,\chi^i]$
is the commutative superalgebra freely generated by even $\lambda$ and odd $\chi^i$'s.
Note that the $\Lambda$-bracket $[a_\Lambda b]$ only involves $(j|J)$-operators 
for $j \ge 0$, so that it can be defined also for an $N_W=N$ SUSY VLA.
Now, let $\scL_W'$ be another commutative superalgebra 
freely generated by even $\gamma$ and odd $\eta^i$'s, 
and define $[a_\Gamma b]$ similarly using $\Gamma^{m|M} \ceq \gamma^m \eta^M$.

\begin{fct}[{c.f.\ \cite[\S 3.2.1]{HK}}]\label{fct:NWL-Jac}
Let $V$ be an $N_W=N$ SUSY VLA. For $a,b,c \in V$ of pure parity, 
we have the following identity in $\scL_W \otimes \scL_W' \otimes V$.
\[
 [a_\Lambda [b_\Gamma c]] = (-1)^{(p(a)+1)N}[[a_\Lambda b]_{\Gamma+\Lambda}c]
 +(-1)^{(p(a)+N)(p(b)+N)}[b_\Gamma[a_\Lambda c]].
\]
\end{fct}

We can regard it as a kind of Jacobi identity.
In order to understand the signs, let us recall:

\begin{dfn}[{\cite[Definition 3.2.5]{HK}}]\label{dfn:sL}
A \emph{Lie superalgebra of parity $q \in \Zt$} is a linear superspace $V$ 
with a binary operation $[\cdot,\cdot]\colon V \otimes V \to V$ 
satisfying the following axioms for $a,b,c \in V$ of pure parity.
\begin{enumerate}[nosep]
\item 
Skew-symmetry: $[a,b]=-(-1)^{p(a)p(b)+q}[b,a]$.
\item
Jacobi identity: $[a,[b,c]]=-(-1)^{p(a)q+q}[[a,b],c]+(-1)^{(p(a)+q)(p(b)+q)}[b,[a,c]]$.
\end{enumerate}
We call $[\cdot,\cdot]$ the \emph{Lie bracket of parity $q$}.
We also call $(V,[\cdot,\cdot])$ an \emph{even} or \emph{odd Lie superalgebra} 
according to $q=\ol{0}$ or $\ol{1}$.
\end{dfn}

\begin{rmk}\label{rmk:sL}
Some comments are in order.
\begin{enumerate}[nosep]
\item 
We slightly modify the terminology in \cite{HK}, 
and use the phrase ``of parity $q$'' instead of ``degree $q$'' used therein. 

\item
An even Lie superalgebra is a Lie superalgebra in the standard sense.

\item \label{rmk:sL:3}
We can find in \cref{fct:NWL-Jac} that ``the $\Lambda=0$ part'' 
of an $N_W=N$ SUSY VLA is a Lie superalgebra of parity $N \bmod 2$.
See \cite[Lemma 3.2.7]{HK} for the precise statement.
\end{enumerate}
\end{rmk}

\subsubsection{$N_W=N$ SUSY vertex Poisson algebras}\label{sss:NWP}

Now we introduce $N_W=N$ SUSY vertex Poisson structure.

\begin{dfn}[{c.f.\ \cite[Definition 3.3.16]{HK}}]\label{dfn:NWP}
An \emph{$N_K=W$ SUSY vertex Poisson algebra} is a data $(V,\vac,T,S^i,Y_+,Y_-)$ consisting of 
\begin{itemize}[nosep]
\item 
a commutative $N_K=N$ SUSY VA $(V,\vac,T,S^i,Y_+)$,
\item
an $N_W=N$ SUSY VLA $(V,T,S^i,Y_-)$
\end{itemize}
such that the vertex Lie structure $Y_-$ is a derivation for 
the commutative superalgebra structure $a b = a_{(-1|N)}b$ coming from $Y_+$.
More precisely speaking, for $a,b,c \in V$ of pure parity, we have 
\begin{align}\label{eq:VPA}
 Y_-(a,Z)(b c) = \bigl(Y_-(a,Z)b\bigr)c + (-1)^{(p(a)+N)p(b)}b\bigl(Y_-(a,Z)c\bigr).
\end{align}
We often abbreviate the word ``vertex Poisson algebra'' to VPA.
\end{dfn}

Recalling \cref{lem:com} of commutative property,
we denote the expansions of $Y_+$ and $Y_-$ by
\begin{align}\label{eq:VPA-exp}
 Y_+(a,Z) = \sum_{(j|J), \, j   < 0} Z^{-1-j|1-J}a_{(j|J)}, \quad 
 Y_-(a,Z) = \sum_{(j|J), \, j \ge 0} Z^{-1-j|1-J}a_{(j|J)}.
\end{align}

We also have the notion of a \emph{module} over an $N_W=N$ SUSY VPA $V$, which is a 
combination of the module structure over the commutative superalgebra associated to 
$(V,\vac,T,S^i,Y_+)$, and the module structure over the $N_W=N$ SUSY VLA $(V,T,S^i,Y_-)$.

As for the derivation axiom \eqref{eq:VPA}, let us recall:

\begin{dfn}\label{dfn:sP}
An \emph{Poisson superalgebra of parity $q \in \Zt$} is a commutative superalgebra $P$
together with a Lie bracket $\{\cdot,\cdot\}\colon P \otimes P \to P$ of parity $q$ 
in the sense of \cref{dfn:sL} such that 
\[
 \{a,b c\}=\{a,b\}c+(-1)^{(p(a)+q)p(b)}a\{b,c\}
\]
for $a,b,c \in P$ of pure parity.
$\{\cdot,\cdot\}$ is called the \emph{Poisson bracket of parity $q$}.
We also call $(V,\{\cdot,\cdot\})$ an \emph{even} or \emph{odd Poisson superalgebra} 
according to $q=\oz$ or $\oo$.
\end{dfn}

\begin{rmk}\label{rmk:sP}
Some comments are in order.
\begin{enumerate}[nosep]
\item 
An even Poisson superalgebra is a Poisson superalgebra in the standard sense.

\item
The structure of an odd Poisson superalgebra 
is a part of the Gerstenhaber algebra structure.

\item
By \cref{dfn:sP} and \cref{rmk:sL} \eqref{rmk:sL:3}, we see that ``the $\Lambda=0$ part" 
of an $N_W=N$ SUSY VPA is a Poisson superalgebra of parity $N \bmod 2$.
\end{enumerate}
\end{rmk}

\subsubsection{$N_W=N$ SUSY VPA structure on superjet algebra}\label{ss:sjVPA}

Let us give a basic example of an $N_K=N$ SUSY VPA.
Recall the $1|N$-superjet algebra 
\[
 A^O = A^{k\dbr{Z}}=\fcS^N(A_\infty)
\]  
of a commutative superalgebra $A$ (\cref{prp:sj}).
It has a structure of commutative $N_W=N$ SUSY VA (\cref{prp:NWsj}).
If $A$ is moreover a Poisson superalgebra of parity $N \mod 2$ in the sense of \cref{dfn:sP}, 
then this structure has an enhancement to vertex Poisson structure. 

\begin{prp}\label{prp:NWlv0}
Let $(P,\{\cdot,\cdot\})$ be a Poisson superalgebra of parity $N \bmod 2$,
and $(P^O,T,S^i)$ be the commutative $N_W=N$ SUSY VA. 
Then $P^O$ has an $N_W=N$ SUSY VPA structure such that 
\begin{align}\label{eq:NWlv0}
 u_{(m|M)}v = \begin{cases} \{u,v\} & (m=0,M=[N]) \\ 0 & (\text{otherwise})\end{cases}
\end{align}
for $u,v \in P \subset P^O$, where we used the expansion \eqref{eq:VPA-exp} of $Y_-$.
We call is the \emph{level $0$ SUSY VPA structure}.
\end{prp}

\begin{rmk}\label{rmk:lv0}
Our naming comes from the \emph{level $0$ VPA structure} \cite[Proposition 2.3.1]{A12} 
of the jet algebra $P_\infty$ (see \eqref{eq:Ainf})  in the even case.
\end{rmk}

For the proof, we need the following 
$N_W=N$ SUSY analogue of \cite[Theorem 3.6, Proposition 3.10]{Li04}.

\begin{lem}\label{lem:NW3.6}
Let $A$ be a unital commutative superalgebra equipped with an even derivation $T$ and 
odd derivations $S^i$ for $i \in [N]$.
Also, let $B \subset A$ be a linear sub-superspace generating 
the differential algebra $(A,T,S^i)$ in the sense of \cref{dfn:dsa}.
\begin{enumerate}[nosep]
\item \label{i:NW3.6:1}
Assume that there is a morphism of linear superspaces 
$Y_-(\cdot,Z)\colon A \to Z^{-1}(\Der A)\dbr{Z}$ satisfying 
\begin{enumerate}[nosep,label=(\roman*)]
\item \label{i:NW:c0} $Y_-(a,Z)a' \in Z^{-1}A[Z^{-1}]$, 
\item \label{i:NW:c1} $Y_-(1,Z)=0$, 
\item \label{i:NW:c2} $Y_-(T a,Z)=\pd_z Y_-(a,Z)$, $Y_-(S^i a,Z)=\pd_{\zeta^i} Y_-(a,Z)$, and 
\item \label{i:NW:c3} $[T,Y_-(a,Z)]=\pd_z Y_-(a,Z)$,  $[S^i,Y_-(a,Z)]=\pd_{\zeta^i} Y_-(a,Z)$
\end{enumerate}
for $a,a' \in A$ and $i \in [N]$.
Assume moreover that for $a \in A$ and $b \in B$ of pure parity, we have
\begin{align}\label{eq:NW3.26}
 Y_-(a,Z)b = \Sing\bigl((-1)^{p(a)p(b)}e^{Z \nabla}Y_-(b,-Z)a\bigr),
\end{align}
where $\Sing$ is given in \eqref{eq:Sing},
and $Z \nabla \ceq z T+\sum_{i=1}^N \zeta^i S^i$.
Then $Y_-$ gives an $N_W=N$ VPA structure on $A$.

\item \label{i:NW3.6:2}
Assume that there is a morphism of linear superspaces 
$Y^0_-(\cdot,Z)\colon B \to Z^{-1}(\iHom(B,A))\dbr{Z}$ satisfying
\begin{align}\label{eq:NW3.54}
 Y^0_-(b,Z)b' = \Sing\bigl((-1)^{p(b)p(b')}e^{Z \nabla}Y_-(b',-Z)b\bigr) 
 \in Z^{-1} A[Z^{-1}]
\end{align}
for $b,b' \in B$ of pure parity. Then $Y^0_-$ extends uniquely to 
$Y_-(\cdot,Z)\colon A \to Z^{-1}(\Der A)\dbr{Z}$ satisfying the conditions 
\ref{i:NW:c0}--\ref{i:NW:c3} and \eqref{eq:NW3.26} in \eqref{i:NW3.6:1}.
\end{enumerate}
\end{lem}

\begin{proof}
The arguments \cite[Theorem 3.6, Proposition 3.10]{Li04} works with little modification.
We omit the detail.
\end{proof}

\begin{proof}[Proof of \cref{prp:NWlv0}]
Apply \cref{lem:NW3.6} \eqref{i:NW3.6:2} to $B=P$, $A=P^O$
and $Y^0_-(b,Z) \ceq b_{(0|N)}$, $b_{(0|N)}b' \ceq \{b,b'\}$ for $b,b' \in P$,
which obviously satisfies the condition \eqref{eq:NW3.54}.
Then we have an extension $Y_-$ of $Y^0_-$ on $P^O$,
which gives the desired VPA structure by \cref{lem:NW3.6} \eqref{i:NW3.6:1}.
\end{proof}

\begin{rmk}\label{rmk:NWlv0-1}
The above proof lacks explicit formulas of $a_{(m|M)}b$ for $a,b \in P^O$.
For a comparison with the even case \cite[Proposition 2.3.1]{A12} mentioned 
in \cref{rmk:lv0}, we explain how to obtain concrete formulas.

Let us consider $u_{(m|M)}a$ with $u \in P$, $a \in P^O$, $m \ge 0$ and $M \subset [N]$.
Note that every $a \in P^O$ is written as a polynomial of $S^L T^{(l)} v$
with $L \subset [N]$, $l \in \bbN$ and $v \in P$, where $S^L \ceq S^{l_1} S^{l_2} \dotsm$
for $L=\{l_1,l_2,\dotsc\}$ and $T^{(l)} \ceq \frac{1}{l!}T^l$.
Hence, by the Leibniz rule \eqref{eq:VPA}, it is enough to determine 
$u_{(m|M)}(S^L T^{(l)} v)$ for $u,v \in P$ of pure parity. 
We determine them by the condition \ref{i:NW:c3} in \cref{lem:NW3.6},
which is equivalent to (recall \eqref{eq:NWrec})
\begin{align*}
 [S^i,u_{(m|M)}] = 
 \begin{cases} \sigma(N \sm M,e_m)a_{(m|M \sm e_i)} & (i \in M) \\
 0 & (i \notin M) \end{cases}, \quad
 [T,u_{(m|M)}] = -m u_{(m-1|M)}
\end{align*}
with the convention $u_{(-1|M)} \ceq 0$. 
These are recursion for the desired $u_{(m|M)}(S^L T^l v)$, and using \eqref{eq:NWlv0} 
as the initial condition, we can solve them to obtain:
\begin{itemize}[nosep]
\item In case $l \ge m$, we set 
\begin{align}
 \label{eq:NWlv0:mN}
 u_{(m|N)}(S^L T^{(l)}v) &\ceq (-1)^{\abs{L}p(u)} S^L T^{(l-m)}\{u,v\}, \\
 \label{eq:NWlv0:mM}
 u_{(m|N \sm L)}(S^L T^{(l)}v) &\ceq (-1)^{\abs{L}p(u)+\lfloor \abs{L}/2 \rfloor}
 \bigl(\tprd_{i=1}^{\abs{L}}\sigma(\{l_{i+2},l_{i+2},\dotsc\},e_{l_i})\bigr) 
 T^{(l-m+\abs{I})}\{u,v\},
\end{align}
and otherwise $u_{(m|M)}(T^{(l)} S^L v) \ceq 0$.
\item
In case $l<m$, we set $u_{(m|M)}(T^{(l)} S^L v) \ceq 0$.
\end{itemize}
We can check directly that these formulas satisfy 
the condition \ref{i:NW:c2} in \cref{lem:NW3.6}.
Let us also mention that the case $L=\emptyset$ in \eqref{eq:NWlv0:mN} 
is analogous to the even case formula $u_{(m)}(T^l v) = \frac{l!}{(l-m)!} T^{l-m}\{u,v\}$
for $l \ge m$ in \cite[Proposition 2.3.1, (10)]{A12}.
\end{rmk}

\begin{rmk}\label{rmk:NWlv0-2}
Let us sketch another description of the level $0$ SUSY VPA structure, following 
\cite[Lemma 3.1]{Mal}, where the level $0$ VPA structure in the even case is treated 
in terms of the language of coisson algebra \cite[\S 2.6]{BD}.
(Although it is not shown explicitly in \cite{Mal} that 
 the coisson structure coincides with the level $0$ VPA structure in \cite{A12}, 
 one can check it by writing down the coisson product.)
The following argument has some overlap with \cite[Remark 3.2.3]{HK}, from which 
we borrow some symbols. 

Let $\scH_W=k[\pd,\delta^1,\dotsc,\delta^N]$ be the commutative superalgebra over 
the base field $k$ generated by an even variable $\pd$ and odd variables $\delta^i$
for $i \in [N]$.
We regard $\scH_W$ as a Hopf superalgebra with comultiplication 
$\Delta(\pd) \ceq \pd \otimes 1 + 1 \otimes \pd$, 
$\Delta(\delta^i) \ceq \delta^i \otimes 1 + 1 \otimes \delta^i$ and 
counit $\ep(\pd)=\ep(\delta^i) \ceq 0$.
On the category $\sfM$ of right $\scH_W$-modules, 
we can define the \emph{$*$-pseudo-tensor structure} 
in a similar way as the even case \cite{BD}, \cite[\S 3.1]{Mal}.
A Lie algebra object $L$ in the corresponding pseudo-tensor category $\sfM^*$ 
is called a \emph{$\text{Lie}^*$ algebra}.
The $\text{Lie}^*$-bracket $[\cdot,\cdot]$ is an element of 
$\Hom_{\scH_W^2}(L \otimes_k L, L \otimes_{\scH_W} \scH_W^2)$,
where $\scH_W^2$ denotes the tensor product algebra $\scH_W \otimes_k \scH_W$,
and we regard it as a right $H$-module by $\Delta$.
Identifying $L \otimes_{\scH_W} \scH_W^2$ with $L[\pd_1,\delta_1^i]$,
where $\pd_1 \ceq \pd \otimes 1$ and $\delta_1^i \ceq \delta^i \otimes 1$, 
we can write $[\cdot,\cdot]$ as 
$[a,b] = \sum_{(j|J), j \ge 0} (a_{[j|J]}b)\pd_1^j \delta_1^J$ for some $a_{[j|J]}b \in L$.
As in the even case, such a $\text{Lie}^*$ algebra is equivalent to our $N_W=N$ SUSY VLA, 
and also to an $N_W=N$ SUSY Lie conformal algebra in \cite{HK}.

The category $\sfM$ has a tensor structure $\otimes^!$
coming from the comultiplication $\Delta$ on $\scH_W$,
and the corresponding tensor category is denoted by $\sfM^!$.
A commutative algebra object in $\sfM^!$ is called a $\text{commutative}^!$ algebra.
A \emph{coisson algebra} $C$ is a $\text{Lie}^*$ algebra and 
a $\text{commutative}^!$ algebra such that the multiplication $C \otimes^! C \to C$
is a $\text{Lie}^*$ algebra morphism.

Now we can restate \cref{prp:NWlv0}.
Let $(P,\{\cdot,\cdot\})$ be a Poisson superalgebra of parity $N \bmod 2$.
Note that $P^O \simeq P \otimes_k \scH_W$ with operators $T$ and $S^i$ corresponding to 
$\pd$ and $\delta^i$, respectively.
The statement is that $P \otimes_k \scH_W$ has a coisson structure 
whose $\text{Lie}^*$ bracket $\{\cdot,\cdot\}$ is given by
$\{u \pd^l \delta^L, v \pd^m \delta^M\} =\{u,v\}\pd_1^m \delta_1^L \pd_2^m \delta_2^M$ 
for $u,v \in P$, $l,m \in \bbN$ and $L,M \subset [N]$, 
where $\pd_2 \ceq 1 \otimes \pd$ and $\delta_2^i \ceq 1 \otimes \delta^i$.
\end{rmk}

\subsubsection{$N_K=N$ SUSY vertex Poisson algebras}

Here we give an $N_K=N$ analogue of \cref{sss:NWL} and \cref{sss:NWP}.
We begin with:

\begin{dfn}\label{dfn:NKL}
An \emph{$N_K=N$ SUSY vertex Lie algebra}is a data $(V,S^i,Y_-)$ consisting of
\begin{itemize}[nosep]
\item 
a linear superspace $V$,
\item
$N$ odd operators $S^i \in \iEnd(V)_{\ol{1}}$ for $i \in [N]$, and 
\item
an even linear map $Y_-(\cdot,Z)\colon V \to \Hom(V,Z^{-1}V[Z^{-1}])$,
\end{itemize}
which satisfies the following axioms.
\begin{enumerate}[nosep]
\item 
Translation invariance: $Y_-(S^i a,Z) = D^i_Z Y_-(a,Z)$.
\item
Skew-symmetry:
$Y_-(a,Z)b = \Sing\bigl((-1)^{p(a) p(b)} e^{Z \nabla}Y(b,-Z)a\bigr)$
with $Z \nabla \ceq z T+\sum_{i=1}^N \zeta^i S^i$, $T \ceq (S^1)^2=\dotsb=(S^N)^2$.

\item
Supercommutator:
$[a_{(m|M)},Y_-(b,Z)]=\Sing\bigl(Y(e^{-Z \nabla} a_{(m|M)} e^{Z \nabla} b,Z)\bigr)$,
where $a_{(m|M)}$ denotes the Fourier mode of the expansion 
$Y_-(a,Z) = \sum_{(m|M), m \ge 0}Z^{-m-1|N \sm M}a_{(m|M)}$.
\end{enumerate}
We often abbreviate the word ``vertex Lie algebra'' to VLA.
\end{dfn}

As in \cref{lem:NWL}, we have:

\begin{lem}
Given an $N_K=N$ SUSY VA $(V,\vac,S^i,Y)$, we have an $N_K=N$ SUSY VLA $(V,S^i,Y_-)$ 
by setting $Y_-(a,Z) \ceq \Sing(Y(a,Z))$.
\end{lem}

Our \cref{dfn:NKL} is equivalent to the notion of an $N_K=N$ SUSY Lie conformal algebra
in \cite[Definition 4.10]{HK}. For later reference, we record the Jacobi identity.
Recall \cref{ntn:NKLam} of the $\Lambda$-bracket.
In particular, $\scL$ denotes the superalgebra generated by even $\lambda$ and odd $\chi^i$.
Note that the $\Lambda$-bracket makes sense for an $N_K=N$ SUSY VLA $V=(V,S^i,Y_-)$: we have 
\[
 [a_\Lambda b] = \sum_{(m|M), \, m \ge 0} 
  \sigma(M,N \sm M) (-1)^{\binom{\abs{M}+1}{2}} \frac{1}{m!} \Lambda^{m|M} a_{(m|M)}b.
\]
for $a,b \in V$.
Also, let $\scL'$ be the superalgebra generated by even $\gamma$ and odd $\eta^i$'s 
subject to the relation $[\gamma,\eta^i]=0$ and $[\eta^i,\eta^j]=-2\gamma \delta_{i,j}$, 
and define $[a_\Gamma b]$ similarly using $\Gamma^{m|M} \ceq \gamma^m \eta^M$.

\begin{fct}[{c.f.\ \cite[Definition 4.10]{HK}}]\label{fct:VL-Jac}
Let $V$ be an $N_K=N$ SUSY VLA. For $a,b,c \in V$ of pure parity, 
we have the following identity in $\scL \otimes \scL' \otimes V$.
\[
 [a_\Lambda [b_\Gamma c]] = (-1)^{(p(a)+1)N}[[a_\Lambda b]_{\Gamma+\Lambda}c]
 +(-1)^{(p(a)+N)(p(b)+N)}[b_\Gamma[a_\Lambda c]].
\]
\end{fct}

\begin{dfn}[{c.f.\ \cite[\S 4.10]{HK}}]\label{dfn:NKP}
An \emph{$N_K=N$ SUSY vertex Poisson algebra} ($N_K=N$ SUSY VPA for short) 
is a data $(V,\vac,S^i,Y_+,Y_-)$ consisting of 
\begin{itemize}[nosep]
\item 
a commutative $N_K=N$ SUSY VA $(V,\vac,S^i,Y_+)$,
\item
an $N_K=N$ SUSY VLA $(V,S^i,Y_-)$
\end{itemize}
such that the vertex Lie structure $Y_-$ is a derivation for 
the commutative superalgebra structure $a b = a_{(-1|N)}b$ coming from $Y_+$
(see \eqref{eq:VPA}).
\end{dfn}

We also have the notion of a \emph{module} $M$ over an $N_K=N$ SUSY VPA $V$.
The VLA action of $V$ on $M$ is denoted as \eqref{eq:VLact}.

\subsubsection{$N_K=N$ SUSY VPA structure on superconformal jet algebra}

We have an $N_K=N$ analogue of \cref{ss:sjVPA}.
Recall the $1|N$-superconformal jet algebra 
\[
 A^{\Osc} = (A^O,S^i) = (\dR{A_\infty}{},S_K^i),
\]  
for a commutative superalgebra $A$ (\cref{dfn:scj}). It has a structure of 
a commutative $N_K=N$ SUSY VA (\cref{prp:sjCV}), and if $A$ is moreover a Poisson superalgebra, 
then this structure can be enhanced to a vertex Poisson structure.

\begin{prp}\label{prp:NKlv0}
Let $(P,\{\cdot,\cdot\})$ be a Poisson superalgebra of parity $N \bmod 2$,
and $P^{\Osc}=(P^O,S_K^i)$  be the superconformal jet algebra of $P$.
Then $P^{\Osc}$ has an $N_K=N$ SUSY VPA structure such that 
\begin{align}\label{eq:NKlv0}
 u_{(m|M)}v = \begin{cases} \{u,v\} & (m=0,M=[N]) \\ 0 & (\text{otherwise})\end{cases}
\end{align}
for $u,v \in P \subset P^O$, where we used the expansion \eqref{eq:VPA-exp} of $Y_-$.
We call is the \emph{level $0$ SUSY VPA structure}.
\end{prp}

The proof is similar to \cref{prp:NWlv0}, using the following lemma.

\begin{lem}\label{lem:NK3.6}
Let $A$ be a unital commutative superalgebra equipped with odd derivations 
$S_K^i$ for $i \in [N]$ subject to the commutation relation (c.f.\ \cref{lem:STL})
\[
 [S_K^i,S_K^j] = 2 T \delta_{i,j}, \quad [S_K^i,T] = 0.
\]
Also, let $B \subset A$ be a sub-superspace generating 
the differential algebra $(A,S_K^i)$ (see \cref{dfn:dsa}).
\begin{enumerate}[nosep]
\item \label{i:NK3.6:1}
Assume that there is a morphism of linear superspaces 
$Y_-(\cdot,Z)\colon A \to Z^{-1}(\Der A)\dbr{Z}$ satisfying 
\begin{enumerate}[nosep,label=(\roman*)]
\item \label{i:NKlv0:c0} $Y_-(a,Z)a' \in Z^{-1}A[Z^{-1}]$, 
\item \label{i:NKlv0:c1} $Y_-(1,Z)=0$, 
\item \label{i:NKlv0:c2} $Y_-(S_K^i a,Z)=D_Z^i Y_-(a,Z)$,
\item \label{i:NKlv0:c3} $[S_K^i,Y_-(a,Z)]=\ol{D}_Z^i Y_-(a,Z)$
\end{enumerate}
for $a,a' \in A$ and $i \in [N]$.
Further assume that for $a \in A$ and $b \in B$ of pure parity, we have
\begin{align}\label{eq:NK3.26}
 Y_-(a,Z)b = \Sing\bigl((-1)^{p(a)p(b)}e^{Z \nabla}Y_-(b,-Z)a\bigr),
\end{align}
where $\Sing$ is given in \eqref{eq:Sing},
and $Z \nabla \ceq z T+\sum_{i=1}^N \zeta^i S_K^i$ as in \cref{dfn:NKL}.
Then $Y_-$ gives an $N_K=N$ VPA structure on $A$.

\item \label{i:Li04:3.6:2}
Assume that there is a morphism of linear superspaces 
$Y^0_-(\cdot,Z)\colon B \to Z^{-1}(\iHom(B,A))\dbr{Z}$ satisfying
\begin{align}\label{eq:Li04:3.54}
 Y^0_-(b,Z)b' = \Sing\bigl((-1)^{p(b)p(b')}e^{Z \nabla}Y_-(b',-Z)b\bigr) 
 \in Z^{-1} A[Z^{-1}]
\end{align}
for $b,b' \in B$ of pure parity. Then $Y^0_-$ extends uniquely to 
$Y_-(\cdot,Z)\colon A \to Z^{-1}(\Der A)\dbr{Z}$ satisfying the conditions 
\ref{i:NKlv0:c0}--\ref{i:NKlv0:c3} and \eqref{eq:NK3.26} in \eqref{i:NK3.6:1}.
\end{enumerate}
\end{lem}

\begin{rmk}
As in the $N_W=N$ case (\cref{rmk:NWlv0-1}),
we have an explicit form of the VPA structure.
We will describe  $u_{(m|M)}(S^L T^l v)$ for $u,v \in P$ of pure parity. 
The condition \ref{i:NKlv0:c3} in \cref{lem:NK3.6} is equivalent to 
(recall \eqref{eq:NKrecS} and \eqref{eq:NKrecT})
\begin{align*}
&[S^i,u_{(m|M)}] = 
 \begin{cases} \sigma(N \sm M,e_m)a_{(m|M \sm e_i)} & (i \in M) \\
 -\sigma(N \sm (M \cup e_i),e_m) m a_{(m-1|M \cup e_i)} & (i \notin M) \end{cases}, \\
&[T,u_{(m|M)}] = -m u_{(m-1|M)}
\end{align*}
with the convention $u_{(-1|M)} \ceq 0$. 
These are recursion for the desired $u_{(m|M)}(S^L T^l v)$, and using \eqref{eq:NKlv0} 
as the initial condition, we can solve them to obtain:
\begin{itemize}[nosep]
\item In case $l \ge m$, we set 
\begin{align*}
 u_{(m|N)}(S^L T^{(l)}v) &\ceq (-1)^{\abs{L}p(u)} S^L T^{(l-m)}\{u,v\}, \\
 u_{(m|N \sm L)}(S^L T^{(l)}v) &\ceq (-1)^{\abs{L}p(u)+\lfloor \abs{L}/2 \rfloor}
 \bigl(\tprd_{i=1}^{\abs{L}}\sigma(\{l_{i+2},l_{i+2},\dotsc\},e_{l_i})\bigr) 
 T^{(l-m+\abs{I})}\{u,v\},
\end{align*}
and otherwise $u_{(m|M)}(T^{(l)} S^L v) \ceq 0$.
\item
In case $l<m$, we set $u_{(m|M)}(T^{(l)} S^L v) \ceq 0$.
\end{itemize}
\end{rmk}

\begin{rmk}
We also have another description of the level $0$ SUSY VPA structure via chiral algebra
analogous to \cref{rmk:NWlv0-2}. 
Instead of using the commutative superalgebra $H=k[\pd,\sld^i]$,
we use the non-commutative superalgebra 
$\scH=k\langle \sld_K^i \rangle$ generated by odd variables $\sld_K^i$ 
subject to the relation $[\sld_K^i,\sld_K^i]=[\sld^j,\sld_K^j]$ and 
$[\sld_K^i,\sld_K^j]=0$ for $i \neq j$.
Let us denote $\pd \ceq \hf [\sld_K^i,\sld_K^i]$.
We regard $\scH$ as a Hopf superalgebra, similarly as $H$.
The category $\sfM$ of right $\scH$-modules has the $*$-pseudo-tensor structure,
and the $\text{Lie}^*$-bracket $[\cdot,\cdot]$ on an object $L \in \sfM$ 
is an element of $\Hom_{\scH^2}(L \otimes_k L, L \otimes_{\scH} \scH^2)$.
Under the identification $L \otimes_{\scH} \scH^2 \cong L\langle \sld_1^i \rangle$
with $\sld_1^i \ceq \sld_K^i \otimes 1$, we can write $[\cdot,\cdot]$ as 
$[a,b] = \sum_{(j|J),j \ge 0} (a_{[j|J]}b)\pd_1^j \sld_1^J$ for some $a_{[j|J]}b \in L$.
A $\text{Lie}^*$ algebra is equivalent to an $N_K=N$ SUSY VLA, 
and also to an $N_K=N$ SUSY Lie conformal algebra in \cite{HK}.
The equivalence is given by $a_{[2j|J]}b = a_{(j|J)}b$.

The category $\sfM_K$ has a tensor structure $\otimes^!$
coming from the comultiplication $\Delta$ on $\scH$,
and has a compound tensor structure $(\otimes^*,\otimes^!)$. 
As a result, we have the notion of a coisson algebra on $\sfM_K$.

Now we can restate \cref{prp:NKlv0}.
Let $(P,\{\cdot,\cdot\})$ be an odd Poisson superalgebra. Note that 
$P^O \simeq P \otimes_k \scH$ with the odd operators $S^i$ corresponding to $\sld_K^i$.
The statement is that $P \otimes_k \scH$ has a coisson structure 
whose $\text{Lie}^*$ bracket $\{\cdot,\cdot\}$ is given by
$\{u \pd^l \sld_K^L, v \pd^m \sld_K^M\} =\{u,v\}\pd_1^l \sld_1^L \pd_2^m \sld_2^M$ 
for $u,v \in P$, $l,m \in \bbN$ and $L,M \subset [N]$, 
where $\sld_2^i \ceq 1 \otimes \sld_K^i$.
\end{rmk}

\subsection{Li filtration of $N_W=N$ SUSY vertex algebra}\label{ss:NWLif}

In this and the next subsections, we introduce a SUSY analogue of 
Li's canonical filtration \cite[Definition 2.7]{Li}.
Let us give a brief recollection of the even case.
Let $V$ be an even vertex algebra, and $M$ be a $V$-module.
Then the Li filtration of $M$ is a decreasing sequence of subspaces
\[
 M=E_0(M) \supset E_1(M) \supset \dotsb \supset E_n(M) \supset \dotsb
\]
which consists of 
\begin{align}\label{eq:ev-Li}
 E_n(M) \ceq \spn\left\{
 a^1_{(-1-k_1)} \dotsm a^r_{(-1-k_r)} m \, \Bigg|
 \begin{array}{l}
  r \in \bbZ_{>0}, \, a^i \in V, \, m \in M, \, k_i \in \bbN \\
  \text{satisfying $k_1+\dotsb+k_r \ge n$}
 \end{array}
 \right\}.
\end{align}

Our strategy for SUSY case is to discard the odd index $J$ of $(j|J)$-operators
and to apply Li's arguments in \cite{Li}.
Hence, some of the arguments in loc.\ cit.\ work as they are, but some do not.
The points where Li's argument need to be modified are the ones using 
Borcherds' commutator formula \eqref{eq:cmt-ev} and iterate formula \eqref{eq:ev-iter}.

Hereafter until the end of this subsection, 
$V=(V,\vac,T,S^i,Y)$ denotes an $N_K=N$ SUSY VA.

The construction of Li filtration in the even case \cite{Li} starts with 
an argument \cite[Proposition 2.6]{Li} 
on a general decreasing filtration of a vertex algebra whose associated graded space
has a natural structure of a vertex Poisson algebra.
The following statement is an $N_W=N$ SUSY analogue of \cite[Proposition 2.6]{Li}.

\begin{prp}\label{prp:NW2.6}
Let $V=E_0 \supset E_1 \supset \dotsb \supset E_n \supset \dotsb$ be a decreasing 
filtration of linear sub-superspaces of $V$ such that $\vac \in E_0$ and 
\begin{align}\label{eq:NW2.6}
 a_{(j|J)}b \in E_{r+s-j-1} \quad 
 (a \in E_r, \, b \in E_s, \, j \in \bbZ, \, J \subset [N]), 
\end{align}
where we used the convention $E_n \ceq V$ for $n \in \bbZ_{<0}$.
\begin{enumerate}[nosep]
\item \label{i:NW2.6:1}
The associated graded linear superspace 
\[
 \gr_E V \ceq \tboplus_{n \in \bbN} E_n/E_{n+1}
\]
is an $N_W=N$ SUSY VA whose state-superfield correspondence is given by 
\begin{align}\label{eq:NW2.6:op}
 (a+E_{r+1})_{(j|J)}(b+E_{s+1}) \ceq a_{(j|J)}b+E_{r+s-j}
\end{align}
for $a \in E_r$, $b \in E_s$, $j \in \bbZ$ and $J \subset [N]$,
whose vacuum is given by $\vac+E_1 \in E_0/E_1$,
and whose even operator $\pd$ and odd operators $\sld^i$ are given by 
\[
 \pd(a+E_{r+1}) \ceq T a + E_{r+2}, \quad \sld^i(a+E_{r+1}) \ceq S^i a + E_{r+2}.
\]

\item
The $N_W=N$ SUSY VA $\gr_E V$ in \eqref{i:NW2.6:1} is commutative if and only if 
\begin{align}\label{eq:NW2.6:com}
 a_{(m|M)}b \in E_{r+s-m}
\end{align}
for all $a \in E_r$, $b \in E_s$, $m \in \bbN$ and $M \subset [N]$.

\item \label{i:NW2.6:3}
Under the condition \eqref{eq:NW2.6:com}, the commutative $N_W=N$ SUSY VA $\gr_E V$ 
has an $N_W=N$ SUSY VPA structure 
\begin{align}\label{eq:NW2.6:Y-}
 Y_-(a+E_{r+1},Z)(b+E_{s+1}) \ceq 
 \sum_{(m|M), \, m \ge 0}Z^{-m-1|N \sm M}(a_{(m|M)}b+E_{r+s-m+1})
\end{align}
for $a \in E_r$ and $b \in E_s$.
\end{enumerate}
\end{prp}

\begin{proof}
The condition \eqref{eq:NW2.6} does not depend on the odd index $J$,
and the proof of \cite[Proposition 2.6]{Li} works, except for \eqref{i:NW2.6:3}
where Borcherds' commutator formula is used. 
So we give a proof of \eqref{i:NW2.6:3} only.

Well-definedness of \eqref{eq:NW2.6:Y-} is guaranteed by the condition \eqref{eq:NW2.6:com}.
Checking the axioms of $N_W=N$ SUSY VLA for $\gr_E V$ is then straightforward.
It remains to check the derivation axiom \eqref{eq:VPA}, which is equivalent to
\begin{align}\label{eq:NW2.6:der}
 a_{(l|L)}(b_{(-1|N)}c) + E_{r+s+t-m+1} = (a_{(l|L)}b)_{(-1|N)}c +
 (-1)^{(p(a)+\abs{L})p(b)} b_{(-1|N)}(a_{(l|L)}c) + E_{r+s+t-m+1}
\end{align}
for $l \in \bbN$, $L \subset [N]$, $a \in E_r$, $b \in E_s$ and $c \in E_t$.
The commutator formula \eqref{eq:NWcmt} yields
\begin{align*}
&a_{(l|L)}(b_{(-1|N)}c) - (-1)^{(p(a)+\abs{L})p(b)} b_{(-1|N)}(a_{(l|L)}c)  \\
&=(a_{(l|L)}b)_{(-1|N)}c +
  \sum_{j=0}^{l-1} \pm \binom{l}{j} (a_{(j|N)}b)_{(l-1-j|N)}c,
\end{align*}
where $\pm$ denotes some sign.
For each term of this equality, the condition \eqref{eq:NW2.6:com} yields
\[
 a_{(l|L)}(b_{(-1|N)}c), \, b_{(-1|N)}(a_{(l|L)}c), \, (a_{(l|L)}b)_{(-1|N)}c 
 \in E_{r+s+t-l},
\]
and since $l-1-j \ge 0$ in the summation, it also yields
\begin{align*}
 (a_{(j|N)}b)_{(l-1-j)}c \in E_{(r+s-j)+t-(l-1-j)} = E_{r+s+t-l+1}.
\end{align*}
Hence we have the desired equality \eqref{eq:NW2.6:der}.
\end{proof}

Now, mimicking the even Li filtration \eqref{eq:ev-Li} with the strategy of 
``discarding the odd index $J$ of the $(j|J)$-operator'', we introduce:

\begin{dfn}\label{dfn:NWLif}
For a $V$-module $M$ and $n \in \bbZ$, we define a linear sub-superspace $E_n(M) \subset M$ by
\[
 E_n(M) \ceq \spn\left\{
 a^1_{(-1-k_1|K_1)} \dotsm a^r_{(-1-k_r|K_r)} m \, \Bigg|
 \begin{array}{l}
  r \in \bbZ_{>0}, \, a^i \in V, \, m \in M, \, k_i \in \bbN, \, K_i \subset [N] \\
  \text{satisfying $k_1+\dotsb+k_r \ge n$}
 \end{array}
 \right\}.
\]
\end{dfn}

Below we show that $\{E_n(M) \mid n \in \bbN\}$ satisfies the conditions in \cref{prp:NW2.6},
so that for $M=V$, the associated graded space $\gr_E V$ has a structure of $N_W=N$ SUSY VPA.
Our argument follows that of \cite[Lemma 2.8--Proposition 2.11]{Li}.

The statements in the next lemma are analogue of \cite[Lemmas 2.8, 2.9]{Li}.

\begin{lem}\label{lem:NW2.89}
For a $V$-module $M$, we have the following.
\begin{enumerate}[nosep]
\item \label{i:NW2.89:1}
$E_n(M) \supset E_{n+1}(M)$ for any $n \in \bbZ$.

\item \label{i:NW2.89:2}
$E_n(M) = M$ for any $n \in \bbZ_{\le 0}$

\item \label{i:NW2.89:3}
$a_{(-1-k|K)}E_n(M) \subset E_{n+k}(M)$ 
for any $a \in V$, $k \in \bbN$, $K \subset [N]$ and $n \in \bbZ$.

\item \label{i:NW2.89:4}
For $n \ge 1$, $E_n(M)$ is equal to 
\begin{align}\label{eq:NW2.32}
 E_n'(M) \ceq \spn\left\{ a_{(-1-k|K)}m \mid 
  \text{ $a \in V$, $k \in \bbZ_{>0}$, $K \subset [N]$, $m \in E_{n-k}(M)$} \right\}.
\end{align}

\item \label{i:NW2.89:5}
For $n \ge 1$, $E_n(M)$ is equal to 
\begin{align}\label{eq:NW2.33}
 E_n''(M) \ceq \spn\left\{ a^1_{(-1-k_1|K_1)} \dotsm a^r_{(-1-k_r|K_r)} m \, \Bigg|
 \begin{array}{l}
  r \in \bbZ_{>0}, \, a^i \in V, \, m \in M, \, k_i \in \bbZ_{\ge1}, \, K_i \subset [N] \\
  \text{satisfying $k_1+\dotsb+k_r \ge n$}
 \end{array} \right\}.
\end{align}
Note that the difference with \cref{dfn:NWLif} is the condition $k_i \ge 1$.
\end{enumerate}
\end{lem}

\begin{proof}
The items \eqref{i:NW2.89:1}--\eqref{i:NW2.89:3} are immediate consequence 
of \cref{dfn:NWLif}, and we omit the proof. 
For the rest, we have by induction on $n$ that 
\begin{align}\label{eq:NWLif:4=5}
 E_n'(M) = E_n''(M).
\end{align}
So it is enough to prove \eqref{i:NW2.89:4} only. For that, we show that each element 
\[
 u = a^1_{(-1-k_1|K_1)} \dotsm a^r_{(-1-k_r|K_r)} m \in E_n(M)
\]
belongs to $E_n'(W)$ by induction on the length $r$, 
using the $N_W=N$ SUSY commutator formula \eqref{eq:NWcmt}.
\begin{itemize}[nosep]
\item 
If $r=1$, then $u=a^1_{(-1-k_1|K_1)} m$ with $k_1 \ge n \ge 1$
and $m \in M = E_{n-k_1}(M)$. Thus we have $u \in E_n'(M)$ by definition.

\item
Next, assume $r \ge 2$. We set 
\[
 u = a^1_{(-1-k_1|K_1)}u', \quad u' \ceq a^2_{(-1-k_2|K_2)} \dotsm a^r_{(-1-k_r|K_r)} m.
\]
If $k_1 \ge 1$, then $u' \in E_{n-k_1}(M)$, and we have
$u =a^1_{(-1-k_1|K_1)}u'\in E_n'(M)$ as desired.
Hereafter we assume $k_1=0$.
Then $k_2+\dotsb+k_r \ge n$ and $u' \in E_n(M)$, 
which yields $u' \in E_n'(M)$ by induction hypothesis.
Now the equality \eqref{eq:NWLif:4=5} yields $u' \in E_n''(M)$.
Hence we may assume $k_2,\dotsc,k_r \ge 1$.
Let us rewrite $a \ceq a^1$, $b \ceq a^2$ and $k \ceq k_2$, so that we have 
\[
 u = a_{(-1|K_1)}b_{(-1-k|K_2)}u'', \quad 
 u'' \ceq a^3_{(-1-k_3|K_3)} \dotsm a^r_{(-1-k_r|K_r)} m \in E_{n-k}(M). 
\]
The commutator formula \eqref{eq:NWcmt} yields
\begin{align*}
 u=\pm b_{(-1-k|K_2)}a_{(-1|K_1)}u'' 
  +\sum_{\substack{(j|J), \\ j \ge 0, \, J \supset K_1 \cap K_2}} 
   \cst \cdot (a_{(j|J)}b)_{(-2-k-j|K)}u'',
\end{align*}
where $\pm$ denotes a sign, $\cst$ denotes a constant and $K \ceq K_2 \cup (K_1 \sm J)$.
Since $u'' \in E_{n-k}(M)$, we have $a_{(-1|K_1)}u'' \in E_{n-k}(M)$ by \eqref{i:NW2.89:3}, 
and the condition $k \ge 1$ yields 
\[
 b_{(-1-k|K_2)}a_{(-1|K_1)}u'' \in E_n'(M). 
\]
Also, $u'' \in E_{n-k}(M) \subset E_{n-(k+j+1)}(M)$ for $j \ge 0$ by \eqref{i:NW2.89:1}, 
and we have 
\[
 (a_{(j|J)}b)_{(-2-k-j|K')}u'' \in E_n'(M).
\]
Hence we have $u \in E_n'(M)$.
\end{itemize}
\end{proof}

The next lemma is an analogue of \cite[Lemma 2.10]{Li}.

\begin{lem}\label{lem:NW2.10}
Let $M$ be a $V$-module. For $a \in V$, $l,n \in \bbZ$ and $L \subset [N]$, we have
\begin{align}\label{eq:NW2.34}
 a_{(l|L)}E_n(M) \subset E_{n-l-1}(M).
\end{align}
If moreover $l \ge 0$, then we have 
\begin{align}\label{eq:NW2.35}
 a_{(l|L)}E_n(M) \subset E_{n-l}(M).
\end{align}
\end{lem}

\begin{proof}
For $l \le -1$, \eqref{eq:NW2.34} follows from \cref{lem:NW2.89} \eqref{i:NW2.89:4}.
So assume $l \ge 0$. Then it is enough to show \eqref{eq:NW2.35} only 
since $E_{n-l-1}(M) \supset E_{n-l}(M)$. The argument is an induction on $n$.
For $n \le 0$, we have $E_n(M)=E_{n-l}(M)=M$ by \cref{lem:NW2.89} \eqref{i:NW2.89:2},
which shows \eqref{eq:NW2.35}. Next, assume $n \ge 1$.
By \cref{lem:NW2.89} \eqref{i:NW2.89:4}, $E_n(M)$ is spanned by elements
$b_{(-1-k|K)}m$ with $b \in V$, $k \ge 1$, $K \subset [N]$ and $m \in E_{n-k}(M)$.
We will show that $a_{(l|L)}b_{(-1-k|K)}m$ belongs to $E_{n-l}(M)$ 
using the commutator formula \eqref{eq:NWcmt}. If says
\[
 a_{(l|L)}b_{(-1-k|K)}m = \pm b_{(-1-k|K)}a_{(l|L)}m + 
 \sum_{\substack{(j|J), \\ j \ge 0, \, J \supset L \cap K}} 
 \cst \cdot (a_{(j|J)}b)_{(l-1-k-j|J')}m
\]
with $J' \ceq K \cup (L \sm J)$.
As for the first term, by $n-k<n$, $l \ge 0$ and the induction hypothesis, we have 
$a_{(l|L)}m \in E_{n-k-l}(M)$.
Then, by $-1-k<0$ and the already-proved \eqref{eq:NW2.34}, we have
\[
 b_{(-1-k|K)}a_{(l|L)}m \in E_{(n-k-l)-(-1-k)-1}(M) = E_{n-l}(M).
\]
As for the terms in the summation, by \eqref{eq:NW2.34} and $j \ge 0$, we have
\begin{align*}
 (a_{(j|J)}b)_{(l-1-k-j|J')}w \in E_{(n-k)-(l-1-k-j)-1}(M) = E_{n-l+j}(M) 
 \subset E_{n-l}(M).
\end{align*}
Hence we have $a_{(l|L)}b_{(-1-k|K)}m \in E_{n-l}(M)$.
\end{proof}

We also have an analogue of \cite[Proposition 2.11]{Li}.
The proof is again similar to loc.\ cit., 
but since the argument uses the iterate formula (\cref{lem:NWiter}), we write down it.

\begin{lem}\label{lem:NW2.11}
Let $M$ be a $V$-module.  
For $u \in E_r(V)$, $m \in E_s(M)$, $l \in \bbZ$ and $L \subset [N]$, we have
\begin{align}\label{eq:NW2.36}
 u_{(l|L)}m \in E_{r+s-l-1}(M).
\end{align}
If moreover $l \ge 0$, then we have
\begin{align}\label{eq:NW2.37}
 u_{(l|L)}m \in E_{r+s-l}(M).
\end{align}
\end{lem}

\begin{proof}
We first consider the case $r \le 0$.
By \cref{lem:NW2.10}, we have $u_{(l|L)}m \in E_{s-l-1}(M) \subset E_{r+s-l-1}(M)$
for any $l \in \bbZ$, and $u_{(l|L)}w \in E_{s-l}(M) \subset E_{r+s-l}(M)$ for $l \ge 0$.
Thus, we have the conclusions.

Second, we show the case $r \ge 0$ by induction on $r$. 
Assume $u \in E_{r+1}(V)$. Then, by \cref{lem:NW2.89} \eqref{i:NW2.89:4}, we can write 
$u=a_{(-2-i|I)}b$ with some $a \in V$, $0 \le i \le r$, $I \subset [N]$ and $b \in E_{r-i}(V)$.
We calculate $(a_{(-2-i|I)}b)_{(l|L)}m$ using the iterate formula in \cref{lem:NWiter}.
It yields
\begin{align}\label{eq:NW2.11:iter}
\begin{split}
 (a_{(-2-i|I)}b)_{(l|L)}m
=\sum_{\substack{(j|J), \\ j \ge 0, \, J \subset I}} 
 \bigl(&\cst \cdot a_{(-2-i-j|J)}b_{(l+j|J')}m 
      + \cst \cdot b_{(l-2-i-j|J')} a_{(j|J)}m \bigr) 
\end{split}
\end{align}
with $J' \ceq L \cup (I \sm J)$ and $\cst$ being some constant.
We divide the argument by the sign of $l$.
\begin{itemize}[nosep]
\item Assume $l \ge 0$. For the first term in \eqref{eq:NW2.11:iter}, we have 
\begin{align*}
 a_{(-2-i-j|J)}b_{(l+j|J')}m
&\in a_{(-2-i-j|J)} E_{(r-i)+s-(l+j)}(M) \\
&\subset E_{(r+s-l-i-j)-(-2-i-j)-1}(M) = E_{(r+1)+s-l}(M),
\end{align*}
where in the first line we used the induction hypothesis with $l+j \ge 0$ and $r-i \le r$,
and in the second line we used \cref{lem:Li:2.10}. 

For the second term in \eqref{eq:NW2.11:iter},
the induction hypothesis and \cref{lem:NW2.10} imply 
\begin{align}\label{eq:NW2.11:2}
\begin{split}
 b_{(l-2-i-j|J')}a_{(j|J)}m
&\in b_{(l-2-i-j|J')} E_{s-j}(M) \\
&\subset E_{(r-i)+(s-j)-(l-2-i-j)-1}(M) = E_{(r+1)+s-l}(M).
\end{split}
\end{align}

Hence we have $(a_{(-2-i|I)}b)_{(l|L)}m \in E_{(r+1)+s-l}(M)$.

\item
Assume $l<0$. By the induction hypothesis and \cref{lem:NW2.10},
the first term in \eqref{eq:NW2.11:iter} is
\begin{align*}
 a_{(-2-i-j|J)}b_{(l+j|J')}m
&\in a_{(-2-i-j|J)} E_{(r-i)+s-(l+j)-1}(M) \\
&\subset E_{(r+s-l-i-j-1)-(-2-i-j)-1}(M) = E_{(r+1)+s-l-1}(M).
\end{align*}
For the second term in \eqref{eq:NW2.11:iter}, the argument \eqref{eq:NW2.11:2} yields
\begin{align*}
 b_{(l-2-i-j|J')}a_{(j|J)}m \in E_{(r+1)+s-l}(M) \subset E_{(r+1)+s-l-1}(M).
\end{align*}
Hence we have $(a_{(-2-i|I)}b)_{(l|L)}m \in E_{(r+1)+s-l-1}(M)$.
\end{itemize}
By these arguments, the induction step works, and we have the conclusion.
\end{proof}

Now \cref{prp:NW2.6} and \cref{lem:NW2.11} yield 
the following analogue of \cite[Theorem 2.12]{Li}:

\begin{thm}\label{thm:NW2.12}
Let $V=(V,\vac,T,S^i,Y)$ be an $N_W=N$ SUSY VA, and $E_n = E_n(V)$ for $n \in \bbZ$
be the linear sub-superspaces of $V$ in \cref{dfn:NWLif},
which form a decreasing filtration 
\[
 \dotsb = E_{-1} = E_0 = V \supset E_1 \supset E_2 \supset \dotsb \supset E_n \supset \dotsb
\]
by \cref{lem:NW2.89} \eqref{i:NW2.89:1}.
Then the associated graded space 
\[
 \gr_E V = \tboplus_{n \in \bbN} E_n/E_{n+1}
\]
has the following structure $(\cdot,1,\pd,\delta^i,Y_-)$ of an $N_W=N$ SUSY VPA.
Let $a \in E_r$ and $b \in E_s$.
\begin{itemize}[nosep]
\item 
The commutative multiplication $\cdot$ is 
\[
 (a+E_{r+1}) \cdot (b+E_{s+1}) \ceq a_{(-1|N)}b+E_{r+s+1}.
\] 

\item
The unit is $1 \ceq \vac+E_1$.

\item
The even operator $\pd$ is $\pd(a+E_{r+1}) \ceq T a + E_{r+2}$.

\item
The odd operator $\delta^i$ for $i \in [N]$ is
$\delta^i(a+E_{r+1}) \ceq S^i a + E_{r+2}$.

\item
The SUSY VLA structure $Y_-$ is 
\[
 Y_-(a+E_{r+1},Z)(b+E_{s+1}) \ceq 
 \sum_{(j|J), \, j \ge 0} Z^{-1-j|N \sm J} (a_{(j|J)}b+E_{r+s-j+1}).
\]
\end{itemize}
\end{thm}

Similarly as in \cite[Proposition 2.13]{Li}, we also have the module structure 
of the associated graded space of an $N_W=N$ SUSY VA module. 
The proof is straightforward, and we omit the detail.

\begin{prp}\label{prp:NWgrE-mod}
Let $V$ be as in \cref{thm:NW2.12}, $M$ be a $V$-module, 
and $\{E_n (W) \mid n \in \bbZ\}$ be the decreasing filtration in \cref{dfn:NWLif}.
Then the associated graded $\gr_E M = \bigoplus_{n \in \bbN}E_n(M)/E_{n+1}(M)$
is a module over the $N_W=N$ SUSY VPA $\gr_E V$ in \cref{thm:NW2.12}
with the following structure for $a \in E_r$ and $m \in E_s(M)$:
\begin{itemize}[nosep]
\item 
The module structure over commutative superalgebra is
\[
 (a+E_{r+1}).(m+E_{s+1}(M)) \ceq a_{(-1|N)}m+E_{r+s+1}(M).
\] 
\item
The module structure over SUSY vertex Lie algebra (see \eqref{eq:VLact}) is
\[
 Y_-^M(a+E_{r+1},Z)(m+E_{s+1}(M)) \ceq
 \sum_{(j|J), \, j \ge 0} Z^{-1-j|N \sm J} \bigl(a_{(-j|J)}m+E_{r+s-j+1}(M)\bigr).
\] 
\end{itemize}
\end{prp}

\subsection{Li filtration of $N_K=N$ SUSY vertex algebra}\label{ss:NKLif}

Throughout this subsection, $V=(V,\vac,S_K^i,Y)$ denotes an $N_K=N$ SUSY vertex algebra.
We begin with an $N_K=N$ analogue of \cref{prp:NW2.6}:

\begin{prp}\label{prp:Li:2.6}
Let $E=\{E_n \mid n \in \bbN\}$ be a decreasing filtration of linear sub-superspaces 
of $V$ such that $\vac \in E_0$ and 
\begin{align}\label{eq:Li}
 a_{(j|J)}b \in E_{r+s-j-1} \qquad 
 (a \in E_r, \, b \in E_s, \, j \in \bbZ, \, J \subset [N]), 
\end{align}
where we used the convention $E_n \ceq V$ for $n \in \bbZ_{<0}$.
\begin{enumerate}[nosep]
\item \label{i:Li:1}
The associated graded linear superspace 
\[
 \gr_E V \ceq \tboplus_{n \in \bbN} E_n/E_{n+1}
\]
is an $N_K=N$ SUSY VA whose state-superfield correspondence is given by 
\begin{align}\label{eq:Li:op}
 (a+E_{r+1})_{(j|J)}(b+E_{s+1}) \ceq a_{(j|J)}b+E_{r+s-j}
\end{align}
for $a \in E_r$, $b \in E_s$, $j \in \bbZ$ and $J \subset [N]$,
whose vacuum is given by $\vac+E_1 \in E_0/E_1$, 
and whose odd operators $\sld^i$ are given by 
\[
 \sld^i(a+E_{r+1}) \ceq S_K^i a + E_{r+2}.
\]

\item
The $N_K=N$ SUSY VA $\gr_E V$ in \eqref{i:Li:1} is commutative if and only if 
\begin{align}\label{eq:Li:com}
 a_{(m|M)}b \in E_{r+s-m}
\end{align}
for any $a \in E_r$, $b \in E_s$, $m \in \bbN$ and $M \subset [N]$, 

\item \label{i:Li:2.6:3}
Under the condition \eqref{eq:Li:com}, 
the commutative $N_K=N$ SUSY VA $\gr_E V$ has an $N_K=N$ SUSY VPA structure 
\begin{align}\label{eq:Li:Y-}
 Y_-(a+E_{r+1},Z)(b+E_{s+1}) \ceq 
 \sum_{(m|M), \, m \ge 0}Z^{-m-1|N \sm M}(a_{(m|M)}b+E_{r+s-m+1})
\end{align}
for $a \in E_r$ and $b \in E_s$.
\end{enumerate}
\end{prp}

\begin{proof}
Similarly as in \cref{prp:NW2.6}, it is enough to check that \eqref{i:Li:2.6:3} 
satisfies the derivation axiom (see \eqref{eq:VPA}), which is equivalent to
\begin{align*}
 a_{(l|L)}(b_{(-1|N)}c) + E_{r+s+t-m+1} = (a_{(l|L)}b)_{(-1|N)}c +
 (-1)^{(p(a)+\abs{L})p(b)} b_{(-1|N)}(a_{(l|L)}c) + E_{r+s+t-m+1}
\end{align*}
for $l \in \bbN$, $L \subset [N]$, $a \in E_r$, $b \in E_s$ and $c \in E_t$.
The commutator formula \eqref{eq:NKcmt} yields
\begin{align*}
&a_{(l|L)}(b_{(-1|N)}c) - (-1)^{(p(a)+\abs{L})p(b)} b_{(-1|N)}(a_{(l|L)}c)  \\
&= \sum_{\substack{(j|J), \, j \ge 0, \\ [N] \cap (J \Delta L) = \emptyset }} 
  \pm \frac{1}{j!} \df{l}{j+\abs{J \sm L}} (a_{(j|J)}b)_{(l-1-j-\# J \sm L |N)}c \\
&=(a_{(l|L)}b)_{(-1|N)}c +
  \sum_{\substack{(j|J), \, 0 \le j \le l-1, \\ J \Delta L = \emptyset}} 
  \pm \frac{1}{j!} \df{l}{j+\abs{J \sm L}} (a_{(j|J)}b)_{(l-1-j|N)}c,
\end{align*}
where $\pm$ denotes some sign.
For each term of this equality, the condition \eqref{eq:Li:com} yields
\[
 a_{(l|L)}(b_{(-1|N)}c), \, b_{(-1|N)}(a_{(l|L)}c), \, (a_{(l|L)}b)_{(-1|N)}c 
 \in E_{r+s+t-l},
\]
and since $l-1-j \ge 0$ in the summation, it also yields
\begin{align*}
 (a_{(j|J)}b)_{(l-1-j)}c \in E_{(r+s-j)+t-(l-1-j)} = E_{r+s+t-l+1}.
\end{align*}
Hence we have the desired equality.
\end{proof}

The definition of Li filtration for an $N_K=N$ SUSY VA is the same as 
that for $N_K=N$ case.

\begin{dfn}\label{dfn:NKLif}
For a $V$-module $M$ and $n \in \bbZ$, 
we define a linear sub-superspace $E_n(M) \subset W$ by
\[
 E_n(M) \ceq \spn\left\{ a^1_{(-1-k_1|K_1)} \dotsm a^r_{(-1-k_r|K_r)} m \, \Bigg|
 \begin{array}{l}
  r \in \bbZ_{>0}, \, a^i \in V, \, m \in M, \, k_i \in \bbN, \, K_i \subset [N] \\
  \text{satisfying $k_1+\dotsb+k_r \ge n$}
 \end{array} \right\}.
\]
\end{dfn}

The following is an $N_K=N$ analogue of \cref{lem:NW2.89}.

\begin{lem}\label{lem:NK2.89}
For a $V$-module $W$, we have the following.
\begin{enumerate}[nosep]
\item \label{i:NK2.89:1}
$E_n(M) \supset E_{n+1}(M)$ for any $n \in \bbZ$.

\item \label{i:NK2.89:2}
$E_n(M) = M$ for any $n \in \bbZ_{\le 0}$

\item \label{i:NK2.89:3}
$a_{(-1-k|K)}E_n(M) \subset E_{n+k}(M)$ 
for any $a \in V$, $k \in \bbN$, $K \subset [N]$ and $n \in \bbZ$.

\item \label{i:NK2.89:4}
For $n \ge 1$, $E_n(M)$ is equal to 
\begin{align}\label{eq:Li:2.32}
 E_n'(M) \ceq \spn\left\{ a_{(-1-k|K)}m \mid 
  \text{ $a \in V$, $k \in \bbZ_{>0}$, $K \subset [N]$, $m \in E_{n-k}(M)$} \right\}.
\end{align}

\item \label{i:NK2.89:5}
For $n \ge 1$, $E_n(M)$ is equal to 
\begin{align}\label{eq:Li:2.33}
 E_n''(M) \ceq \spn\left\{ a^1_{(-1-k_1|K_1)} \dotsm a^r_{(-1-k_r|K_r)} m \, \Bigg|
 \begin{array}{l}
  r \in \bbZ_{>0}, \, a^i \in V, \, m \in M, \, k_i \in \bbZ_{\ge1}, \, K_i \subset [N] \\
  \text{satisfying $k_1+\dotsb+k_r \ge n$}
 \end{array} \right\}.
\end{align}
\end{enumerate}
\end{lem}

\begin{proof}
Similarly as in \cref{lem:NW2.89}, we only show \eqref{i:NK2.89:4} and \eqref{i:NK2.89:5}.
We can show by induction on $n$ that 
\begin{align}\label{eq:Lif:4=5}
 E_n'(M) = E_n''(M),
\end{align}
so that it is enough to prove \eqref{i:NK2.89:4} only. For that, we show that each element 
\[
 u = a^1_{(-1-k_1|K_1)} \dotsm a^r_{(-1-k_r|K_r)} m \in E_n(M)
\]
belongs to $E_n'(M)$ by induction on the length $r$.
The $r=1$ case is the same as in \cref{lem:NW2.89}. So let us assume $r \ge 2$. We set 
\[
 u = a^1_{(-1-k_1|K_1)}u', \quad u' \ceq a^2_{(-1-k_2|K_2)} \dotsm a^r_{(-1-k_r|K_r)} m.
\]
If $k_1 \ge 1$, then $u' \in E_{n-k_1}(M)$, and we have
$u =a^1_{(-1-k_1|K_1)}u'\in E_n'(M)$ as desired.
Hereafter we assume $k_1=0$.
Then $k_2+\dotsb+k_r \ge n$ and $u' \in E_n(M)$, 
which yields $u' \in E_n'(M)$ by induction hypothesis.
Now the equality \eqref{eq:Lif:4=5} yields $u' \in E_n''(M)$.
Hence we may assume $k_2,\dotsc,k_r \ge 1$.
Let us rewrite $a \ceq a^1$, $b \ceq a^2$ and $k \ceq k_2$, so that we have 
\[
 u = a_{(-1|K_1)}b_{(-1-k|K_2)}u'', \quad 
 u'' \ceq a^3_{(-1-k_3|K_3)} \dotsm a^r_{(-1-k_r|K_r)} m \in E_{n-k}(M). 
\]
The commutator formula \eqref{eq:NKcmt} yields
\begin{align*}
 u=\pm b_{(-1-k|K_2)}a_{(-1|K_1)}u'' 
  +\sum_{(j|J), \, j \ge 0} \cst \cdot (a_{(j|J)}b)_{(-2-k-j-\# K|K')}u'',
\end{align*}
where $\pm$ denotes a sign, $\cst$ denotes a constant, $K \ceq K_1 \sm J$ 
and $K' \ceq K_2 \cup (K_1 \Delta J)$.
Since $u'' \in E_{n-k}(M)$, we have $a_{(-1|K_1)}u'' \in E_{n-k}(M)$ by \eqref{i:NK2.89:3}, 
and the condition $k \ge 1$ yields 
\[
 b_{(-1-k|K_2)}a_{(-1|K_1)}u'' \in E_n'(M). 
\]
Also, $u'' \in E_{n-k}(M) \subset E_{n-(k+j+1+\# K))}(M)$ for $j \ge 0$ 
by \eqref{i:NK2.89:1}, and we have 
\[
 (a_{(j|J)}b)_{(-2-k-j-\# K \sm|K')}u'' \in E_n'(M).
\]
Hence we have $u \in E_n'(W)$.
\end{proof}

The next lemma is an analogue of \cref{lem:NW2.10}.

\begin{lem}\label{lem:Li:2.10}
Let $M$ be a $V$-module. For $a \in V$, $l,n \in \bbZ$ and $L \subset [N]$, we have
\begin{align}\label{eq:Li:2.34}
 a_{(l|L)}E_n(M) \subset E_{n-l-1}(M).
\end{align}
If moreover $l \ge 0$, then we have 
\begin{align}\label{eq:Li:2.35}
 a_{(l|L)}E_n(M) \subset E_{n-l}(M).
\end{align}
\end{lem}

\begin{proof}
Similarly as \cref{lem:NW2.10}, we show \eqref{eq:Li:2.35} by induction on $n$.
The case $n \le 0$ is done in the same way as \cref{lem:NW2.10}. So assume $n \ge 1$.
By \cref{lem:NK2.89} \eqref{i:NK2.89:4}, $E_n(M)$ is spanned by elements
$b_{(-1-k|K)}m$ with $b \in V$, $k \ge 1$, $K \subset [N]$ and $m \in E_{n-k}(M)$.
We will show that $a_{(l|L)}b_{(-1-k|K)}m$ belongs to $E_{n-l}(M)$ 
using the commutator formula \eqref{eq:NKcmt}. If says
\[
 a_{(l|L)}b_{(-1-k|K)}m = \pm b_{(-1-k|K)}a_{(l|L)}m + 
 \sum_{(j|J), \, j \ge 0} \cst \cdot (a_{(j|J)}b)_{(l-1-k-j-\# J'|J'')}m
\]
with $J' \ceq L \sm J$ and $J'' \ceq K \cup (J \Delta L)$.
As for the first term, by $n-k<n$, $l \ge 0$ and the induction hypothesis, 
we have $a_{(l|L)}m \in E_{n-k-l}(M)$.
Then, by $-1-k<0$ and the already-proved \eqref{eq:Li:2.34}, we have
\[
 b_{(-1-k|K)}a_{(l|L)}m \in E_{(n-k-l)-(-1-k)-1}(M) = E_{n-l}(M).
\]
As for the terms in the summation, by \eqref{eq:Li:2.34} and $j+\# J' \ge 0$, we have
\begin{align*}
 (a_{(j|J)}b)_{(l-1-k-j-\# J'|J'')}m \in E_{(n-k)-(l-1-k-j-\# J')-1}(M)
 = E_{n-l+j+\# J'}(M) \subset E_{n-l}(M).
\end{align*}
Hence we have $a_{(l|L)}b_{(-1-k|K)}m \in E_{n-l}(M)$.
\end{proof}

The next lemma is an $N_K=N$ analogue of \cref{lem:NW2.11}.

\begin{lem}\label{lem:Li:2.11}
Let $M$ be a $V$-module.  
For $u \in E_r(V)$, $m \in E_s(M)$, $l \in \bbZ$ and $L \subset [N]$, we have
\begin{align}\label{eq:Li:2.36}
 u_{(l|L)}m \in E_{r+s-l-1}(M).
\end{align}
If moreover $m \ge 0$, then we have
\begin{align}\label{eq:Li:2.37}
 u_{(l|L)}m \in E_{r+s-l}(M).
\end{align}
\end{lem}

\begin{proof}
As in \cref{lem:NW2.11}, the non-trivial case is $r \ge 0$, 
which we show by induction on $r$. 
Assume $u \in E_{r+1}(M)$. By \cref{lem:NK2.89} \eqref{i:NK2.89:4}, we can write 
$u=a_{(-2-i|I)}b$ with some $a \in V$, $0 \le i \le r$, $I \subset [N]$ and $b \in E_{r-i}(V)$.
Thus we want to calculate $(a_{(-2-i|I)}b)_{(l|L)}m$,
which is, by the iterate formula in \cref{lem:NKiter}, equal to  
\begin{align}\label{eq:2.11:iter}
\begin{split}
 (a_{(-2-i|I)}b)_{(l|L)}m
=\sum_{(j|J), \, j \ge 0} \sum_{K \subset I \cap J} 
 \bigl(&\cst \cdot a_{(-2-i-j-x|J)}b_{(l+j|J')}m \\
      +&\cst \cdot b_{(l-2-i-j-x|J')} a_{(j|J)}m \bigr) 
\end{split}
\end{align}
with $x \ceq \abs{J \sm K}$, $J' \ceq L \cup (I \sm K) \cup (J \sm K)$ 
and $\cst$ being some constant.
\begin{itemize}[nosep]
\item Assume $l \ge 0$. Then, the first term in \eqref{eq:2.11:iter} satisfies
\begin{align*}
&a_{(-2-i-j-x|J)}b_{(l+j|J')}m
 \in a_{(-2-i-j-x|J)} E_{(r-i)+s-(l+j)}(M) \\
&\subset E_{(r+s-l-i-j)-(-2-i-j-x)-1}(M) = E_{(r+1)+s-l+x}(M) 
 \subset E_{(r+1)+s-l}(M),
\end{align*}
Here, in the first line we used the induction hypothesis with $l+j \ge 0$ and $r-i \le r$,
and in the second line we used \cref{lem:Li:2.10}. 

For the second term in \eqref{eq:2.11:iter},
the induction hypothesis and \cref{lem:Li:2.10} imply 
\begin{align}\label{eq:2.11:2}
\begin{split}
&b_{(l-2-i-j-x|J')}a_{(j|J)}m
 \in b_{(l-2-i-j-x|J')} E_{s-j}(M) \\
&\subset E_{(r-i)+(s-j)-(l-2-i-j-x)-1}(M) = E_{(r+1)+s-l+x}(M) 
 \subset E_{(r+1)+s-l}(M).
\end{split}
\end{align}

Hence we have $(a_{(-2-i|I)}b)_{(l|L)}m \in E_{(r+1)+s-l}(M)$.

\item
Assume $l<0$. By the induction hypothesis and \cref{lem:Li:2.10},
the first term in \eqref{eq:2.11:iter} is
\begin{align*}
&a_{(-2-i-j-x|J)}b_{(l+j|J')}m 
 \in a_{(-2-i-j-x|J)} E_{(r-i)+s-(m+j)-1}(M) \\
&\subset E_{(r+s-l-i-j-1)-(-2-i-j-x)-1}(M) = E_{(r+1)+s-l-1+x} 
 \subset E_{(r+1)+s-l-1}(M).
\end{align*}
For the second term in \eqref{eq:2.11:iter}, the argument \eqref{eq:2.11:2} yields
\begin{align*}
 b_{(l-2-i-j-x|M')}a_{(j|J)}m \in E_{(r+1)+s-m}(M) \subset E_{(r+1)+s-m-1}.
\end{align*}
Hence we have $(a_{(-2-i|I)}b)_{(l|L)}m \in E_{(r+1)+s-l-1}(M)$.
\end{itemize}
By these arguments, the induction step works, and we have the conclusion.
\end{proof}

By \cref{prp:Li:2.6} and \cref{lem:Li:2.11}, 
we have the following $N_K=N$ analogue of \cref{thm:NW2.12}.

\begin{thm}\label{thm:NK2.12}
Let $V=(V,\vac,S_K^i,Y)$ be an $N_K=N$ SUSY VA, and $\{E_n = E_n(V) \mid n \in \bbZ\}$
be the decreasing filtration in \cref{dfn:NKLif}.
Then the associated graded space $\gr_E V = \bigoplus_{n \in \bbN} E_n/E_{n+1}$
has the following structure $(\cdot,1,\sld^i,Y_-)$ of an $N_K=N$ SUSY VPA.
Let $a \in E_r$ and $b \in E_s$.
\begin{itemize}[nosep]
\item 
The commutative multiplication $\cdot$ is 
$(a+E_{r+1}) \cdot (b+E_{s+1}) \ceq a_{(-1|N)}b+E_{r+s+1}$. 

\item
The unit is $1 \ceq \vac+E_1$.

\item
The odd operator $\sld^i$ for $i \in [N]$ is $\sld^i(a+E_{r+1}) \ceq S_K^i a + E_{r+2}$.

\item
The SUSY VLA structure $Y_-$ is 
\[
 Y_-(a+E_{r+1},Z)(b+E_{s+1}) \ceq 
 \sum_{(j|J), \, j \ge 0} Z^{-1-j|N \sm J}(a_{(j|J)}b+E_{r+s-j+1}).
\]
\end{itemize}
\end{thm}

We also have an analogue of \cref{prp:NWgrE-mod}. The proof is omitted.

\begin{prp}\label{prp:NKgrE-mod}
Let $V$ be as in \cref{thm:NK2.12}, $M$ be a $V$-module, 
and $\{E_n (M) \mid n \in \bbZ\}$ be the decreasing filtration in \cref{dfn:NKLif}.
Then the associated graded $\gr_E M=\bigoplus_{n \in \bbN}E_n(M)/E_{n+1}(M)$
is a module over the $N_K=N$ SUSY VPA $\gr_E V$ in \cref{thm:NK2.12}
with the following structure for $a \in E_r$ and $m \in E_s(M)$:
\begin{itemize}[nosep]
\item 
The module structure over commutative superalgebra is
\[
 (a+E_{r+1}).(m+E_{s+1}(M)) \ceq a_{(-1|N)}m+E_{r+s+1}(M).
\] 
\item
The module structure over SUSY vertex Lie algebra (see \eqref{eq:VLact}) is
\[
 Y_-^M(a+E_{r+1},Z)(m+E_{s+1}(M)) \ceq
 \sum_{(j|J), \, j \ge 0} Z^{-1-j|N \sm J} \bigl(a_{(j|J)}m+E_{r+s-j+1}(M)\bigr).
\] 
\end{itemize}
\end{prp}

\section{Associated superschemes and singular supports for SUSY vertex algebras}
\label{s:ss}

In this section, $V$ denotes an $N_W=N$ or $N_K=N$ SUSY vertex algebra
over the field $k$ of characteristic $0$.

\subsection{$C_2$-Poisson superalgebra and associated superscheme}\label{ss:C2}

For the even case, Li showed in \cite[Proposition 3.8]{Li} that the degree zero subspace 
$V/E_1(V) \subset \gr_E V$ of the Li filtration is a Poisson algebra 
which coincides with the one introduced by Zhu \cite[\S 4.4]{Z}.
Following \cite[\S 2.3.1, p.11614]{A15}, we call this Poisson algebra
\emph{Zhu's $C_2$-Poisson algebra of $V$}.

There are several notions of finiteness condition on vertex algebras.
One of them is the \emph{$C_2$-cofiniteness},
which guarantees the modular property for a vertex operator algebra \cite{Z,Mi}.
Though the $C_2$-cofiniteness looks a technical condition,
Arakawa illustrated in \cite{A12} that it has a clear geometric meaning,
by introducing the notion of \emph{associated variety}.

In this subsection, we give SUSY analogue of $C_2$-Poisson algebras, 
and of the theory of associated varieties.

\subsubsection{$C_2$-Poisson superalgebra of SUSY vertex algebra}

Let $V$ be an $N_W=N$ or $N_K=N$ SUSY VA. For a $V$-module $M$, 
We denote by $\{E_n(M) \mid n \in \bbZ\}$ the Li filtration of $M$,
given in \cref{dfn:NWLif,dfn:NKLif}.
We also abbreviate $E_n \ceq E_n(V)$ as before.

\begin{dfn}\label{dfn:C2cof}
For a $V$-module $M$, we denote 
\[
 C_2(M) \ceq \spn\{a_{(-j|J)}m \mid a\in V, \, m \in M, \, \, j \ge 2, \, J \subset [N]\},
\] 
and call $M$ \emph{$C_2$-cofinite} if $\dim_k M/C_2(M) M < \infty$.
We also say \emph{$V$ is $C_2$-cofinite} if it is $C_2$-cofinite as a $V$-module.
\end{dfn}

By \cref{lem:NW2.89} \eqref{i:NW2.89:4} and \cref{lem:NK2.89} \eqref{i:NK2.89:4}, we have
\[
 C_2(M) = E_1(M),
\]
which yields the first half of \cref{prp:C2} below.

\begin{prp}\label{prp:C2}
For an $N_W=N$ or $N_K=N$ SUSY VA $V$, the quotient space
\[
 R_V \ceq V/C_2(V)
\]
is equal to the degree zero subspace $E_0(V)/E_1(V) = V/E_1(V) \subset \gr_E V$
of the Li filtration.
Moreover, it is a Poisson superalgebra of parity $N \bmod 2$ in the sense of \cref{dfn:sP}
whose commutative multiplication $\cdot$ 
and Poisson bracket $\{\cdot,\cdot\}$ are given by 
\[
 \ol{a} \cdot \ol{b} \ceq \ol{a_{(-1|N)}b}, \quad 
 \{\ol{a},\ol{b}\} \ceq \ol{a_{(0|N)}b}
\]
for $\ol{a} \ceq a+C_2(V)$, $\ol{b} \ceq b+C_2(V)$ with $a,b \in V$.
We call $R_V$ the \emph{$C_2$-Poisson superalgebra of $V$}.
\end{prp}

\begin{proof}
The first half is already explained.
The second half can be shown directly using \cref{dfn:NWP,dfn:NKP} of SUSY VPAs.
\end{proof}

\begin{eg}\label{eg:NS-RV}
Let us study the $C_2$-Poisson superalgebra of 
the Neveu-Schwarz SUSY vertex algebra $V$ (\cref{eg:NS2}).
Recall that it is an $N_K=1$ SUSY VA over $k=\bbC$, and as a linear superspace we have 
\[
 V=\bbC[S_K^n \tau \mid n \in \bbN]
\]
with $\tau=G_{-\frac{3}{2}}\vac$ being the Neveu-Schwarz element. 
Then  $E_1(V) = C_2(V) = 
 \spn\{a_{(-m|M)}b \mid a,b \in V, \, m \in \bbZ_{\ge 2}, \, M \in\{ 0,1\} \}$,
and by \eqref{eq:NK-TSa}, we have
$S_K^m \tau = (S_K^m \tau)_{(-1|1)}\vac = (S_K^{m-2} \tau)_{(-2|1)} \vac \in E_1(V)$
for $m \ge 2$. Also, by \eqref{eq:Li:2.36}, 
we have $E_1(V)_{(-1|1)}E_1(V) \subset E_2(V) \subset E_1(V)$.
These imply $E_1(V) = \bbC[S_K^m \tau \mid m \in \bbZ_{\ge 2}]$, and we have
\[
 R_V \simeq \bbC[\ol{\tau},\ol{S_K \tau}] = \bbC[\ol{\tau},\ol{\nu}]
\]
as commutative superalgebras, 
where $\nu \ceq \frac{1}{2}S_K \tau$ is the Virasoro element of $V$.
The Poisson structure is given by 
\[
 \{\ol{\tau},\ol{\tau}\} = 2 \ol{\nu}, \quad 
 \{\ol{\tau},\ol{\nu}\} = \{\ol{\nu},\ol{\nu}\} = 0.
\]
Here is the calculation:
Using \eqref{eq:tau_op}, \eqref{eq:NS} and \eqref{eq:tau-nu}, we have 
$\tau_{(0|0)}\tau=G_{-\hf}G_{-\frac{3}{2}}\vac=2L_{-2}\vac=2\nu$,
which gives the first equality.
Similarly, we have 
$\tau_{(0|1)}\nu = G_{-\hf}L_{-2}\vac = \hf G_{-\frac{5}{2}}\vac = \hf \tau_{(-2|1)}\vac$,
which belongs to $C_2(V)=E_1(V)$. Hence $\{\ol{\tau},\ol{\nu}\}=0$.
Finally, we have 
$\nu_{(0|1)}\nu = L_{-1}L_{-2}\vac = L_{-3}\vac = \hf \tau_{(-2|0)}\vac$,
which belongs to $E_1(V)$, implying $\{\ol{\nu},\ol{\nu}\}=0$.
\end{eg}

In the even case, Arakawa showed in \cite[Proposition 2.5.1]{A12} that 
the embedding $R_V = E_0/E_1 \inj \gr_E V=\bigoplus_{n \ge 0} E_n/E_{n+1}$
can be extended to a surjection $(R_V)_{\infty} \srj \gr_E V$ of vertex Poisson algebras,
where $(R_V)_{\infty}$ is the jet algebra of $R_V$ equipped with 
the level $0$ vertex Poisson structure (\cref{rmk:lv0}).
We have the following SUSY analogue of this fact.

\begin{prp}\label{prp:A12:2.5.1}
Let $V$ be a SUSY VA, and $\phi\colon R_V \inj \gr_E V$ be the embedding in \cref{prp:C2}.
\begin{enumerate}[nosep]
\item 
If $V$ is an $N_W=N$ SUSY VA, then $\phi$ extends to a surjective morphism 
\[
 \Phi\colon (R_V)^O \lsrj \gr_E V
\]
of $N_W=N$ SUSY VPAs, where $(R_V)^O$ is the $1|N$-superjet algebra of $R_V$
equipped with the level $0$ SUSY VPA structure (\cref{prp:NWlv0}).
\item
If $V$ is an $N_K=N$ SUSY VA, then $\phi$ extends to a surjective morphism 
\[
 \Phi\colon (R_V)^{\Osc} \lsrj \gr_E V
\]
of $N_K=N$ SUSY VPAs, where $(R_V)^{\Osc}$ is the $1|N$-superconformal jet algebra of $R_V$
with the level $0$ SUSY VPA structure (\cref{prp:NKlv0}).
\end{enumerate}
\end{prp}

For the proof, we need some preliminary. 
The following is an analogue of \cite[Lemma 4.2, (4.7)]{Li}.

\begin{lem}\label{lem:Li:4.2}
Let $V$ be a SUSY VA. We regard the SUSY VPA $\gr_E V$ in \cref{thm:NW2.12,thm:NK2.12} 
as a commutative superalgebra, and denote it by $A$.
Hence, for an $V$-module $M$, the $\gr_E V$-module $\gr_E M$ 
in \cref{prp:NWgrE-mod,prp:NKgrE-mod} can be regarded as an $A$-module.
Then the $A$-module $\gr_E M$ is generated by $E_0(M)/E_1(M)$.
\end{lem}

\begin{proof}
Considering the decomposition $\gr_E M = \bigoplus_{n \in \bbN}E_n(M)/E_{n+1}(M)$,
we show by induction on $n$ that every element of $E_n(M)/E_{n+1}(M)$ can be written 
in the form $u.m$ with $u \in A$ and $m \in E_0(M)/E_1(M)$.
We may assume $n \ge 1$. The rest argument is divided into $N_W=N$ case and $N_K=N$ case.
\begin{itemize}[nosep]
\item
Assume $V=(V,\vac,T,S^i,Y)$ is an $N_W=N$ SUSY VA.
Then, by \cref{lem:NW2.89} \eqref{i:NW2.89:4}, 
$E_n(M)$ is linearly spanned by the subspaces $a_{(-2-j|J)}E_{n-1-j}(M)$ 
with $a \in V$, $0 \le j \le n-1$ and $J \subset [N]$. For $m \in E_{n-1-j}(M)$, 
using \eqref{eq:NW-TSa} and $T^{(j)} \ceq \frac{1}{j!}T^j$, we have
\begin{align*}
 a_{(-2-j|J)}m + E_{n+1}(M) 
&=\pm (S^{N \sm J}a)_{(-2-j|N)}m+E_{n+1}(M) \\
&=\pm (T^{(j+1)}S^{N \sm J}a)_{(-1|N)}m+E_{n+1}(M) \\
&=\pm\bigl(T^{(j+1)}S^{N \sm J}a+E_{j+2}(V)\bigr)_{(-1|N)}\bigl(m+E_{n-j}(M)\bigr)+E_{n+1}(M).
\end{align*}
Denoting $\ol{m} \ceq m+E_n(M)$ for $m \in E_n(M)$, we have 
\[
 \ol{a_{(-2-j|J)}m} = \pm \bigl(\pd^{j+1}\delta^{N \sm J}\ol{a}\bigr).\ol{m},
\]
where $\pd$ and $\delta^i$ are the operators of the $N_W=N$ SUSY VPA $\gr_E V$
(see \cref{thm:NW2.12}).
It implies that $\ol{a_{(-2-j|J)}m}$ belongs to the subspace generated by $E_0(M)/E_1(M)$.
Thus, the induction step works, and we have the result.

\item 
A similar argument works in the case $V=(V,\vac,S_K^i,Y)$ is an $N_W=N$ SUSY VA.
By \cref{lem:NK2.89} \eqref{i:NK2.89:4}, 
$E_n(M)$ is linearly spanned by the subspaces $a_{(-2-j|J)}E_{n-1-j}(M)$ 
with $a \in V$, $0 \le j \le n-1$ and $J \subset [N]$. 
For $m \in E_{n-1-j}(M)$, using \eqref{eq:NK-TSa}, we have 
\begin{align*}
 a_{(-2-j|J)}m + E_{n+1}(M) = \pm\bigl(T^{(j+1)}S_K^{N \sm J}a+E_{j+2}(V)\bigr)_{(-1|N)}
 \bigl(m+E_{n-j}(M)\bigr)+E_{n+1}(M).
\end{align*}
Using the odd operator $\sld^i$ of the $N_K=N$ SUSY VPA $\gr_E V$ (\cref{thm:NK2.12}),
we have
\[
 \ol{a_{(-2-j|J)}m} = \pm \bigl((\sld^i)^{(2j+2)}\sld^{N \sm J}\ol{a}\bigr).\ol{m},
\]
The rest part is the same as the $N_W=N$ case.
\end{itemize}
\end{proof}

Now we start:

\begin{proof}[Proof of \cref{prp:A12:2.5.1}]
First we consider the $N_W=N$ case.
By \cref{lem:sj-gen,lem:Li:4.2}, we find that the embedding $R_V \inj \gr_E V$ 
extends to a surjection $\Phi\colon (R_V)^O \srj \gr_E V$ of differential superalgebras.
By \cref{thm:NW2.12}, the $N_W=N$ SUSY VPA structure on $\gr_E V$ is restricted to 
the sub-superspace $R_V=E_0/E_1$ as 
\[
 \ol{a}_{(m|M)}\ol{b} = 
 \begin{cases}\ol{a_{(0|N)}b} & (m = 0, M=[N]) \\ 0 & (\text{otherwise}) \end{cases}
\]
for $m \in \bbN$, $M \subset [N]$ and $\ol{a}=a+E_1$, $\ol{b}=b+E_1$ with $a,b \in E_0=V$.
It coincides with the level $0$ SUSY VPA structure on $(R_V)^O$ (\cref{prp:C2}).
Then, an obvious SUSY analogue of \cite[Lemma 3.3]{Li04} shows
that the surjection $\Phi$ is a morphism of $N_K=N$ SUSY VPAs.

The $N_K=N$ case can be treated similarly, 
using \cref{lem:scj-gen} and \cref{thm:NK2.12} instead.
\end{proof}

\subsubsection{Associated superscheme of SUSY vertex algebra}

Following \cite[\S 3.1, \S 3.2]{A12}, we give an analogue of the theory of associated varieties.
As before, $V$ denotes an $N_W=N$ or $N_K=N$ SUSY VA.

\begin{dfn}\label{dfn:X_V}
The \emph{associated superscheme} of $V$ is defined to be
\[
 X_V \ceq \Spec R_V.
\]
By \cref{prp:C2}, it is a Poisson superscheme of parity $N \bmod 2$ (see \cref{dfn:sP}).
\end{dfn}

Let us study a condition when $R_V$ is finitely generated as a superalgebra
(so that $X_V$ is of finite type).

\begin{dfn}\label{dfn:fsg}
$V$ is called \emph{finitely strongly generated} if there exists a finite subset $G \subset V$ 
such that $V$ is linearly spanned by the elements of the form
\begin{align}\label{eq:fsg}
 a^1_{(-p_1|P_1)} \dotsm a^r_{(-p_r|P_r)}\vac \quad
 (r \in \bbN, \, a^i \in G, \, p_i \in \bbZ_{>0}, \, P_i \subset [N]).
\end{align}
We call $G$ a set of \emph{strong generators} of $V$.
\end{dfn}

Below is an analogue of a half of \cite[Corollary 2.6.2]{A12}.

\begin{lem}\label{lem:fsg}
If $V$ is finitely strongly generated, 
then the $C_2$-superalgebra $R_V$ is finitely generated.
More precisely, if $G=\{a^1,\dotsc,a^r\}$ is a set of strong generators of $V$, 
then $\ol{G} \ceq \{\ol{a^1},\dotsc,\ol{a^r}\}$ with $\ol{a^i} \ceq a^i + C_2(V)$ 
generates the superalgebra $R_V = V/C_2(V)$.
\end{lem}

\begin{proof}
By \cref{dfn:NWLif,dfn:NKLif} of the Li filtration, $C_2(V)=E_1$ is linearly spanned by 
elements of the form $a^1_{(-p_1|P_1)} \dotsm a^1_{(-p_r|P_r)}b$ 
with $r \ge 1$, $a^i,b \in V$, $p_i \in \bbZ_{>0}$ and $P_i \subset [N]$ satisfying 
$p_1+\dotsb+p_r \ge r+1$. Hence, both $V$ and $E_1$ are linearly spanned by the elements 
of the form \eqref{eq:fsg}, and $R_V=V/E_1$ is generated as a superalgebra by $\ol{G}$.
\end{proof}

Next, we turn to the associated scheme of a SUSY VA module.
Let $M$ be a $V$-module. Then the $\gr_E V$-module structure on $\gr_E M$ 
(\cref{prp:NWgrE-mod,prp:NKgrE-mod}) induces a Poisson $R_V$-module structure 
on the quotient superspace
\[
 M/C_2(M) = E_0(M)/E_1(M).
\]
In other words, it has a superalgebra $R_V$-action
\begin{align}\label{eq:RV-mod:alg}
 \ol{a}.\ol{m} \ceq \ol{a_{(-1|N)}m} \quad (a \in V, \, m \in M)
\end{align}
with $\ol{a} \ceq a + C_2(V)$ and $\ol{m} \ceq m + C_2(M)$,
and has a Lie superalgebra $R_V$-action (of parity $N \bmod 2$)
\[
 \{\ol{a},\ol{m}\} \ceq \ol{a_{(0|N)}m}  \quad (a \in V, \, m \in M),
\]
which are compatible in the sense that
\begin{align}\label{eq:A12:27}
 \{\ol{a},\ol{b}.\ol{m}\} = \{\ol{a},\ol{b}\}.\ol{m} + 
 (-1)^{(p(a)+N)p(b)} \ol{b}.\{\ol{a},\ol{m}\}. 
\end{align}
(c.f.\ \cref{dfn:sP} of Poisson superalgebras.) 
The last formula \eqref{eq:A12:27} shows 
the following SUSY analogue of \cite[Lemma 3.2.1 (1)]{A12}:

\begin{lem}\label{lem:A12:3.2.1-1}
Let $M$ be a $V$-module, and $\Ann_{R_V}(M/C_2(M))$ be the annihilator of $M/C_2(M)$ 
in the superalgebra $R_V$ (see \eqref{eq:RV-mod:alg}). Then $\Ann_{R_V}(M/C_2(M))$ 
is a Poisson ideal of the Poisson superalgebra $R_V$ of parity $N \bmod 2$.
\end{lem}

Now, extending \cref{dfn:X_V}, we introduce:

\begin{dfn}
For a $V$-module $M$, we define the \emph{associated superscheme} of $M$ to be 
\[
 X_M \ceq \Supp_{R_V}\bigl(M/C_2(M)\bigr) = 
 \bigl\{\frp \in \Spec R_V \mid \bigl(M/C_2(M)\bigr)_{\frp} \neq 0\bigr\}.
\]
\end{dfn}

It is natural to study the condition when $M/C_2(M)$ is finitely generated
over the superalgebra $R_V$. Following \cite[\S 3.1]{A12}, we introduce:

\begin{dfn}\label{dfn:fsgM}
A $V$-module $M$ is \emph{finitely strongly generated over $V$}
if $M/C_2(M)$ is finitely generated over $R_V$ as a superalgebra.
\end{dfn}

If $M$ is finitely strongly generated over $V$, then 
\[
 X_M = \bigl\{\frp \in \Spec R_V \mid 
       \frp \supset \Ann_{R_V}\bigl(M/C_2(M)\bigr) \bigr\},
\]
and by \cref{lem:A12:3.2.1-1}, $X_M$ is a closed Poisson subscheme of $X_V$.
We have an immediate consequence:

\begin{lem}[{\cite[Lemma 3.2.2]{A12}}]
Assume $V$ is finitely strongly generated,
so that $R_V$ is a finitely generated superalgebra by \cref{lem:fsg}.
Then, for a finitely strongly generated $V$-module $M$, we have 
\[
 \text{$M$ is $C_2$-cofinite (\cref{dfn:C2cof})} \iff
 \dim X_M = 0.
\]
\end{lem}

\subsection{Singular support}\label{ss:ss}

In this subsection, we introduce the notion of singular supports for SUSY VAs
and study the relation to the lisse condition. 
The contents are more or less straightforward analogue of the even case \cite[\S 3.3]{A12}.
Throughout of this subsection, we take the base field $k$ to be $\bbC$.
Also, a SUSY VA means an $N_W=N$ or $N_K=N$ SUSY VA.

First, we need to restrict the class of SUSY VAs and modules.

\begin{dfn}
Let $V$ be a SUSY VA.
\begin{enumerate}[nosep]
\item
$V$ is \emph{graded} if it is equipped with 
an even semisimple operator $H$, called a \emph{Hamiltonian}, such that 
\[
 [H,a_{(j|J)}] = -(j+1)a_{(j|J)}+(H a)_{(j|J)}
\]
for any $a \in V$, $j \in \bbZ$ and $J \subset [N]$.
We denote the eigenspace decomposition by $V=\bigoplus_{\Delta \in \bbC} V_{\Delta}$, 
$V_\Delta \ceq \{a \in V \mid H a = \Delta a\}$.
An element of $V_\Delta$ for some $\Delta$ will be called 
\emph{homogeneous of weight $\Delta$}.
For $a \in V_{\Delta_a}$ and $b \in V_{\Delta_b}$,
we have $a_{(j|J)}b \in V_{\Delta_a+\Delta_b-j-1}$.
 
\item
For a subset $I \subset \bbC$, $V$ is \emph{$I$-graded} 
if it is graded and $V_\Delta=0$ for $\Delta \notin I$.

\item
Assume $V$ is graded with Hamiltonian $H$. Then a $V$-module $M$ is \emph{graded} if 
\[
 [H,a^M_{(j|J)}] = -(j+1)a^M_{(j|J)}+(H a)^M_{(j|J)}
\]
for any $a \in V$, $j \in \bbZ$ and $J \subset [N]$,
and moreover if there is a decomposition
\begin{align}\label{eq:grM}
 M = \bigoplus_{d \in \bbC} M_d, \quad M_d \ceq \{m \in M \mid H m = d m\}.
\end{align}
An element of $M_d$ for some $d$ is called \emph{homogeneous of weight $d$}.
For $a \in V_{\Delta_a}$ and $m \in M_d$, we have $a^M_{(j|J)}m \in M_{d+\Delta_a-j-1}$.
\end{enumerate}
\end{dfn}

\begin{rmk}\label{rmk:gr-sc}
An $N_W=N$ conformal SUSY VA (\cref{dfn:cNW}) 
is graded with Hamiltonian $\omega_{(1|0)}$.
An $N_K=N$ conformal SUSY VA (\cref{dfn:cNK})
is graded with Hamiltonian $\tau_{(1|0)}$.
\end{rmk}

We consider the following class of SUSY VAs and their modules.
The conditions are simple analogue of those in \cite{A12}.

\begin{asm}\label{asm:A12}
Let $V$ be a SUSY VA which is 
\begin{itemize}[nosep]
\item $\frac{1}{r_0}\bbN$-graded with some $r_0 \in \bbZ_{>0}$, and 
\item finitely strongly generated (\cref{dfn:fsg}).
\end{itemize}
Also, let $M$ be a $V$-module which is
\begin{itemize}[nosep]
\item graded and lower truncated, i.e., 
there exists a finite subset $\{d_1,\dotsc,d_s\} \subset \bbC$ such that 
$M_d = 0$ unless $d \in d_i + \frac{1}{r_0}\bbN$ in the decomposition \eqref{eq:grM}, and 
\item finitely strongly generated (\cref{dfn:fsgM}).
\end{itemize}
\end{asm}

\begin{rmk}
Strongly conformal SUSY VAs (\cref{dfn:scNW,dfn:scNK}) satisfy \cref{asm:A12}.
\end{rmk}

By \cref{lem:fsg}, the $C_2$-superalgebra $R_V$ is finitely generated.
We also have the following lemma, 
the even version of which is stated in \cite[\S3.1, p.570]{A12}.

\begin{lem}
Let $V$ be a SUSY VA and $M$ be a $V$-module, both satisfying \cref{asm:A12}.
Then the Li filtration $\{E_n(M) \mid n \in \bbZ\}$ is separated, i.e., 
$\bigcap_{n \in \bbZ} E_n(M) =0$.
\end{lem}

\begin{proof}
It is essentially shown in \cite[Lemma 2.14]{Li}, 
but let us write down a proof for completeness.
Let $E_n(M) = \bigoplus_d E_n(M)_d$, $E_n(M)_d \ceq E_n(M) \cap M_d$ 
be the induced decomposition. It is enough to show 
$E_n(M)_d=0$ unless $d \ge n+d_i$ for some $i \in \{1,\dotsc,s\}$.

The case $n \le 0$ holds by the condition on $M$.
For $n \ge 1$, recall \cref{dfn:NKLif}, 
and consider a basis element of $E_n(M)$: 
\[
 v \ceq a^1_{(-1-k_1|K_1)} \dotsm a^r_{(-1-k_r|K_r)} m
\]
with $r \ge 1$, $a^j \in V$, $m \in M$, $k_j \in \bbZ_{>0}$ and $K_j \subset [N]$
satisfying $k_1+\dotsb+k_r \ge n$. 
We may assume that $a^j$ and $m$ are homogeneous, say $a^j \in V_{\Delta_j}$ and $m \in M_d$.
By the assumption on $M$, we have $d \ge d_i$ for some $i \in \{1,\dotsc,s\}$.  
Then $v$ is also homogeneous of weight
\[
 (\Delta_1+k_1)+\dotsb+(\Delta_1+k_1)+d \ge k_1+\dots+k_r+d \ge n+d \ge n+d_i.
\]
Thus we have $E_n(M)_d=0$ unless $d \ge n+d_i$.
\end{proof}

The following lemma corresponds to a half of \cite[Lemma 3.1.5]{A12}.

\begin{lem}
Let $V$ be a SUSY VA and $M$ be a $V$-module, both satisfying \cref{asm:A12}.
Then $\gr_E M$ is a finitely generated module over the superalgebra $\gr_E V$.
\end{lem}

\begin{proof}
By \cref{lem:Li:4.2}, the $\gr_E V$-module $\gr_E M$ is generated by $E_0(M)/E_1(M)=M/C_2(M)$.
By \cref{asm:A12}, $M/C_2(M)$ is finitely generated over 
the superalgebra $R_V \subset \gr_E V$. So we have the consequence.
\end{proof}

Let us recall \cref{prp:A12:2.5.1}, 
by which we have a surjection of commutative superalgebras
\begin{align}\label{eq:Phi}
 \Phi\colon (R_V)^O \lsrj \gr_E V
\end{align}
from the superjet algebra. On the other hand, For a $V$-module $M$,
$\gr_E M$ is an $N_K=N$ SUSY VPA module of $\gr_E V$ by \cref{prp:NKgrE-mod},
Hence, $\gr_E M$ can also be regarded as a module over the commutative superalgebra $(R_V)^O$.
Under \cref{asm:A12}, $\gr_E M$ is finitely generated over $\gr_E V$,
so that it is also finitely generated over $(R_V)^O$

Now, following \cite[\S 3.3]{A12}, we introduce:

\begin{dfn}\label{dfn:SS}
Let $V$ be a SUSY VA $V$ and a $V$-module $M$.
We define the \emph{singular support of $M$} to be
\[
 \SSup(M) \ceq \Supp_{(R_V)^O}(\gr_E M).
\]
Here we regard $(R_V)^O$ as a commutative superalgebra.
Under \cref{asm:A12}, it can be written as 
\begin{align*}
 \SSup(M) = \left\{\frp \in \Spec (R_V)^O \mid 
                   \frp \supset \Ann_{(R_V)^O}(\gr_E M)\right\}.
\end{align*}
\end{dfn}

For a superscheme $X=(\ul{X},\shO)$, its \emph{reduced part} \cite[(1.1.5)]{KV2} 
is a scheme defined by 
\[
 X_{\trd} \ceq \bigl(\ul{X},\shO_{\oz}/\sqrt{\shO_{\oz}} \bigr).
\]
Mimicking the even case \cite[\S 3.3]{A12}, let us also introduce:

\begin{dfn}\label{dfn:lisse}
Let $V$ be a SUSY VA. 
\begin{enumerate}[nosep]
\item 
A $V$-module $M$ is called \emph{lisse} if it is finitely strongly generated and 
$\SSup(M)_{\trd}$ is $0$-dimensional.
\item
$V$ itself is called \emph{lisse} if it is lisse as a $V$-module.
\end{enumerate}
\end{dfn}

We have the following analogue of \cite[Theorem 3.3.3]{A12}.

\begin{thm}\label{thm:A12:3.3.3}
Let $V$ be a SUSY VA satisfying \cref{asm:A12}. We have
\[
 \text{$V$ is $C_2$-cofinite (\cref{dfn:C2cof})} \iff \text{$V$ is lisse}.
\] 
\end{thm}

For the proof, we need some preliminary. 
First, recall the projection $(R_V)^{O_m} \to (R_V)^O$ of superjet algebras 
in \eqref{eq:fsj-sj}. We rewrite it as a morphism of superschemes
\begin{align}\label{eq:pm}
 p_m\colon (X_V)^O \ceq \Spec (R_V)^O \lto (X_V)^{O_m} \ceq \Spec (R_V)^{O_m}.
\end{align}
Then we have the following analogue of \cite[Lemma 3.3.1]{A12}.

\begin{lem}\label{lem:A12:3.3.1}
Let $V$ be a SUSY VA and $M$ be a $V$-module, both satisfying \cref{asm:A12}.
Then we have $X_M = p_0(\SSup(M))$ (scheme-theoretic image).
\end{lem}

\begin{proof}
By \cref{lem:Li:4.2}, the $\gr_E V$-module $\gr_E M$ is generated by 
$E_0(M)/E_1(M)=M/C_2(M)$, which implies 
\begin{align}\label{eq:A12:30}
 R_V \cap \Ann_{\gr_E V}(\gr_E M) = \Ann_{R_V}\bigl(M/C_2(M)\bigr).
\end{align}
It yields the consequence.
\end{proof}

Now we start:

\begin{proof}[Proof of \cref{thm:A12:3.3.3}]
For the direction $\Longrightarrow$, $C_2$-cofiniteness of $V$ implies that 
the reduced part of $(X_V)^O=\Spec(R_V)^O$ is $0$-dimensional.
Since $\SSup(V) \subset (X_V)^O$ by \cref{dfn:SS}, 
we find that $\SSup(V)_{\trd}$ is $0$-dimensional. 
For the converse $\Longleftarrow$, using \eqref{eq:pm},
we find that $p_0(\SSup(V)_{\trd})$ is $0$-dimensional.
Then \cref{lem:A12:3.3.1} yields that $(X_V)_{\trd}$ is $0$-dimensional.
\end{proof}

We have a similar statement for $V$-modules as \cref{thm:A12:3.3.3}
if we assume $V$ to be strongly conformal (\cref{dfn:scNW,dfn:scNK}).
Recall that such a strongly conformal $V$ is graded 
with Hamiltonian $H=\nu_{(1|0)}$ or $\tau_{(1|0)}$ (\cref{rmk:gr-sc}).

\begin{thm}\label{thm:A12:3.3.4}
Let $V$ be a strongly conformal SUSY VA whose grading for the Hamiltonian $H$ 
satisfies \cref{asm:A12}. Also, let $M$ be a $V$-module satisfying \cref{asm:A12}.
Then, we have  
\[
 \text{$M$ is $C_2$-cofinite} \iff \text{$M$ is lisse}.
\]
\end{thm}

For the proof, we need a few lemmas.

\begin{lem}[{c.f.\ \cite[Lemma 3.2.1 (ii)]{A12}}]\label{lem:A12:3.2.1}
Let $V$ and $M$ be as in \cref{thm:A12:3.3.4}. 
\begin{enumerate}[nosep]
\item
In the $N_W=N$ case, the annihilator $\Ann_{(R_V)^O}(\gr_E M)$ is a vertex Poisson ideal 
of the level $0$ $N_W=N$ SUSY VPA $\bigl((R_V)^O,T,S^i)$ 
(\cref{prp:NWlv0,prp:A12:2.5.1}).

\item 
In the $N_K=N$ case, the annihilator $\Ann_{(R_V)^O}(\gr_E M)$ is a vertex Poisson ideal 
of the level $0$ $N_K=N$ SUSY VPA $(R_V)^{\Osc}=\bigl((R_V)^O,S_K^i)$ 
(\cref{prp:NKlv0,prp:A12:2.5.1}).
\end{enumerate}
\end{lem}

\begin{proof}
We give an argument only for the $N_K=N$ case, since the $N_W=N$ case is quite similar.
Recall the surjection $\Phi\colon (R_V)^{\Osc} \srj \gr_E V$ of $N_K=N$ SUSY VPAs 
(\cref{prp:A12:2.5.1}) and the argument around \eqref{eq:Phi}.
In particular, $\gr_E M$ is a module over the commutative superalgebra $(R_V)^O$.
We also have 
\[
 I \ceq \Ann_{(R_V)^O}(\gr_E M) = \Phi^{-1}\bigl(\Ann_{\gr_E V}(\gr_E M)\bigr),
\]
and it is an ideal of $(R_V)^O$. 
Since $\Phi$ is a VLA morphism, it is also a VLA ideal.
Thus, it is enough to show that $I$ is a differential ideal
over the differential algebra $(R_V)^{\Osc}=\bigl((R_V)^O,S_K^i\bigr)$.
Now, recall that $\gr_E V$ is a differential algebra 
with derivations $S_K^i$ (\cref{dfn:dsa}).
We have that $\gr_E M$ is a differential module over $\gr_E V$.
Indeed, using the conformal element $\tau \in V$, we have the action of $S_K^i$ by
$S_K^i.\ol{m} = \ol{\tau_{(0|e_i)}m}$ for $i \in [N]$.
Hence, $I' \ceq \Ann_{\gr_E V}(\gr_E M)$ is a differential ideal,
and so is $I=\Phi^{-1}(I')$.
\end{proof}

\begin{lem}[{c.f.\ \cite[Lemma 3.3.1 (ii)]{A12}}]\label{lem:A12:3.3.1-2}
Let $V$ and $M$ be as in \cref{thm:A12:3.3.4}. 
Then, we have $\SSup(M) \subset (X_M)^O$.
\end{lem}

\begin{proof}
Again, we only give a proof for $N_K=N$ case.
By \cref{lem:A12:3.2.1}, $\Ann_{\gr_E V}(\gr_E M)$ is a differential ideal
of the differential algebra $\gr_E V$ with derivations $S_K^i$.
Then \eqref{eq:A12:30} implies that $\Ann_{\gr_E V}(\gr_E M)$ contains the ideal 
$I \subset \bbC[(X_V)^O]=R_V^O$ which is minimal among the $S_K^i$-stable ideals
containing $\Ann_{R_V}\bigl(M/C_2(M)\bigr)$. 
Since $I$ is the defining ideal of $(X_M)^O$, we have the conclusion.
\end{proof}

\begin{proof}[Proof of \cref{thm:A12:3.3.4}]
The direction $\Longleftarrow$ can be shown in the same way as \cref{thm:A12:3.3.3} 
using \cref{lem:A12:3.3.1}. For the direction $\Longrightarrow$, the reduced part 
of $(X_M)^O$ is $0$-dimensional by the assumption. Then, from \cref{lem:A12:3.3.1}, 
we find that $\SSup(M)_{\trd}$ is $0$-dimensional.
\end{proof}

\subsubsection{Example: Neveu-Schwarz SUSY vertex algebra}

We continue \cref{eg:NS2,eg:NS-RV}. 
Let us denote by $\frg$ the Neveu-Schwarz Lie superalgebra 
with central extension $c \in \bbC$. We use the basis 
$\{L_n \mid n \in \bbZ\} \cup \{G_r \mid r \in \hf+\bbZ\}$ satisfying \eqref{eq:NS}.
Let us also denote by $V^c$ the Neveu-Schwarz SUSY VA with central charge $c \in \bbC$.
Recall that $V^c$ is an $N_K=1$ SUSY VA.

As a $\frg$-module, we have $V^c=M_{0,c}/M_{\hf,c}$ with $M_{h,c}$ 
the Verma module of $\frg$ generated by the highest weight vector $\ket{h,c}$ 
satisfying $L_0.\ket{h,c}=h$ and $1.\ket{h,c}=\ket{h,c}$.
$V^c$ has a unique maximal sub-$\frg$-module $N_c$.
Now let us cite:

\begin{fct}[{\cite[5.2.1. Theorem (ii),(iv)]{GK}}]\label{fct:GK}
Let $Y \ceq \{(p,q) \in \bbZ_{\ge 1}^2 \mid p-q \in 2\bbZ, (\tfrac{p-q}{2},q)=1\}$,
and for $(p,q) \in Y$ we set $c_{p,q}^S \ceq \frac{3}{2}\bigl(1-\frac{2(p-q)^2}{p q}\bigr)$.
Then the following conditions are equivalent.
\begin{itemize}[nosep]
\item
$V^c$ is not simple as a $\frg$-module.
\item
$c \neq c^S_{p,q}$ for $(p,q) \in Y$ and $p>q \ge 2$.
\item
$\dim_\bbC V^c/C_2^{\GK}$ is finite,
where $C_2^{\GK} \subset V^c$ is the linear sub-superspace given by 
\[
 C_2^{\GK} \ceq \spn \left\{L_{-n}v, G_{-r}v \mid n,r-\thf > 2, \, v \in V^c\right\}.
\]
\end{itemize}
\end{fct}

We can restate this fact as follows, which is a Neveu-Schwarz analogue of 
the statement \cite[Proposition 3.4.1]{A12} for the universal Virasoro vertex algebra.

\begin{prp}
The following conditions are equivalent.
\begin{itemize}[nosep]
\item
$V^c$ is $C_2$-cofinite as a SUSY VA module over itself (\cref{dfn:C2cof}).
\item
$V^c$ is not simple as a $\frg$-module.
\item
$c \neq c^S_{p,q}$ for $(p,q) \in Y$ and $p>q \ge 2$.
\end{itemize}
\end{prp}

\begin{proof}
It is enough to show $C_2^{\GK}=C_2(V^c)$.
Recall the Neveu-Schwarz element $\tau$ and $\nu \ceq \frac{1}{2}S_K \tau$ of $V^c$.
By $V^c=\bbC[S_K^n \tau \mid n \in \bbZ]\vac$ and \eqref{eq:NK-TSa}, we have 
$C_2(V^c) = \spn\{\tau_{(-n|J)} v \mid n \in \bbZ_{\ge 2}, \, J \in \{0,1\}, \, v \in V^c\}$.
On the other hand, by \eqref{eq:tau_op} and \eqref{eq:NK-TSa}, 
the Fourier modes of $\tau$ satisfy
$G_r = \tau_{(r+\hf|1)}$ and 
$L_n = \hf \tau_{(n+1|0)}$. 
Thus, we have 
\[
 C_2^{\GK} = \spn\{\tau_{(-n|0)}v, \tau_{(-n|1)}v \mid 
                   n \in \bbZ_{\ge 2}, \, v \in V^c\} = C_2(V^c).
\]
\end{proof}

We also have a Neveu-Schwarz analogue of the statement \cite[Proposition 3.4.2]{A12} 
on the relation between $C_2$-cofiniteness of Virasoro modules and the zero singular 
support condition in the sense of Beilinson, Feigin and Mazur \cite{BFM}.

Let us denote by $\{U_p(\frg) \mid p \in \bbN\}$ the PBW filtration 
of the universal enveloping algebra $U(\frg)$.
The associated graded $\gr U(\frg)$ is a commutative superalgebra 
isomorphic to the symmetric algebra $S(\frg)$.
Let $M$ be a highest weight representation of $\frg$ with central charge $c$
and the highest weight vector $v$.
Then $\{U_p(\frg).v \mid \bbN\}$ gives a filtered module structure on $M$
over the PBW-filtered algebra $U(\frg)$.
The associated graded space, denoted by $\gr_{\PBW} M$,
is an $S(\frg)$-module generated by the image of $v$.
Now, following \cite[\S 7.1.1]{BFM} and borrowing the symbol in \cite[\S 3.4]{A12},
we define
\[
 \SBFM{M} \ceq \Supp_{S(\frg)}\bigl(\gr_{\PBW}(M)\bigr).
\]

On a highest weight representation $M$ over $\frg$,
we have a standard module structure over the Neveu-Schwarz vertex superalgebra,
which is equivalent to a module structure over the $N_K=1$ SUSY VA $V^c$.
Thus we can discuss the $C_2$-cofiniteness of $M$. 

\begin{prp}
For a highest weight representation $M$ over 
the Neveu-Schwarz Lie superalgebra $\frg$ of central charge $c$,
the following conditions are equivalent.
\begin{itemize}[nosep]
\item 
$M$ is $C_2$-cofinite as a $V^c$-module in the sense of \cref{dfn:C2cof}.
\item
$\SBFM{M} =\{0\}$.
\end{itemize}
\end{prp}

\begin{proof}
The argument for the Virasoro algebra in \cite[Proposition 3.4.2]{A12} works 
by replacing $L_{-1},L_{-2}$ therein by $G_{-\hf},G_{-\frac{3}{2}}$.
We omit the detail.
\end{proof}

In \cite[7.3.10 Proposition]{BFM}, it is shown that the condition $\SBFM{M} =\{0\}$ 
for a highest weight module $M$ over the Virasoro Lie algebra of central charge $c$
is equivalent to the condition that $M$ belongs to the minimal series representations
(or the Belavin-Polyakov-Zamolodchikov representations).
This equivalence and \cite[Proposition 3.4.2]{A12} yields 
that a $C_2$-cofinite module over the universal Virasoro vertex algebra is nothing but
a minimal series representation of Virasoro Lie algebra \cite[Theorem 3.4.3]{A12}.

The Neveu-Schwarz Lie superalgebra $\frg$ also has 
the family of minimal series representations $M_{c,h}$ with $(c,h)=(c^S_{p,q},h_{p,q}^{r,s})$, 
where $c^S_{p,q}$ is given in \cref{fct:GK}, and $h_{p,q}^{r,s}=\frac{(sp-rq)^2-(p-q)^2}{8pq}$
for $(r,s) \in \bbZ^2$ satisfying $0<r<p$, $0<s<q$ and $r-s \in 2\bbZ$.
See \cite[\S 3]{KW} and \cite[\S 1]{Ad} for the detail.
At this moment, the author is not sure whether the argument in \cite{BFM} works 
for Neveu-Schwarz minimal representations, but believes that it should do.
Let us state it as:

\begin{cnj}
Let $M$ be a simple module over 
the Neveu-Schwarz Lie superalgebra $\frg$ with central charge $c$.
Then the following conditions are equivalent.
\begin{itemize}[nosep]
\item 
$M$ is $C_2$-cofinite as a module over the Neveu-Schwarz SUSY vertex algebra $V^c$.
\item
$M$ is a minimal series representation of $\frg$.
\end{itemize}
\end{cnj}

\section{Concluding comments}\label{s:con}

As mentioned in \cref{s:intro}, this note is written as a first step toward 
a semi-infinite Poisson-geometric study of SUSY vertex algebras.
There are already several geometric studies of SUSY vertex algebras,
and our next step would be to relate this note and those studies.
Here is a list of possible directions.
\begin{enumerate}
\item \label{i:con:1} 
The $C_2$-cofiniteness guarantees the modular property of conformal blocks 
in the even case (see, e.g., \cite{Z,Mi}), as mentioned in \cref{ss:C2}.
Mathematical study of conformal blocks of a SUSY vertex algebra 
seems to be still in infancy. 
One remarkable study in this direction is, as mentioned in \cref{s:intro},
the work of Heluani and Van Ekeren \cite{HV}, which shows that 
the normalized character of a positively graded module over a charge cofinite $N=2$ 
topological vertex algebra can be regarded as a solution of a flat connection 
over the moduli space of elliptic supercurves, and actually are Jacobi forms.  
An $N=2$ topological vertex algebra has a structure of 
a strongly conformal $N_W=1$ SUSY vertex algebra by the result of Heluani and Kac \cite{HK},
and by the work of Heluani \cite{H1}, we have the vector bundle with flat connection 
of conformal blocks on the moduli space.
It is shown in \cite[A.2.\ Lemma]{HV} that the charge cofiniteness is equivalent to 
the $C_2$-cofiniteness for the underlying vertex superalgebra.
We can show that it is also equivalent to the $C_2$-cofiniteness of the associated $N_W=1$ 
SUSY vertex algebra in the sense of \cref{dfn:C2cof}.
Thus, the result of \cite{HV} implies that modules over a $C_2$-cofinite $N_W=1$ 
SUSY vertex algebra satisfying \cref{asm:A12} would enjoy super-modularity.

\item \label{i:con:2}
We also expect an analogous result of the item \eqref{i:con:1} for the $N_W=N>1$ case.
The argument of \cite{HV} is based on a SUSY analogue \cite[Theorem 3.4]{H1} of Huang's 
coordinate change formula, applied to a natural $N_W=1$ SUSY analogue 
$(z,\zeta) \mapsto (e^{2\pi i z}-1,e^{2\pi i z}\zeta)$ of Zhu's coordinate change 
$z \mapsto e^{2\pi i z}-1$, and based on $N=1$ super analogue of Weierstrass elliptic functions.
Both have straightforward $N>1$ analogue, but the analysis gets much complicated,
and at this moment the author has no clear picture.

\item
Along the line of \eqref{i:con:1} and \eqref{i:con:2}, 
the $N_K=N$ case is mysterious for the author.
The $N_K=N$ SUSY analogue of Huang's formula is given in \cite[Theorem 3.6]{H1}.
Mimicking the argument in \cite{HK}, the author tried to find a clean coordinate change 
which preserves the superconformal structure (also called SUSY structure or 
super Riemann surface structure), but has failed until this moment.

\item
Also related to super-modularity, let us recall the notion of 
\emph{quasi-lisse vertex algebras} introduced in the work of Arakawa and Kawasetsu \cite{AK}.
It is a finitely strongly generated vertex algebra whose associated variety
(the reduced part of the associated scheme in our terminology) 
has only finitely many symplectic leaves.
In \cite[Theorem 5.1]{AK}, it is shown that the normalized character of a module 
over a quasi-lisse vertex operator algebra enjoys modularity.
An $N_W=1$ SUSY analogue of this result might be obtained 
by borrowing some arguments in \cite{HV}, but one must take care to establish
a super (non-reduced) analogue of the results of Poisson geometry used in \cite{AK}.
The other $N_W=N>1$ and $N_K=N$ cases are much more unclear at this moment.
\end{enumerate}



\end{document}